\pdfoutput=1

\documentclass[12pt]{article} %

\usepackage{etoolbox}
\newtoggle{preprint}
\newcommand{\preprint}[1]{\iftoggle{preprint}{#1}{}}
\newcommand{\infor}[1]{\iftoggle{preprint}{}{#1}}

\toggletrue{preprint} %

\usepackage{amsmath,amssymb,bbm}
\usepackage{natbib}
\usepackage[hyperfootnotes=false, hidelinks]{hyperref}
\usepackage[capitalize,nameinlink]{cleveref}

\usepackage{breqn}
\usepackage{changepage}
\usepackage{dsfont}
\usepackage{multirow}
\usepackage{makecell}
\usepackage{booktabs}
\usepackage{array}
\usepackage{mathtools}
\usepackage{esvect}
\usepackage{url}
\usepackage{bm}
\usepackage{wrapfig}
\usepackage{mathrsfs}
\usepackage{subfiles}
\usepackage[caption = false]{subfig} 
\usepackage{graphicx}
\usepackage{float}
\usepackage{tabularx}
\usepackage[symbol]{footmisc}
\usepackage{accents}
\usepackage{enumerate}
\usepackage{comment}
\usepackage{appendix}

\usepackage[ruled,vlined]{algorithm2e}
\def\[#1\]{\begin{equation}\begin{aligned}#1\end{aligned}\end{equation}}
\def\*[#1\]{\begin{equation*}\begin{aligned}#1\end{aligned}\end{equation*}}

\usepackage{crossreftools}
\hypersetup{%
    bookmarksnumbered, bookmarksopen=true, bookmarksopenlevel=1,colorlinks=false%
}

\crefname{subsection}{Section}{Sections}
\crefname{lemma}{Lemma}{Lemmas}
\crefname{corollary}{Corollary}{Corollaries}
\crefname{theorem}{Theorem}{Theorems}
\crefname{definition}{Definition}{Definitions}
\crefname{assumption}{Assumption}{Assumptions}
\crefname{supplement}{Supplement}{Supplements}
\crefname{algorithm}{Algorithm}{Algorithms}

\newcommand{\N}{\mathbb{N}}

\newcommand{\Z}{\mathbb{Z}}
\newcommand{\R}{\mathbb{R}}

\newcommand{\1}{\mathds{1}}

\newcommand{\dee}{\mathrm{d}}

\DeclareMathOperator{\EE}{\mathbb{E}}
\DeclareMathOperator{\PP}{\mathbb{P}}

\DeclareMathOperator{\argmin}{\arg\min}

\preprint{%

\usepackage{authblk}
\usepackage[letterpaper, margin=1in]{geometry}

\setlength{\parindent}{2em}
\setlength{\parskip}{0pt}

\usepackage{fancyhdr}
\fancyhf{}
\pagestyle{fancy}

\fancyheadoffset{0pt}
\cfoot{\thepage}

\makeatletter
\renewcommand{\maketag@@@}[1]{\hbox{\m@th\normalsize\normalfont#1}}%
\makeatother

\makeatletter
\let\reftagform@=\tagform@
\def\tagform@#1{\maketag@@@{\ignorespaces\textcolor{gray}{(#1)}\unskip\@@italiccorr}}
\renewcommand{\eqref}[1]{\textup{\reftagform@{\ref{#1}}}}
\makeatother

\usepackage{amsthm}

\usepackage{thmtools}

\declaretheorem[name=Theorem]{theorem}

\declaretheorem[name=Lemma]{lemma}
\declaretheorem[name=Proposition]{proposition}
\declaretheorem[name=Corollary]{corollary}

\declaretheorem[name=Definition]{definition}

\declaretheorem[name=Assumption, numbered=no]{assumption*}
\declaretheorem[qed=$\triangleleft$,name=Example]{example}
\declaretheorem[qed=$\triangleleft$,name=Remark]{remark}

\bibpunct[, ]{(}{)}{;}{a}{}{,}
\usepackage{tikz}

\usepackage{titlesec}
	\titlelabel{\thetitle.\quad}
	\titleformat*{\section}{\large\bfseries}
	\titleformat*{\subsection}{\bfseries}
	\titleformat*{\subsubsection}{\itshape}

\titlespacing*\section{0pt}{12pt plus 4pt minus 2pt}{4pt plus 2pt minus 2pt}
\titlespacing*\subsection{0pt}{12pt plus 4pt minus 2pt}{4pt plus 2pt minus 2pt}
\titlespacing*\subsubsection{0pt}{12pt plus 4pt minus 2pt}{4pt plus 2pt minus 2pt}
}
\infor{
\usepackage{epstopdf}%
\usepackage[caption=false]{subfig}%

\bibpunct[, ]{(}{)}{;}{a}{}{,}%
\renewcommand\bibfont{\fontsize{10}{12}\selectfont}%

\theoremstyle{plain}%
\newtheorem{theorem}{Theorem}[section]
\newtheorem{lemma}[theorem]{Lemma}

\theoremstyle{definition}

\theoremstyle{remark}
\newtheorem{remark}{Remark}

}
\usepackage{tikz}
\DeclareRobustCommand\full  {\tikz[baseline=-0.6ex]\draw[thick] (0,0)--(0.5,0);}
\DeclareRobustCommand\dotted{\tikz[baseline=-0.6ex]\draw[thick,dotted] (0,0)--(0.54,0);}
\DeclareRobustCommand\dashed{\tikz[baseline=-0.6ex]\draw[thick,dashed] (0,0)--(0.54,0);}

\newcommand{\apqname}{APQ}
\newcommand{\npqname}{NPQ}
\newcommand{\fcfsname}{FCFS}
\newcommand{\dapqname}{delayed APQ}

\newcommand{\highname}{high-priority}

\newcommand{\lowname}{low-priority}

\newcommand{\classname}{class}

\newcommand{\zexp}{\text{Z-Exp}}

\newcommand{\serv}{\mathcal{S}} %
\newcommand{\resid}[1]{ %
\ifx&#1&
  \mathcal{R}
\else
  \mathcal{R}_{#1}
\fi} 
\newcommand{\lam}{\lambda} %
\newcommand{\lamA}{\lam^{\mathrm{A}}} %
\newcommand{\rhoA}{\rho^{\mathrm{A}}} %

\newcommand{\dapq}{\text{\tiny\upshape DAPQ}} %
\newcommand{\npq}{\text{\tiny\upshape NPQ}} %
\newcommand{\fcfs}{\text{\tiny\upshape FCFS}} %

\newcommand{\wait}{\mathcal{W}}
\newcommand{\waitD}{\wait^{\dapq}} %
\newcommand{\waitDparams}[2]{\wait^{\dapq(#1,#2)}} %
\newcommand{\waitN}{\wait^{\npq}} %
\newcommand{\waitF}{\wait^{\fcfs}} %

\newcommand{\cdf}{F}
\newcommand{\cdfD}{\cdf^{\dapq}} %
\newcommand{\cdfDparams}[2]{\cdf^{\dapq(#1,#2)}} %

\newcommand{\den}{f}

\newcommand{\cdfres}[1]{\cdf^{\scriptstyle \resid{#1}}} %
\newcommand{\cdfresc}[1]{\cdf^{\scriptscriptstyle \resid{} \mid {#1}}} %

\newcommand{\denres}[1]{\den^{\scriptstyle \resid{#1}}} %
\newcommand{\denresc}[1]{\den^{\scriptscriptstyle \resid{} \mid {#1}}} %

\newcommand{\lst}{\tilde F}
\newcommand{\lstD}{{\lst}^{\dapq}} %
\newcommand{\lstN}{{\lst}^{\npq}} %

\newcommand{\lstres}[1]{\lst^{\scriptstyle \resid{#1}}} %
\newcommand{\lstresc}[1]{\lst^{\scriptscriptstyle \resid{} \mid {#1}}} %

\newcommand{\num}{N}

\newcommand{\sta}{\pi}

\newcommand{\acc}{\eta}

\newcommand{\accres}[1]{\acc_{\scriptscriptstyle \resid{#1}}} %
\newcommand{\accresc}[1]{\acc_{\scriptscriptstyle \resid{} \mid {#1}}} %

\newcommand{\accs}{\tilde\eta}
\newcommand{\accDs}{{\accs}^{\dapq}} %
\newcommand{\accNs}{{\accs}^{\npq}} %

\newcommand{\accsres}[1]{\accs_{\scriptscriptstyle \resid{#1}}} %
\newcommand{\accsresc}[1]{\accs_{\scriptscriptstyle \resid{} \mid {#1}}} %

\newcommand{\accDsresc}[1]{\accDs_{\scriptscriptstyle \resid{} \mid {#1}}}

\newcommand{\accNsresc}[1]{\accNs_{\scriptscriptstyle \resid{} \mid {#1}}}

\newcommand{\numresid}[2]{p_{#2}(#1)} 
\newcommand{\numarriv}[2]{q_{#2}(#1)}
\newcommand{\probempty}[4]{
\ifx&#1&
  {B_{#2,#3}(#4)}
\else
  {B_{#2,#3}(#1, #4)}
\fi} 

\usepackage[hide=false,setmargin=true,marginparwidth=0.75in]{marginalia}

\preprint{
\title{High-Priority Expected Waiting Times in the\\ Delayed Accumulating Priority Queue\\ with Applications to Health Care KPIs}

\date{} 					%

\author[,1]{Blair Bilodeau\footnote{Correspondence to: \url{blair.bilodeau[at]mail.utoronto.ca}}}
\author[2]{David A.~Stanford}
\affil[1]{University of Toronto}
\affil[2]{Western University}
}

\begin{document}	

\infor{
\title{High-Priority Expected Waiting Times in the Delayed Accumulating Priority Queue with Applications to Health Care KPIs}

\author{
\name{Blair Bilodeau\textsuperscript{a}\thanks{CONTACT Blair Bilodeau. Email: blair.bilodeau@mail.utoronto.ca} and David A.~Stanford\textsuperscript{b}}
\affil{\textsuperscript{a}Department of Statistical Sciences, University of Toronto;\\ \textsuperscript{b}Department of Statistical and Actuarial Sciences, Western University}
}
}

\maketitle

\preprint{\vspace{-40pt}}

\begin{abstract}
We provide the first analytical expressions for the expected waiting time of \highname{} customers in the \dapqname{} by exploiting a classical conservation law for work-conserving queues. Additionally, we describe 
an algorithm to compute
the expected waiting times of both \lowname{} and \highname{} customers, 
which requires only the truncation of sums that converge quickly in our experiments. These insights are used to demonstrate how the accumulation rate and delay level should be chosen by health care practitioners to optimize common key performance indicators (KPIs). In particular, we demonstrate that for certain nontrivial KPIs, an accumulating priority queue with a delay of zero is always preferable.
Finally, we present a detailed investigation of the quality of an exponential approximation to the \highname{} waiting time distribution, which we use to optimize the choice of queueing parameters with respect to both classes' waiting time distributions.
\end{abstract}

\infor{
\begin{keywords}
Priority queues; Health care KPIs; Exponential approximation 
\end{keywords}
}

\section{Introduction} \label{sec:intro}

\emph{Accumulating priority queues} (\apqname{}s) are a class of queueing disciplines in which waiting customers accrue credit over time at class-dependent rates. By convention, the highest priority customers belong to \classname{}-1, and they accumulate credit at the highest rate. At service completion instants, the customer present with the greatest amount of accumulated credit is the next one selected for service. These are especially useful in highly congested systems when even a moderate proportion of the arrival load is due to \highname{} customers, since with high probability there will be least one \highname{} customer in system and thus \lowname{} customers can have extremely long wait times under a strict priority regime.

\apqname{}s are well-understood theoretically, beginning with the derivation of the M/G/1 waiting times for all classes in \citet{stanford14apq}. 
Further extensions include preemptive service \citep{fajardo15preemptive}, nonlinear priority accumulation \citep{li17nonlinear}, hard upper limits on waiting time \citep{cildoz19apqh}, and applications to COVID-19 policies \citep{oz20covid}. 
Most relevant to the present work is \citet{mojalal19dapq}, in which for the first time, not all classes accumulate priority credit starting from their arrival instant. Instead, while the higher of two classes starts to accrue credit immediately, customers from the lower of the two classes only do so after a fixed period of initial delay. For this reason, \citeauthor{mojalal19dapq} named this the \emph{delayed accumulating priority queue}.

Both the original \apqname{} models and the \dapqname{} variant were developed in response to a perceived need stemming from the health care setting. Many health care systems are measured against a set of \emph{key performance indicators} (KPIs), and it is quite common for these KPIs to comprise a delay target, representing a time by which treatment should commence, and a compliance probability, which specifies the minimal acceptable fraction of customers to be seen by this time. Clearly, \apqname{}s provide more flexibility than, for instance, classical \emph{non-preemptive priority queues} (\npqname{}s), in terms of fine-tuning a queueing system to better comply with a given set of KPIs. This extra flexibility is provided by the accumulation rates we are free to choose. The \dapqname{} model goes one step further, in terms of its choice of initial delay period during which \lowname{} customers do not accumulate priority.

\citet{mojalal19dapq} relate the \lowname{} waiting time distribution in the \dapqname{} to that of the \lowname{} waiting time distribution in a related \npqname{}. In particular, they establish that, up to the end of the delay period, there is no difference in the actual waiting times incurred in these queues. Additionally, they provide a generic formula for the \lowname{} waiting time distribution in terms of the number of customers found in system.
However, the formula is not easily implemented to compute the necessary terms, and does not provide any information about the waiting time of \highname{} customers.
While an arriving \highname{} customer will necessarily wait for all other \highname{} customers they find in system upon arrival, the same thing cannot be said about the \lowname{} customers they find. Indeed, the longer the arriving \highname{} customer waits, the greater the amount of accrued credit they earn, leading to a greater likelihood that their credits will exceed some or all of the \lowname{} customers they find in the system. Consequently, existing techniques cannot be used to compute the waiting time distribution of \highname{} customers.

Turning to the problem this poses for the analysis of the \dapqname{}, it means we have an incomplete set of tools to determine the best \dapqname{} to meet the KPIs of some two-class health care systems that might employ such a strategy. For each delay period, subject to computation of the \lowname{} waiting time, we can determine the optimal accumulation rate to comply with the KPI for the \lowname{} patient class. However, we have no such tool to do so for the \highname{} class. Other than simulating such systems, which defeats the purpose of developing an analytical tool in the first place, we have, at present, 
no means to assess compliance of a \highname{} KPI of the delay-limit-and-compliance-level sort.

This paper is intended to address both of these existing limitations.
We provide the exact expected waiting time for \emph{both} the \lowname{} and \highname{} classes, as well as a computational algorithm to evaluate these by truncating an infinite sum. 
We are able to do so because the \dapqname{} is a work-conserving queueing system, and as such the expected delays incurred in it must obey the M/G/1 conservation law for waiting times \citep[see][]{kleinrock65conservation}. 
This development means that we have a more complete package available to analyse KPIs for a two-class \dapqname{}. The consequence of this for delay-limit-and-compliance-level KPIs is that we are able to quantify the impact on the expected waiting time of the \highname{} class of various combinations of parameter values (initial delay and priority accumulation rates). Thus, we have two pieces of information to optimize over with these two free parameters, and so for every set of KPI values it is possible to define an optimal choice of parameter values (which we find to be unique in our experiments).

While the exactness of the \highname{} expected waiting time is desirable for optimizing queueing parameters relative to health care KPIs, it is still unable to capture the tail behaviour of the waiting time, which may change the optimal choice of queueing parameters. 
Consequently, we also propose a zero-inflated exponential approximation of the \highname{} waiting time, whose efficacy we demonstrate using both exact \apqname{} waiting time CDFs and simulated \dapqname{} waiting time CDFs.
This provides all the information needed to fully optimize the queueing parameters for health care KPIs, subject to approximation error.

The rest of the paper is arranged as follows. \cref{sec:notation} defines notation and reviews the current theoretical results relating to the \dapqname{}. 
\cref{sec:results} contains our new analytical expressions for the \classname{}-1 expected waiting time in the \dapqname{} along with a detailed analysis of the delay level's impact on this. 
We apply this analysis to health care KPIs in \cref{sec:optimizing}.
In \cref{sec:approximation}, we define an exponential approximation for the \highname{} waiting time distribution, and evaluate its quality through numerical experiments and simulation.
We also present further optimization for health care KPIs using this approximation.
We conclude the paper in \cref{sec:conclusions} with our observations and future theoretical directions. All code is available at \url{https://github.com/blairbilodeau/delayed-apq-avg-wait}.

\section{Notation and preliminaries} \label{sec:notation}

Consider two classes of customers, labelled \classname{}-1 and \classname{}-2, where by convention \classname{}-1 has more urgency to be seen than \classname{}-2. Suppose they experience exponential inter-arrival times with rates $\lam_1, \lam_2 \in (0,\infty)$, so that overall customers arrive to the system at rate $\lam = \lam_1 + \lam_2$. Let $\serv$ denote the common service time of any customer and $1/\mu = \EE \serv$. As usual, define the corresponding occupancy rates $\rho_1 = \lam_1 / \mu$ and $\rho_2 = \lam_2 / \mu$, and for stability assume that $\rho = \lam / \mu < 1$. 
Since the relevant results depend only on the ratio of the \classname{}-2 accumulation rate to that of \classname{}-1, let \classname{}-1 customers accumulate priority at rate one credit per time unit and \classname{}-2 customers accumulate at rate $b \in [0,1]$. Furthermore, in the \dapqname{}, there is some $d \geq 0$ such that all \classname{}-2 customers only begin accumulating after they have been in system for $d$ units of time. The queueing discipline is such that at every service completion, the customer with the highest accumulated priority enters into service immediately, with no preemption, and consequently the server is only idle when the system is empty.

Denote the waiting time random variable of a customer by $\wait$, with a superscript to specify the queueing discipline and subscript to specify the priority class when required. For example, $\waitD_1$ is the waiting time of a \classname{}-1 customer in a \dapqname{}, while $\waitF$ is the waiting time of a customer in a \emph{first-come-first-serve} (FCFS) queue. For any random variable $X$ that has distribution function $\cdf$, denote the \emph{Laplace-Stieltjes transform} (LST) of $\cdf$ by $\lst(s) = \EE e^{-sX}$. We also introduce the notation $\lst(s;d) = \EE \left[e^{-sX} \1\{X > d\} \right]$. Let the CDF of the service time be $\cdf^\serv(x) = \PP [\serv \leq x]$ and have LST $\lst^\serv$. The same superscript and subscript notation is used to denote the distribution function of a waiting time; for example, $\cdfD_1(x) = \PP\left[\waitD_1 \leq x \right]$ and $\lstD_1(s) = \EE \exp\left\{-s\waitD_1 \right\}$.
Unless otherwise stated, we suppress the dependence of all random variables on $d$ and $b$ to simplify the notation.

We will make frequent use of the notion of an \emph{accreditation interval}, first introduced in \citet{stanford14apq}. For completeness, we restate the key concepts here. 
At any time $t$, 
let $\tau_t$ denote the time of the most recent service commencement.
Further, let $V(\tau_t)$ denote the amount of priority accumulated at time $\tau_t$ by the customer who commenced service at time $\tau_t$.
Define $M_2(t)$ to be the maximum amount of priority a \classname{}-2 customer could have accumulated by time $t$, given only the previous times at which a service commenced.
More precisely, $M_2(t)=0$ when the queue is empty, and otherwise it is defined recursively by 
\*[
  M_2(t) = \min\{M_2(\tau_t), V(\tau_t)\} + b(t-\tau_t).
\]
Note that the definition of $M_2(t)$ depends on $b$ but not $d$.

A \classname{}-1 customer becomes \emph{accredited} once their accumulated priority at time $t$ is strictly larger than $M_2(t)$.
Accreditation can never be undone, since the \classname{}-1 accumulated priority grows linearly with slope $1$ while $M_2(t)$ grows at most linearly with slope $b\leq 1$. 
A \classname{}-2 customer is always unaccredited by definition. 
Crucially, an unaccredited customer will not be served until there are no accredited customers remaining in system. Further, each unaccredited customer entering into service generates an \emph{accreditation interval}, which consists of their service time plus the service times of all accredited customers served before the next unaccredited service. We denote the CDF of the random variable corresponding to the length of an accreditation interval by $\acc$ and the LST by $\accs$, with a superscript for the queueing discipline. 
Lemma~4.2 of \citet{stanford14apq} shows that, under Poisson arrivals, customers in the APQ become accredited according to a Poisson process at rate $\lamA_1 = \lam_1(1-b)$, which we refer to as the \emph{accreditation rate}. For notational simplicity, we define $\rhoA_1 = \lamA_1 / \mu$. 

Relative to a specific tagged customer, let $\num_t$ denote the number of customers ahead of them in system (including the one in service) $t$ units of time after their arrival, and $\sta_i = \PP[\num_0 = i]$ denote the stationary distribution of the number of customers found in system upon arrival. Let $\resid{}$ denote the residual service time of the customer currently in service upon arrival of the tagged customer, and denote its CDF by $\cdfres{}$ with LST $\lstres{}$. 
For the random variable corresponding to the length of a \emph{residual accreditation interval}, which is composed of $\resid{}$ plus the service times of all accredited customers served before the next unaccredited service, we denote the CDF by $\accres{}$ and the LST by $\accsres{}$.
For any $j \in \N$, define the conditional CDF $\cdfresc{j}(t) = \PP[\resid{} \leq t \mid \num_d = j]$ and the conditional LST $\lstresc{j}$. 
Finally, conditional on $\num_d=j$, denote the conditional CDF of the residual accreditation interval length by $\accresc{j}$, and denote the corresponding conditional LST by $\accsresc{j}$.
Observe that the number in system and the residual service time are independent of the queueing discipline, depending only on the arrival rates and service distribution, while the residual accreditation interval depends on the queueing discipline, which will be denoted by a superscript as usual. 

We are now able to restate the following main results from \citet{mojalal19dapq} that will be used in the remainder of the paper. While the results are stated out of order from the original paper, we feel this is more natural for observing how the M/M/1 \dapqname{} is a special case of the M/G/1 \dapqname{} where the residual accreditation interval has the same distribution as a standard accreditation interval.

$ $\\
\textbf{Corollary~3.1 of \citet{mojalal19dapq}} (M/G/1 \classname{}-2 Equivalence)\infor{.}
\*[
  \lstD_2(s) - \lstD_2(s;d) = \lstN_2(s) - \lstN_2(s;d).
\]

$ $\\
\textbf{Corollary~3.2 of \citet{mojalal19dapq}} (M/G/1 \classname{}-2 LST)\infor{.}
\*[
  \lstD_2(s;d) = \sum_{i=1}^\infty \sta_i \sum_{j=1}^\infty \PP \left[\num_d = j, \num_t > 0 \ \forall \ t \in [0,d) \big\lvert \num_0 = i \right] e^{-sd} \accDsresc{j}(s) \left(\accDs(s) \right)^{j-1},
\]
\emph{where}
\*[
  \accDs(s) = \lst^\serv \left(s + \lamA_1 (1-\accDs(s)) \right),
\]
\emph{and}
\*[
  \accDsresc{j}(s) = \lstresc{j} \left(s + \lamA_1 (1-\accDs(s)) \right).
\]

$ $\\
\textbf{Theorem~3.2 of \citet{mojalal19dapq}} (M/M/1 \classname{}-2 LST)\infor{.}
\*[
  \lstD_2(s;d) = \sum_{i=1}^\infty \sta_i \sum_{j=1}^\infty \PP \left[\num_d = j, \num_t > 0 \ \forall \ t \in [0,d) \big\lvert \num_0 = i \right] e^{-sd} \left(\accDs(s) \right)^j,
\]
\emph{where}
\*[
  \accDs(s) = \frac{s+\mu+\lamA_1 - \sqrt{(s+\mu+\lamA_1)^2 - 4\mu\lamA_1}}{2\lamA_1}.
\] 
\section{Impact of delay level on expected waiting times}
\label{sec:results}

In this section, we provide analytical expressions that can be evaluated by truncating an infinite sum to find both low and \highname{} expected waiting times in the \dapqname{}. These expressions allow us to visualize the impact of the delay level on the \highname{} expected waiting time, providing a deeper understanding of the dynamics of the \dapqname{}. Then, in \cref{sec:optimizing}, we use our computation algorithm for the analytical expected waiting time expressions to choose the optimal parametrizations for the \dapqname{} under various conditions.

\subsection{Computation of waiting times} \label{sec:computation}

The primary takeaway of the following results is that we have analytical statements that can be implemented as an algorithm requiring only the truncation of infinite sums that converge quickly in our experiments. This allows for rapid testing of various parameters to tune the accumulation and delay rates to meet any KPIs of interest to practitioners. We present the results for both exponential and deterministic queueing disciplines, and while the derivation will follow the same strategy, the specific expression will change significantly for an alternative service distribution.

The first result that we use to obtain our results is the application of Corollary~3.2 from \citet{mojalal19dapq} to the \npqname{}, corresponding to $b=0$. This lemma shows that, in addition to the previously known fact that the \npqname{} and \dapqname{} waiting times agree when within the delay period, their divergence following the end of the delay period is completely characterized by their different accreditation rates. Thus, in order to quantify the expected waiting time, we simply need to compute the expected increase from the differing accreditation rate and combine it with known expected waiting times for the \npqname{}.

\begin{lemma}\label{lem:npq}
\*[
  \lstN_2(s;d) = \sum_{i=1}^\infty \sta_i \sum_{j=1}^\infty \PP \left[\num_d = j, \num_t > 0 \ \forall \ t \in [0,d) \big\lvert \num_0 = i \right] e^{-sd} \accNsresc{j}(s) \left(\accNs(s) \right)^{j-1},
\]
where
\*[
  \accNs(s) = \lst^\serv \left(s + \lam_1 (1-\accNs(s)) \right),
\]
and
\*[
  \accNsresc{j}(s) = \lstresc{j} \left(s + \lam_1 (1-\accNs(s)) \right).
\]
\end{lemma}

Additionally, the key fact that allows us to obtain the \classname{}-1 expected waiting time is that for work-conserving queues, the expected waiting time between classes can be related to the FCFS waiting time by their respective occupancy rates.

$ $\\
\textbf{Theorem~1 of \citet{kleinrock65conservation}} (Work-Conserving Conservation Law)
\emph{For any queue with $K$ classes, each with a Poisson arrival rate of $\lam_k$ and common service distribution $\serv$, and a single-server non-preemptive queueing discipline,}
\*[
  \frac{\rho}{1-\rho} \frac{\lambda \EE \serv^2}{2} = \sum_{k=1}^K \rho_k \EE \wait_k.
\]

Observe that the \dapqname{} (which includes the \apqname{} and \npqname{} as special cases) satisfies the conditions of this theorem; that is, all of these queues are work-conserving. Thus, we can apply these results to obtain analytical expressions for the average waiting time of both \classname{}-1 and \classname{}-2 customers in the M/M/1 and M/D/1 \dapqname{}s. The main technique is to differentiate the LST expressions for the waiting time, leading to expectations, and then compute only the difference between these terms for the \dapqname{} and the \npqname{}. The cancellation within this difference allows for the computation of the expected waiting time rather than only the expected waiting time conditional on whether it falls before or after the delay period. The terms are then simplified to provide an explicit implementation; full derivations are provided in \cref{app:proofs}.

\begin{theorem}[M/M/1 Expected Waiting Time Computation]\label{thm:mm1_avg}
\*[
  \EE \left[\waitD_2\right] 
  = \ & \frac{\rho}{\mu (1-\rho_1)(1-\rho)} - \frac{\rho_1 b}{\mu (1 - \rhoA_1)(1 - \rho_1)} \times \nonumber \\ 
  &\hspace{5pt} \left[(1-\rho)\sum_{k=0}^\infty \frac{e^{-\nu d} (\nu d)^k}{k!} \left(\sum_{\ell=1}^k \ell x_\ell^{(k)}  \right) +
    \rho e^{-\nu d + r \nu d} \left(\frac{1}{1-\rho} + r \nu d \right) \right],
\]
where $q = \frac{\mu}{\mu + \lam_1}$, $p = \frac{\lam_1}{\mu + \lam_1}$, $r=p+q\rho^2$, and $\nu = \mu + \lambda_1$; the $x_\ell^{(k)}$'s are defined recursively as
\*[
  x_1^{(1)} = r-p,
  \quad
  x_1^{(2)} = q r p,
  \quad
  x_2^{(2)} = r^2 - p^2,
\]
and for $k \geq 3$,
\*[
  &x_1^{(k)} = q x_2^{(k-1)},
  \quad
  x_k^{(k-1)} = \rho r^{k-1}, \\
  &x_\ell^{(k)} = p x_{\ell-1}^{(k-1)} + q x_{\ell+1}^{(k-1)} \ \text{ for } \ \ell \in \{2,\dots,k-1\}, \\
  &x_k^{(k)} = r^k - p^k.
\]
Furthermore,
\*[
  \EE \left[\waitD_1\right] = \frac{1}{\rho_1} \left(\frac{\rho^2}{\mu(1 - \rho)} - \rho_2 \EE \left[\waitD_2\right] \right).
\]
\end{theorem}

Next, we consider the M/D/1 case. The additional difficulty comes from the fact that service is no longer memoryless, which leads to a more complex expression. However, once the residual service times are handled using results from \citet{adan09residual}, the result follows from the same method as for the M/M/1. Note that we have the same limitation as in \citet{mojalal19dapq}, where the delay level must be an integer multiple of the mean service length.

\begin{theorem}[M/D/1 Expected Waiting Time Computation]\label{thm:md1_avg}
If $d=0$,
\*[
  \EE \left[\waitD_2\right] = \frac{\rho_1 b \rho}{2\mu(1 - \rhoA_1)(1 - \rho_1)(1-\rho)}.
\]
\preprint{Otherwise, when $d=\ell/\mu$ for $\ell \in \N$,
\*[
  \EE \left[\waitD_2\right]
  = \ & \frac{\rho}{2\mu (1-\rho)(1-\rho_1)}  
  - \frac{\rho_1 b e^{-\lam_1 d}}{\mu(1 - \rhoA_1)(1 - \rho_1)} \times \nonumber \\
  &\hspace{5pt} \sum_{j=1}^\infty \Bigg\{ 
    \sum_{k=2}^{j+\ell} \sum_{a=0}^{j+\ell-k} 
    \frac{(-1)^a (\lam_1 d)^{j+\ell-k}}{d^a (j+\ell-k-a)! a!}
    \sum_{n=0}^{k-1} \frac{\sta_{k-n} \rho_1^n}{n!\mu} \left[\frac{j-1}{n+a+1} - \frac{1}{n+a+2} \right] + \nonumber \\
    &\hspace{30pt}
    \sum_{k=2}^\ell \sum_{m=k}^\ell
    \frac{(m-1)^{m-k}}{(m-k)!} 
    \left(\frac{k-1}{m-1} \right)
    \sum_{a=0}^{j+\ell-m} \frac{(-1)^a (\lam_1 d - \rho_1(m-1))^{j+\ell}}{(d - (m-1)/\mu)^{m+a}(j+\ell-m-a)! a!} \times \nonumber \\
    &\hspace{35pt}
    \sum_{n=0}^{k-1} \frac{\sta_{k-n} \lambda_1^{n-k}}{n!} \left[\frac{j-1}{n+a+1} + \frac{1}{n+a+2} \right]
  \Bigg\},
\]}
\infor{Otherwise, when $d=\ell/\mu$ for $\ell \in \N$, $\EE \left[\waitD_2\right]$ is equal to
\*[
  &\hspace{-1em}\frac{\rho}{2\mu (1-\rho)(1-\rho_1)}  
  - \frac{\rho_1 b e^{-\lam_1 d}}{\mu(1 - \rhoA_1)(1 - \rho_1)} \times \nonumber \\
  &\hspace{5pt} \sum_{j=1}^\infty \Bigg\{ 
    \sum_{k=2}^{j+\ell} \sum_{a=0}^{j+\ell-k} 
    \frac{(-1)^a (\lam_1 d)^{j+\ell-k}}{d^a (j+\ell-k-a)! a!}
    \sum_{n=0}^{k-1} \frac{\sta_{k-n} \rho_1^n}{n!\mu} \left[\frac{j-1}{n+a+1} - \frac{1}{n+a+2} \right] + \nonumber \\
    &\hspace{15pt}
    \sum_{k=2}^\ell \sum_{m=k}^\ell
    \frac{(m-1)^{m-k}}{(m-k)!} 
    \left(\frac{k-1}{m-1} \right)
    \sum_{a=0}^{j+\ell-m} \frac{(-1)^a (\lam_1 d - \rho_1(m-1))^{j+\ell}}{(d - (m-1)/\mu)^{m+a}(j+\ell-m-a)! a!} \times \nonumber \\
    &\hspace{20pt}
    \sum_{n=0}^{k-1} \frac{\sta_{k-n} \lambda_1^{n-k}}{n!} \left[\frac{j-1}{n+a+1} + \frac{1}{n+a+2} \right]
  \Bigg\},
\]}
where $\sta_i$ is given by
\*[
  \sta_i = (1-\rho) \left\{e^{i\rho} + (-1)^i \sum_{k=1}^{i-1} (-1)^k \frac{e^{k\rho} (k\rho)^{i-k}}{(i-k)!}\left[\frac{i - k(1-\rho)}{k\rho} \right] \right\}.
\]
Also,
\*[
  \EE \left[\waitD_1\right] = \frac{1}{\rho_1} \left(\frac{\rho^2}{2\mu(1-\rho)} - \rho_2 \EE \left[\waitD_2\right] \right).
\]
\end{theorem}

\begin{remark}
The formula for $\sta_i$ is difficult to implement efficiently for large $i$, but can be approximated by $\sta_{i+1} / \sta_i = \tilde \sigma$, where $\tilde \sigma$ solves $e^{\rho \sigma} = \sigma e^\rho$ (c.f.\ Appendix~C of \citet{tijms94stochastic}).
\end{remark}

The utility in computing this result for the M/D/1 is that it allows the effect of the service time variation on \highname{} waiting times to be isolated. We may then approximate the average waiting time for a service distribution with the same mean but smaller variance than exponential service by simple interpolation between the M/D/1 and the M/M/1.

\subsection{Visualization of delay level impact}
Using a truncation of the infinite sums from \cref{thm:mm1_avg,thm:md1_avg} (which is exact in the infinite sum limit), we are now able to visualize the effect of introducing a delay on the \highname{} waiting time.
The truncation is performed such that the individual contribution of the terms has become smaller than $10^{-5}$.
The computations were performed in the R programming language on a 2017 Macbook Pro with 16GB of RAM, and all took (sometimes significantly) less than 5 minutes of run time.
We did not conduct an extensive study of computational complexity, and instead only wish to highlight that the computation timescale is minutes rather than days, and that arbitrarily better accuracy can be achieved by sacrificing run time in favour of computing more terms in the sum.

Recall that the accumulation parameter $b$ and delay parameter $d$ generalize both the FCFS and \npqname{} regimes. Specifically, the \apqname{} with $b=1$ corresponds to FCFS, the \apqname{} with $b=0$ or the \dapqname{} with $d=\infty$ correspond to the \npqname{}, and the \dapqname{} with $d=0$ corresponds to the \apqname{}. Consequently, increasing the delay level will yield a waiting time between that of the \apqname{} and that of the \npqname{}, where the former has the shortest \classname{}-2 waiting times and the latter has the shortest \classname{}-1 waiting times. To demonstrate this interpolation, we present the results for how changing the accumulation rate and delay period affects the expected waiting time for both \classname{}-1 and \classname{}-2 customers. 

\cref{fig:mm1_bvals} shows the effect of varying accumulation rate within $[0,1]$ on the \classname{}-1 (left panel) and \classname{}-2 (right panel) expected waiting times for M/M/1. We discuss the effect on \classname{}-1, since the \classname{}-2 values are just a vertical reflection and scaling by occupancy due to the M/G/1 conservation law. Observe that, by definition, the \npqname{} expected waiting time is unaffected by accumulation rate. However, for each delay level, the curve begins equal to \npqname{} at $b=0$, and then increases sub-linearly as $b$ tends to 1. This confirms that allowing \classname{}-2 customers to accumulate credit more rapidly penalizes \classname{}-1 customers, but reveals that this is marginally less impactful as the limit of $b=1$ is approached. Furthermore, as $d$ gets smaller, the curves shift up vertically, approaching the \apqname{}, which corresponds to $d = 0$. The continuous effect of this change is explored in \cref{fig:mm1_dvals}.

\preprint{
\begin{figure}[htb]
  \centering
  \subfloat{
    \includegraphics[width=80mm, height=70mm]{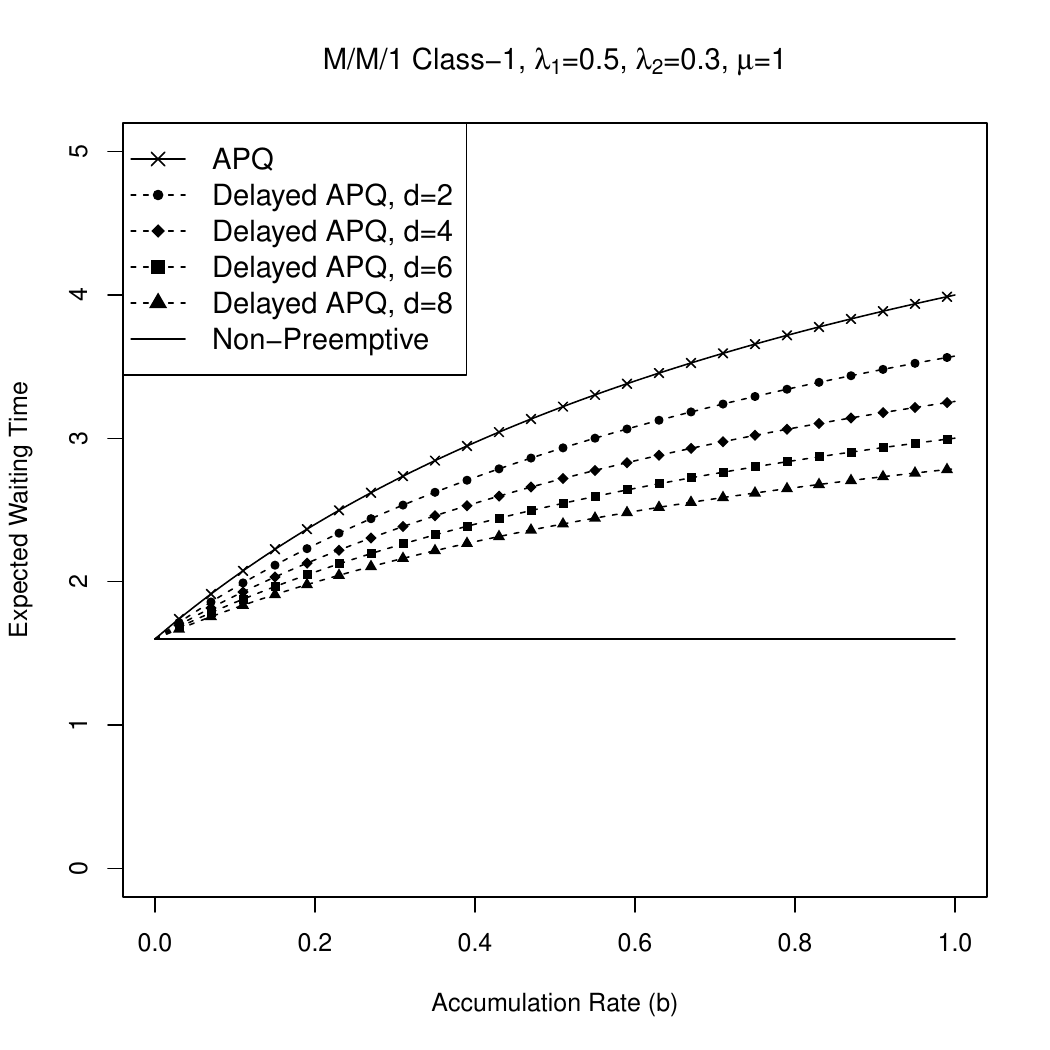}
  }
  \subfloat{
    \includegraphics[width=80mm, height=70mm]{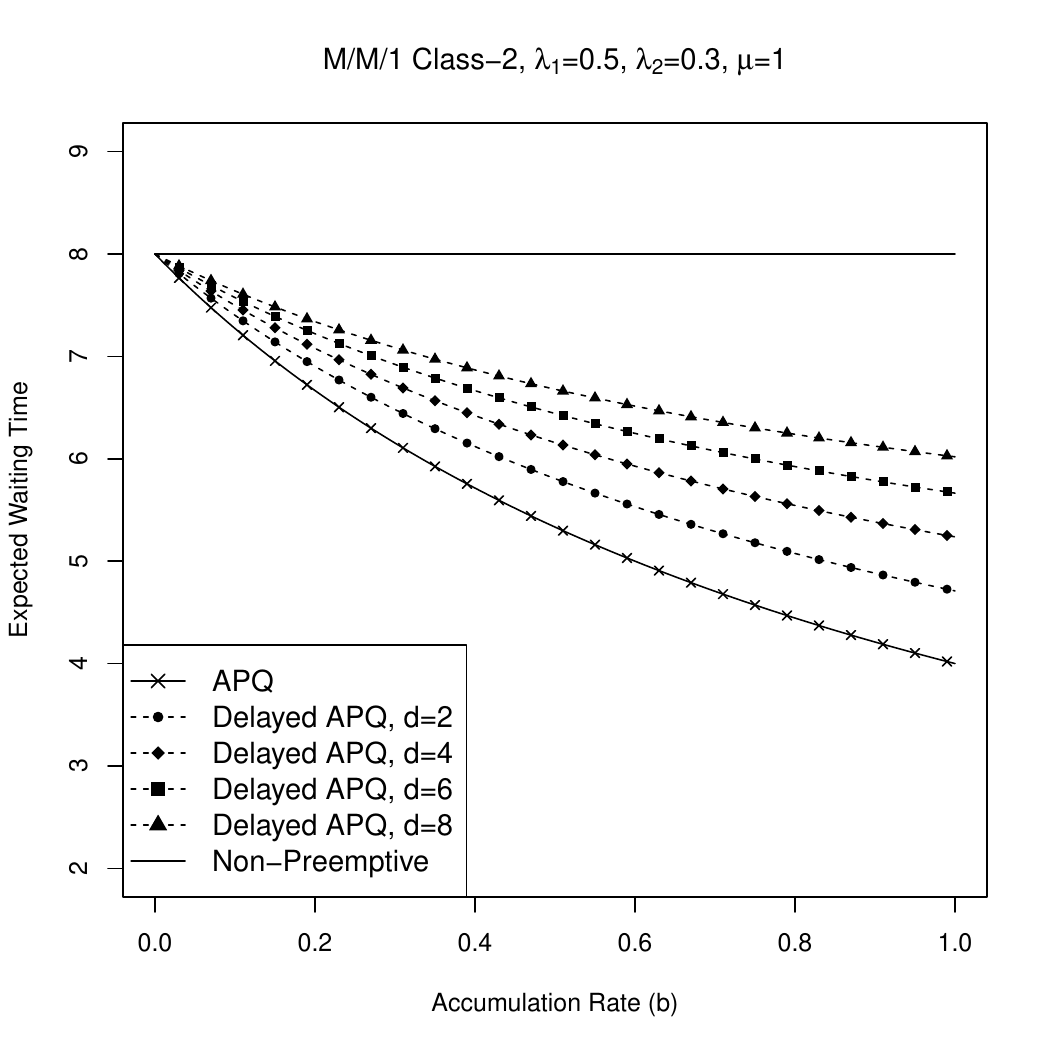}
  }
  \caption{The effect of accumulation rate $b$ on expected waiting time for the M/M/1 \dapqname{} when $\rho = 0.8$.}
  \label{fig:mm1_bvals}
\end{figure}}
\infor{
\begin{figure}[htb]
  \centering
  \subfloat{
    \includegraphics[width=70mm, height=55mm]{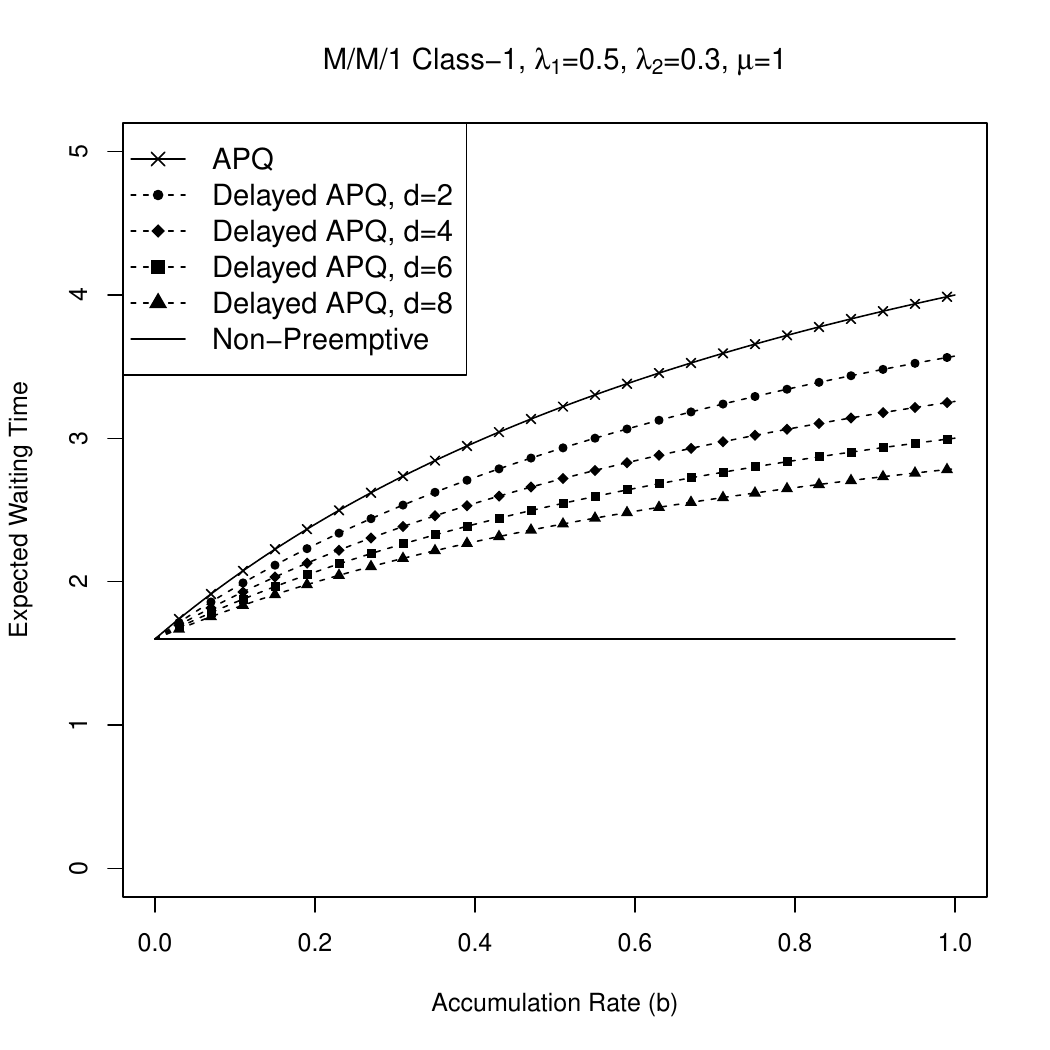}
  }
  \subfloat{
    \includegraphics[width=70mm, height=55mm]{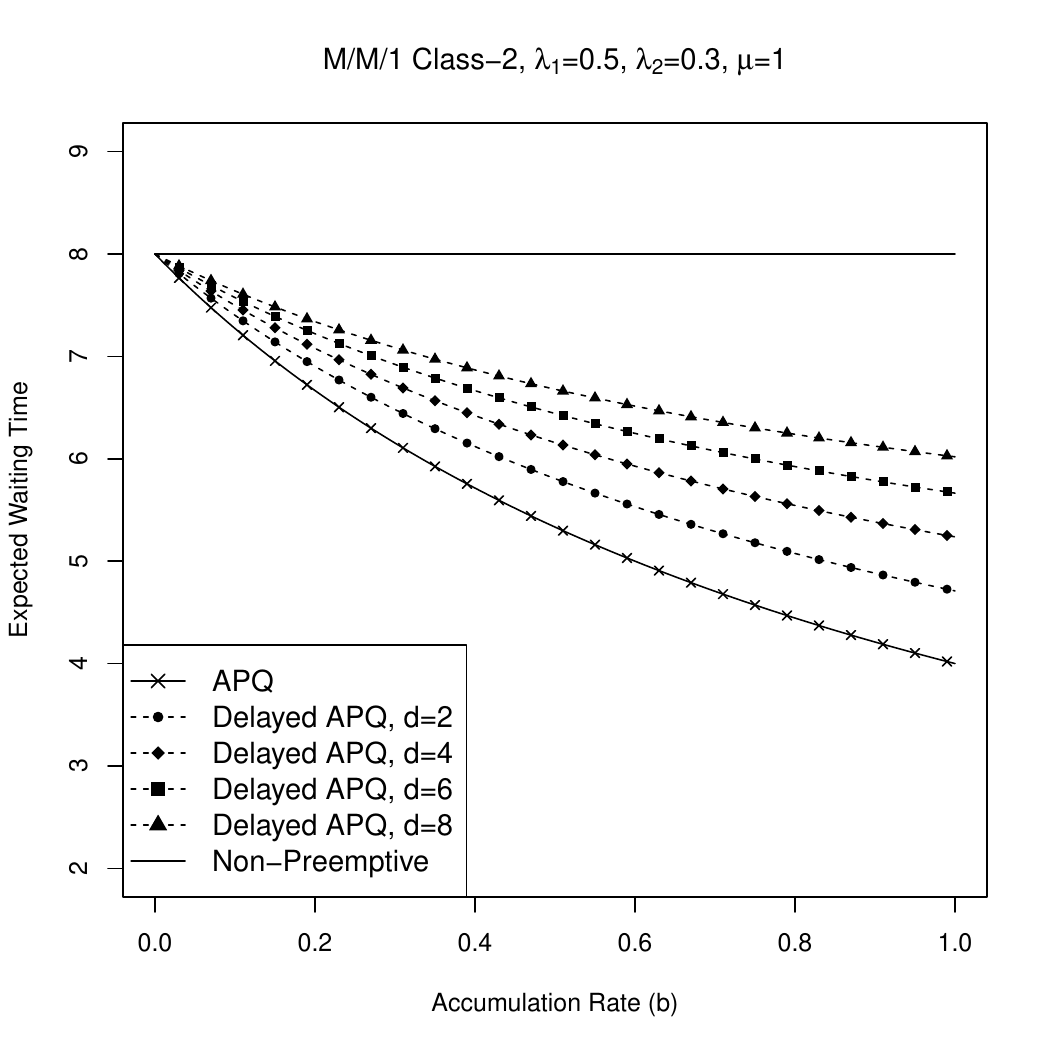}
  }
  \caption{The effect of accumulation rate $b$ on expected waiting time for the M/M/1 \dapqname{} when $\rho = 0.8$.}
  \label{fig:mm1_bvals}
\end{figure}}

The same patterns apply for the M/D/1 case in \cref{fig:md1_bvals}, although all the waiting times are lower as there is no longer variation in the service times. It is of interest that the effect seems to be roughly halving the wait, which is exactly the impact on the expected waiting time in a FCFS queue when moving from M/M/1 to M/D/1. 

\preprint{
\begin{figure}[htb]
  \centering
  \subfloat{
    \includegraphics[width=80mm, height=70mm]{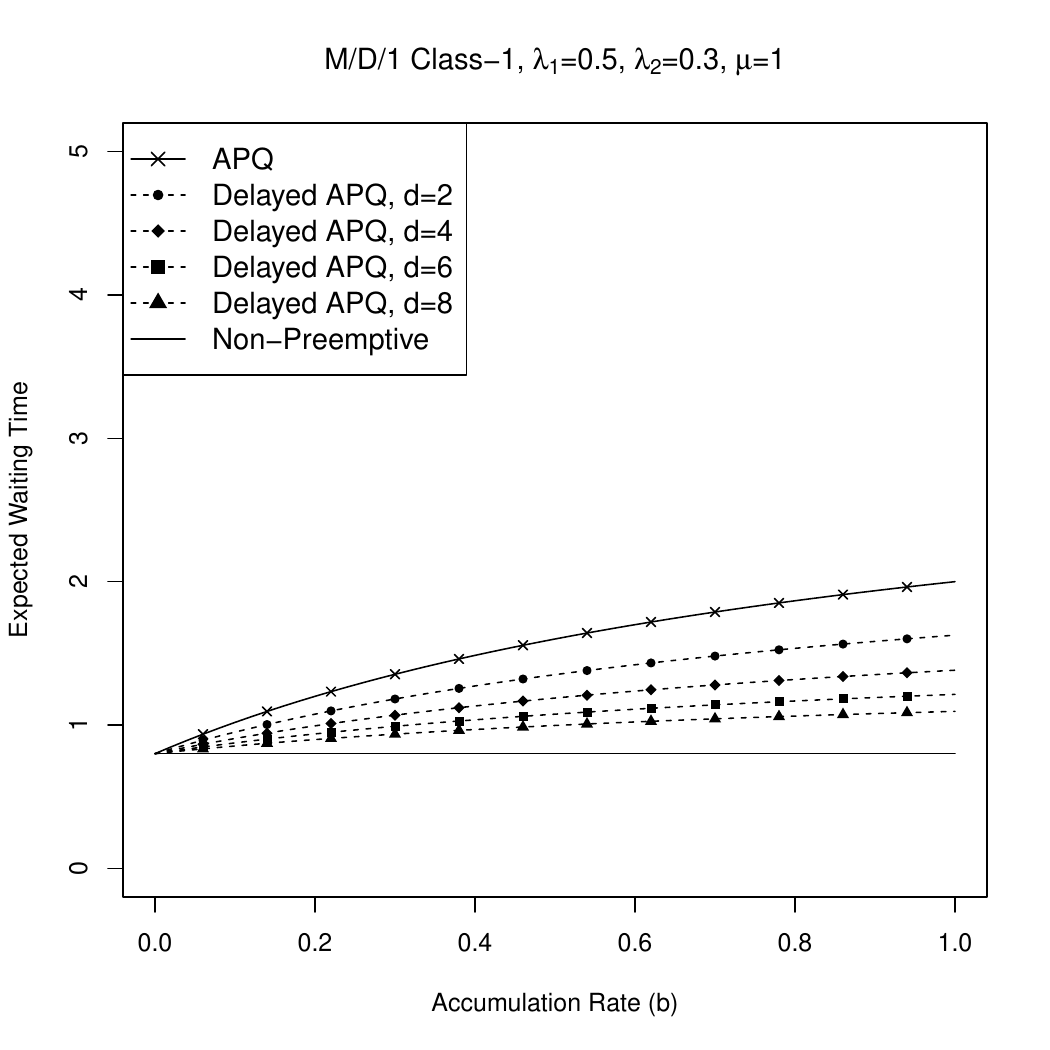}
  }
  \subfloat{
    \includegraphics[width=80mm, height=70mm]{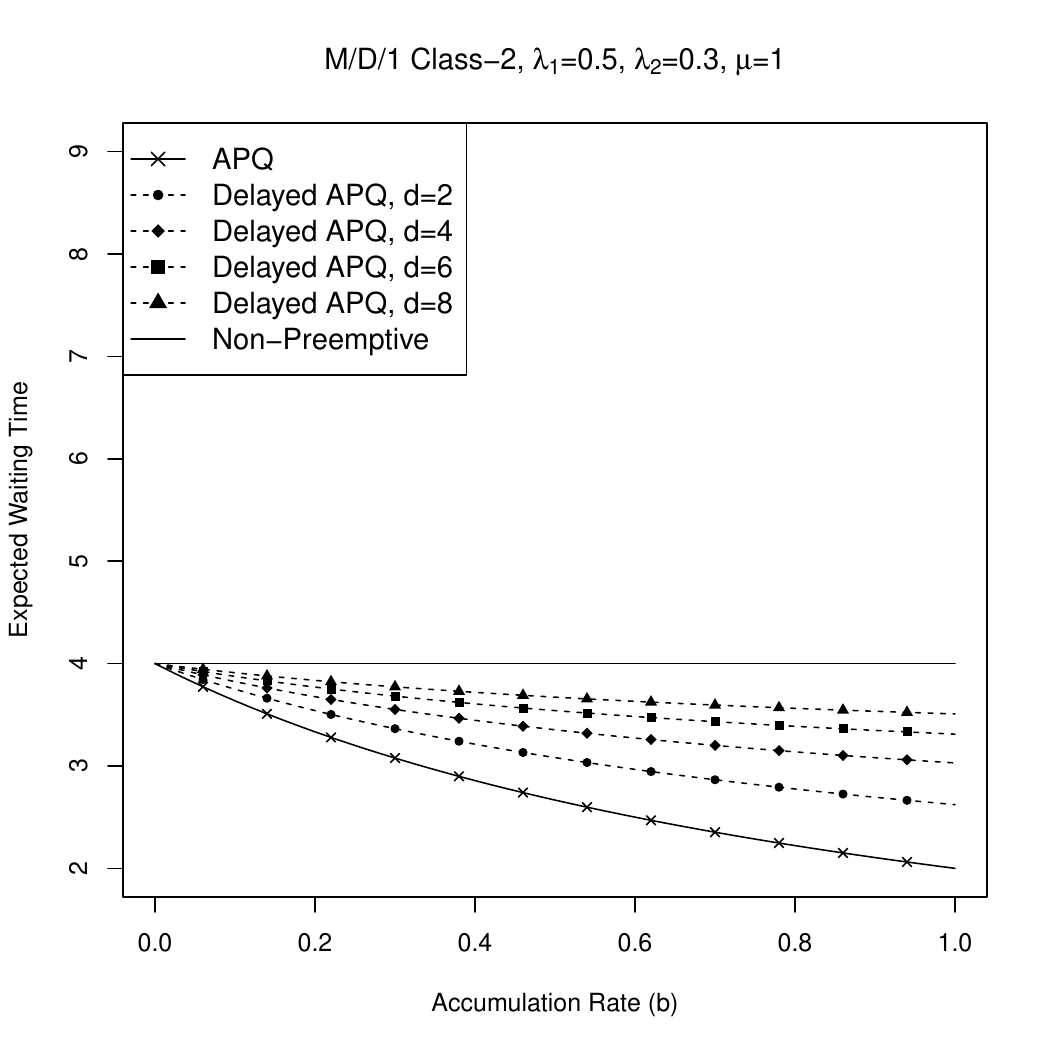}
  }
  \caption{The effect of accumulation rate $b$ on expected waiting time for the M/D/1 \dapqname{} when $\rho = 0.8$.}
  \label{fig:md1_bvals}
\end{figure}}
\infor{
\begin{figure}[htb]
  \centering
  \subfloat{
    \includegraphics[width=70mm, height=55mm]{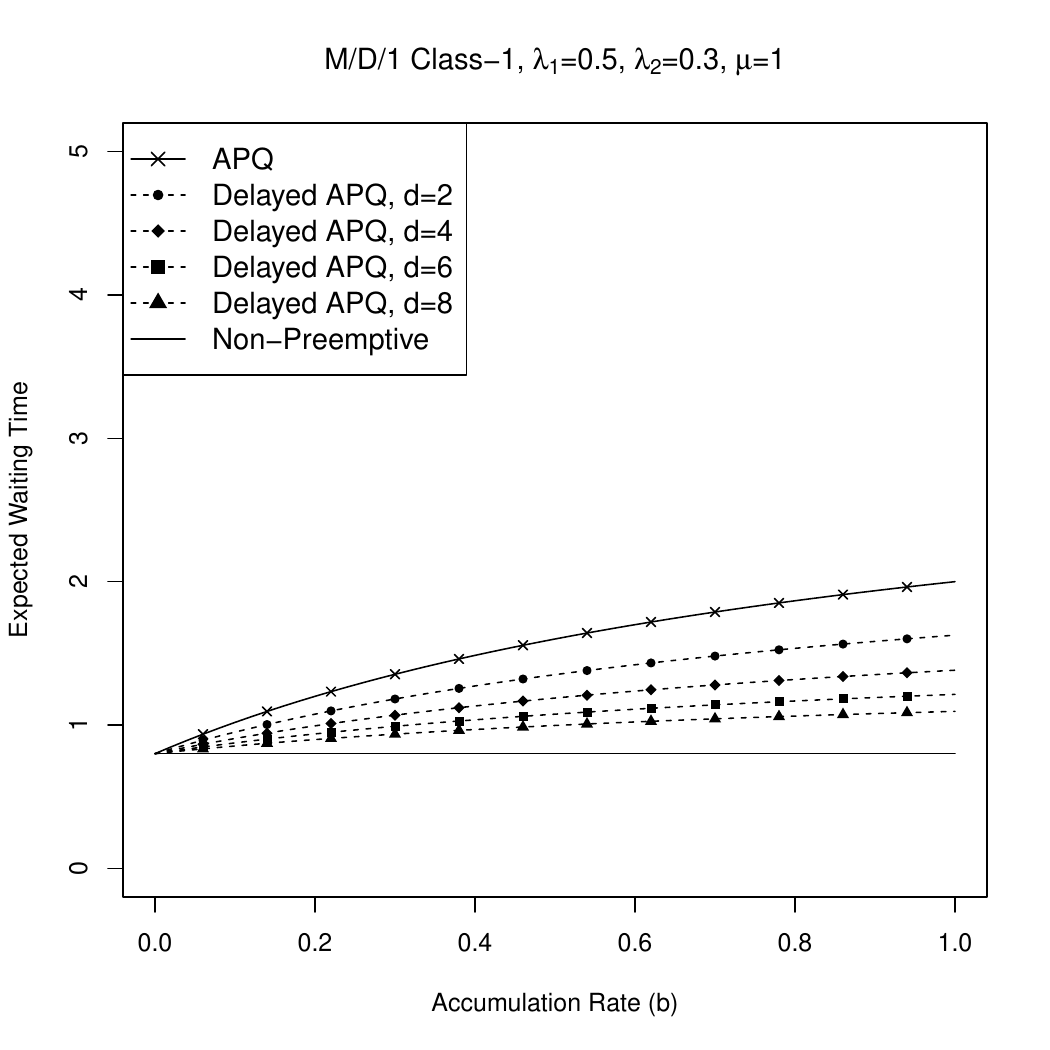}
  }
  \subfloat{
    \includegraphics[width=70mm, height=55mm]{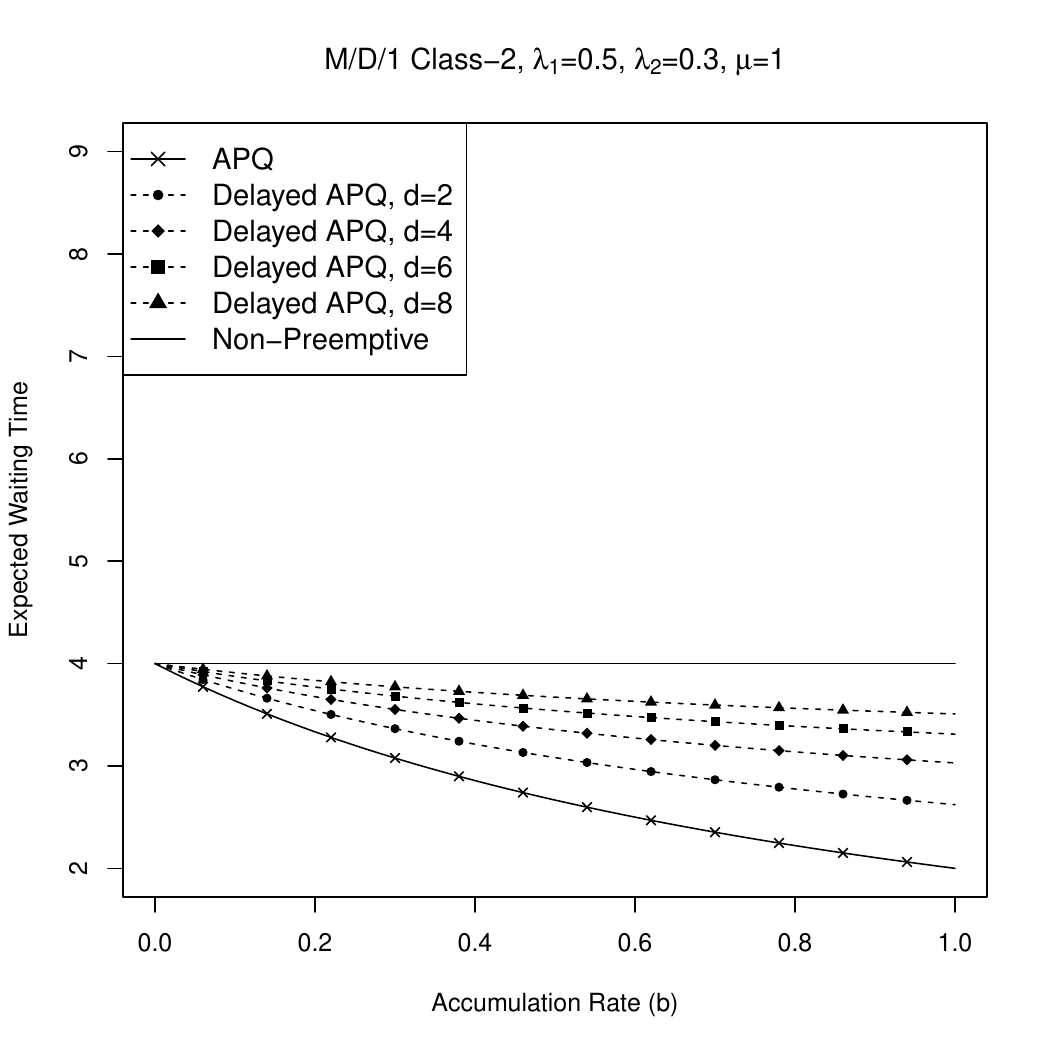}
  }
  \caption{The effect of accumulation rate $b$ on expected waiting time for the M/D/1 \dapqname{} when $\rho = 0.8$.}
  \label{fig:md1_bvals}
\end{figure}}

Next, we are interested in the effect of changing $d$ over different values of $b$ in \cref{fig:mm1_dvals}, where again the left panel pertains to \classname{}-1 and the right panel pertains to \classname{}-2. The FCFS case (corresponding to $b=1$ and $d=0$) will provide expected waiting times that act as an upper bound for the \classname{}-1 expected waiting time. Then, starting from $d=0$ (the \apqname{} expected waiting time), each $b$ curve decreases smoothly towards the \npqname{} expected waiting time, which corresponds to $b=0$. While we observe that the marginal impact of increasing $d$ always becomes smaller as $d$ becomes very large, the initial changes are much more pronounced (steeper slope) for small values of $b$. The same patterns hold, with the same scaling of about 1/2, for the M/D/1 case in \cref{fig:md1_dvals}. Again, as mentioned for \cref{thm:md1_avg}, we can only compute this at integer multiples of the mean service length (one, in this case).

\preprint{
\begin{figure}[htb]
  \centering
  \subfloat{
    \includegraphics[width=80mm, height=70mm]{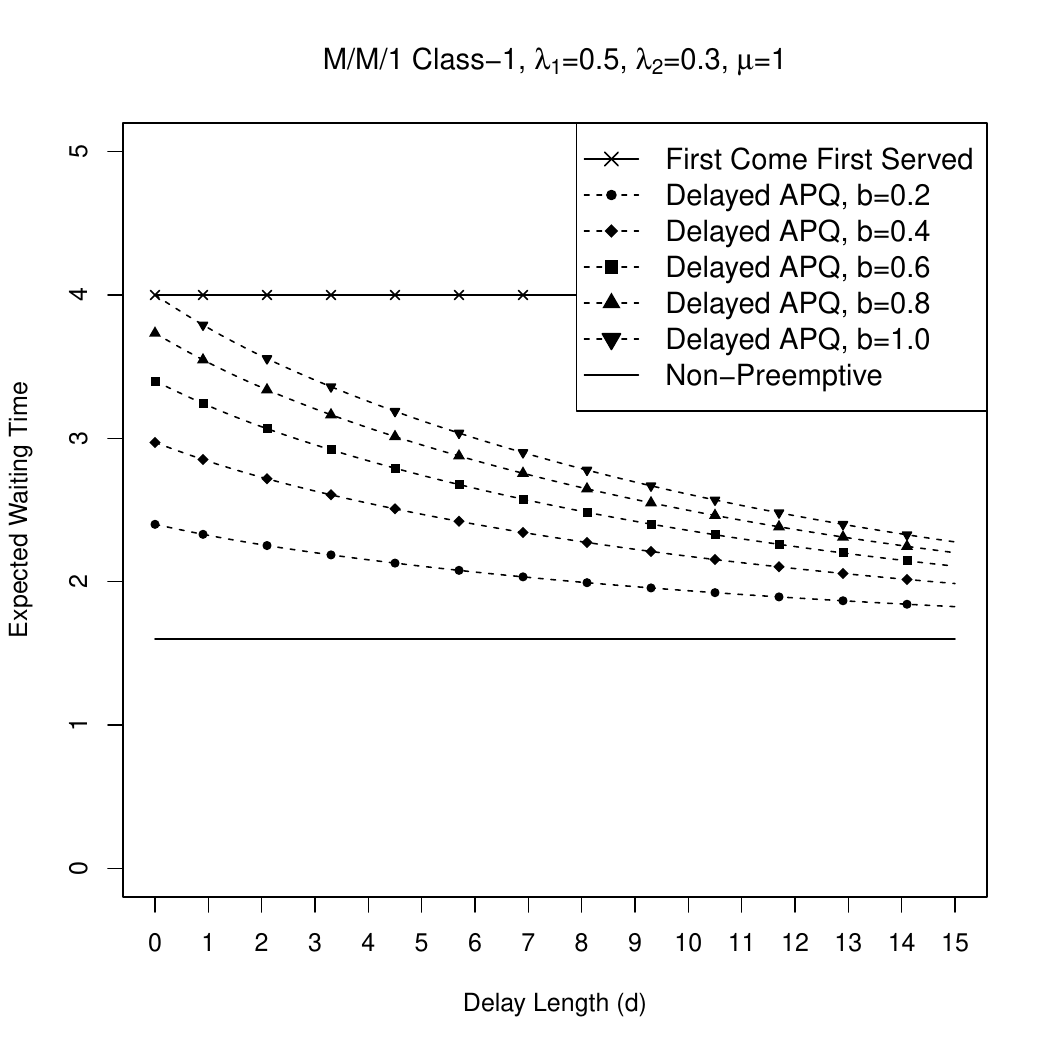}
  }
  \subfloat{
    \includegraphics[width=80mm, height=70mm]{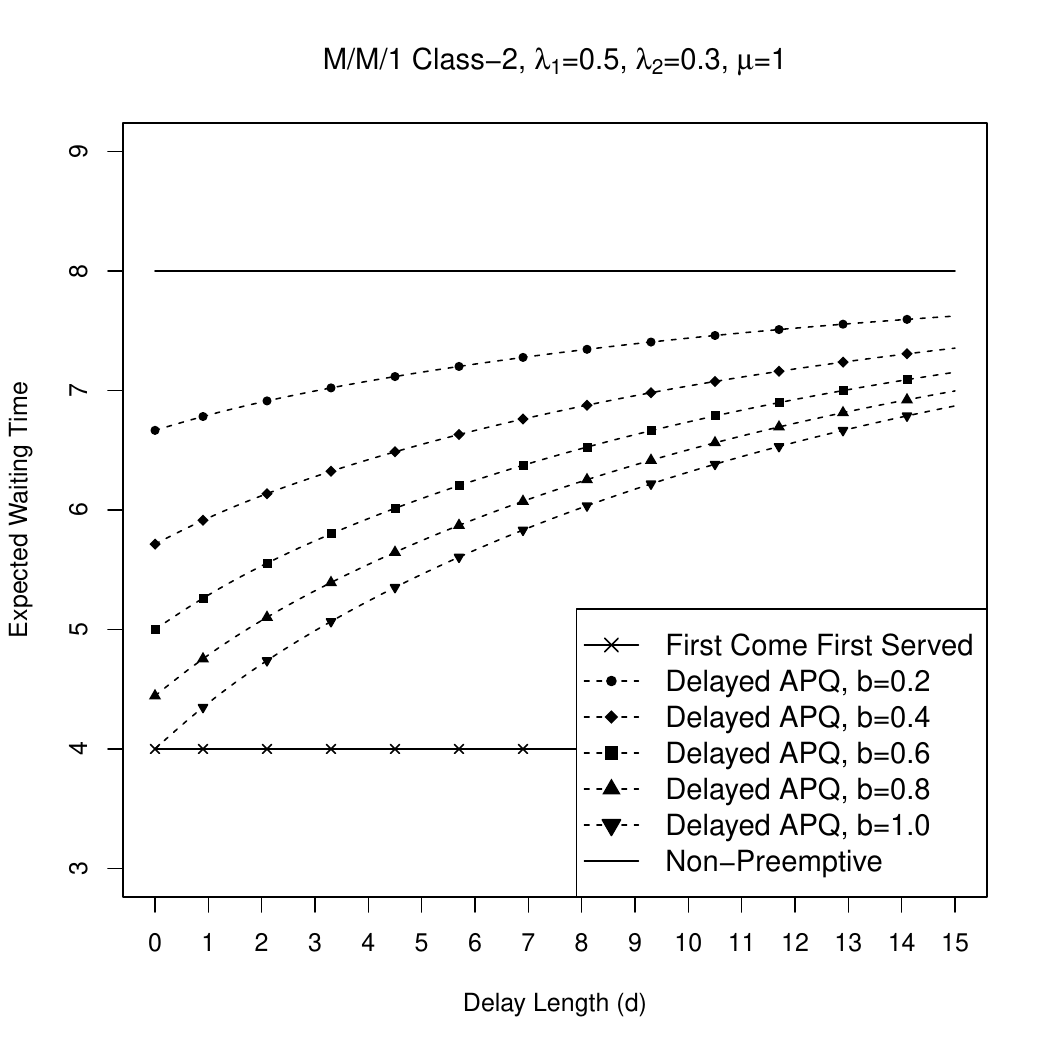}
  }
  \caption{The effect of delay length $d$ on expected waiting time for the M/M/1 \dapqname{} when $\rho = 0.8$.}
  \label{fig:mm1_dvals}
\end{figure}
\begin{figure}[htb]
  \centering
  \subfloat{
    \includegraphics[width=80mm, height=70mm]{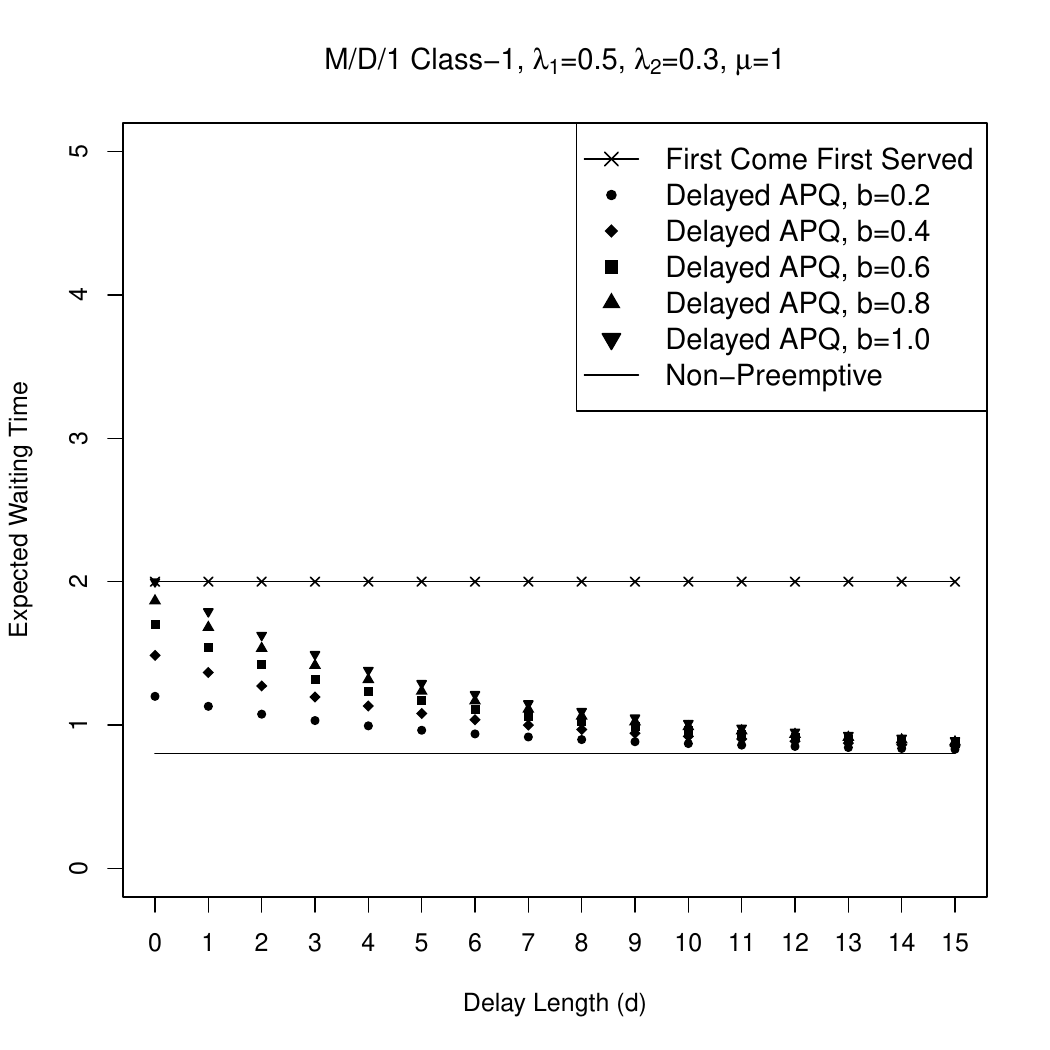}
  }
  \subfloat{
    \includegraphics[width=80mm, height=70mm]{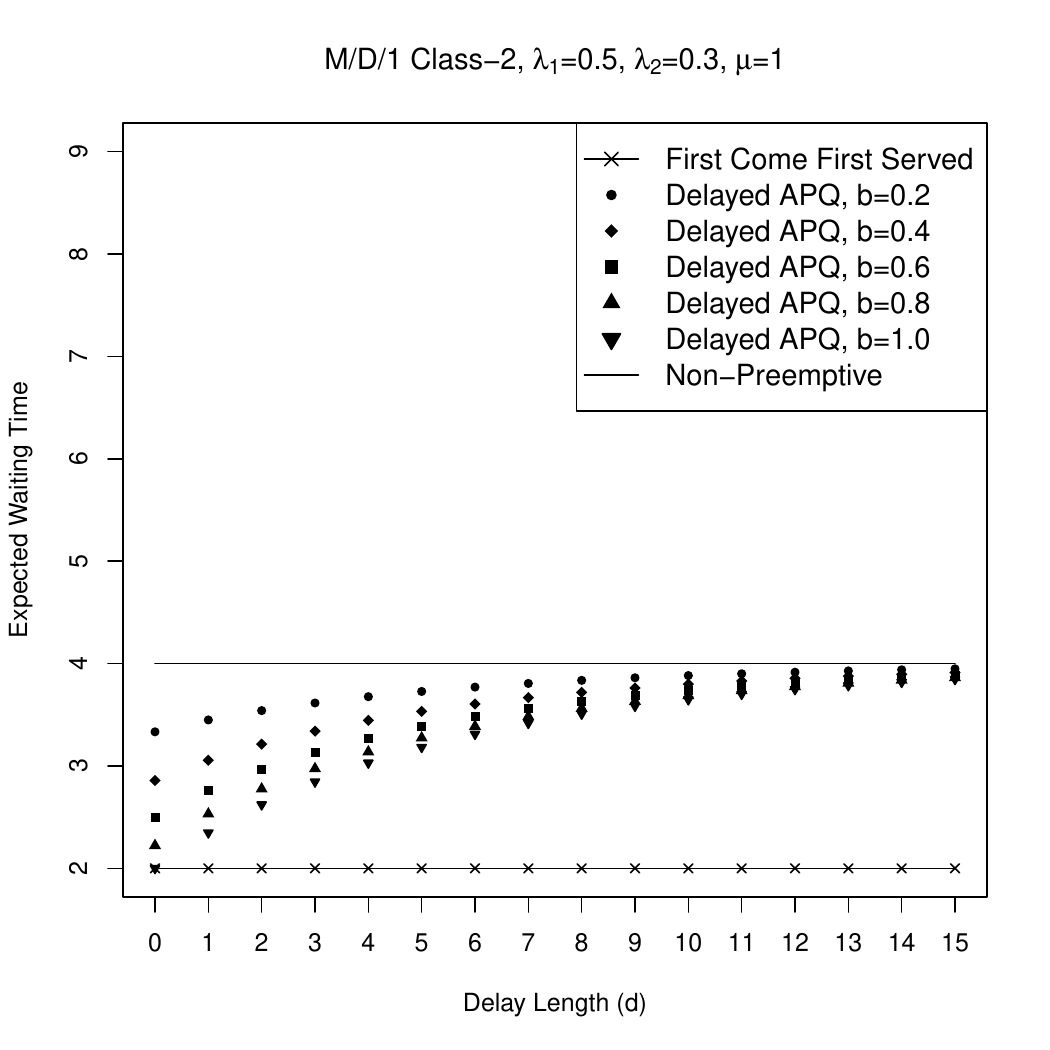}
  }
  \caption{The effect of delay length $d$ on expected waiting time for the M/D/1 \dapqname{} when $\rho = 0.8$.}
  \label{fig:md1_dvals}
\end{figure}}

\infor{
\begin{figure}[htb]
  \centering
  \subfloat{
    \includegraphics[width=70mm, height=55mm]{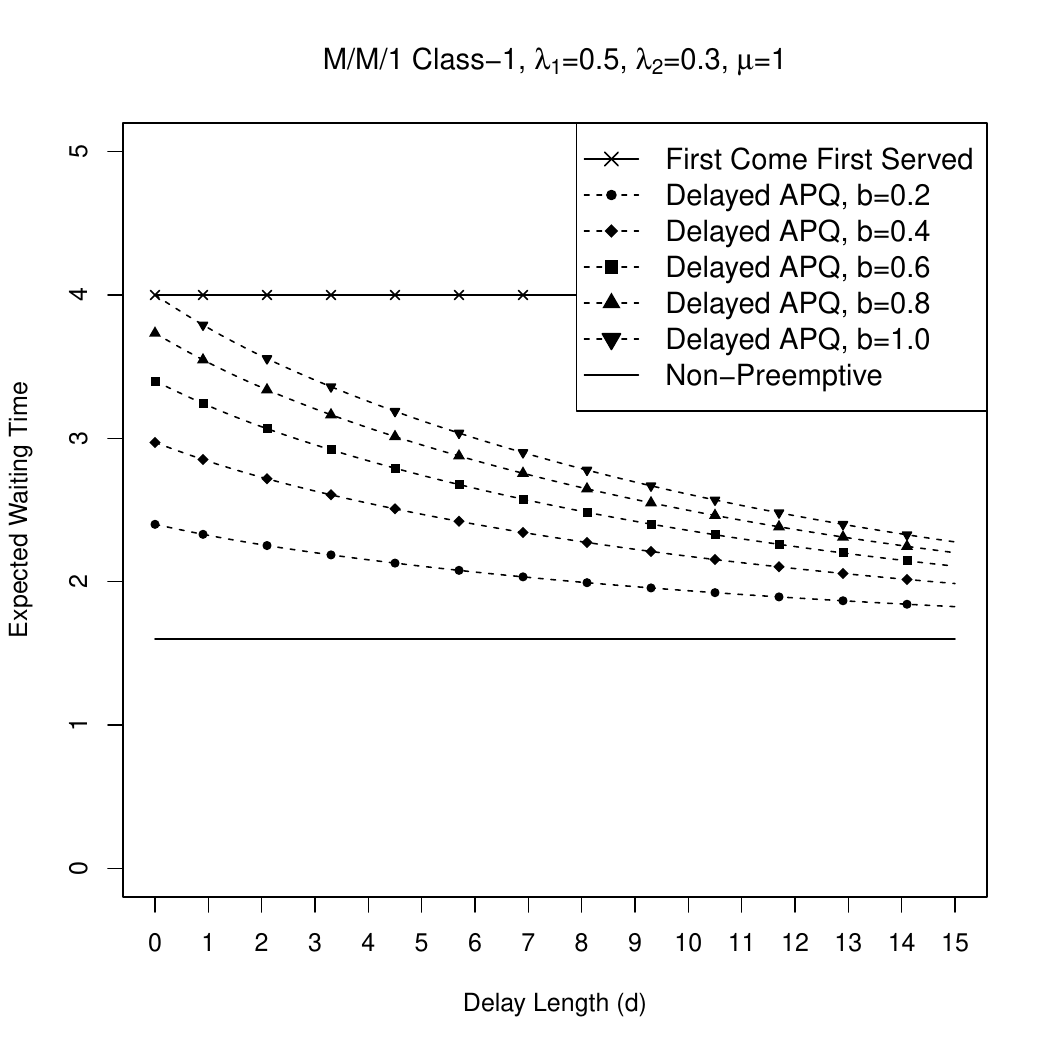}
  }
  \subfloat{
    \includegraphics[width=70mm, height=55mm]{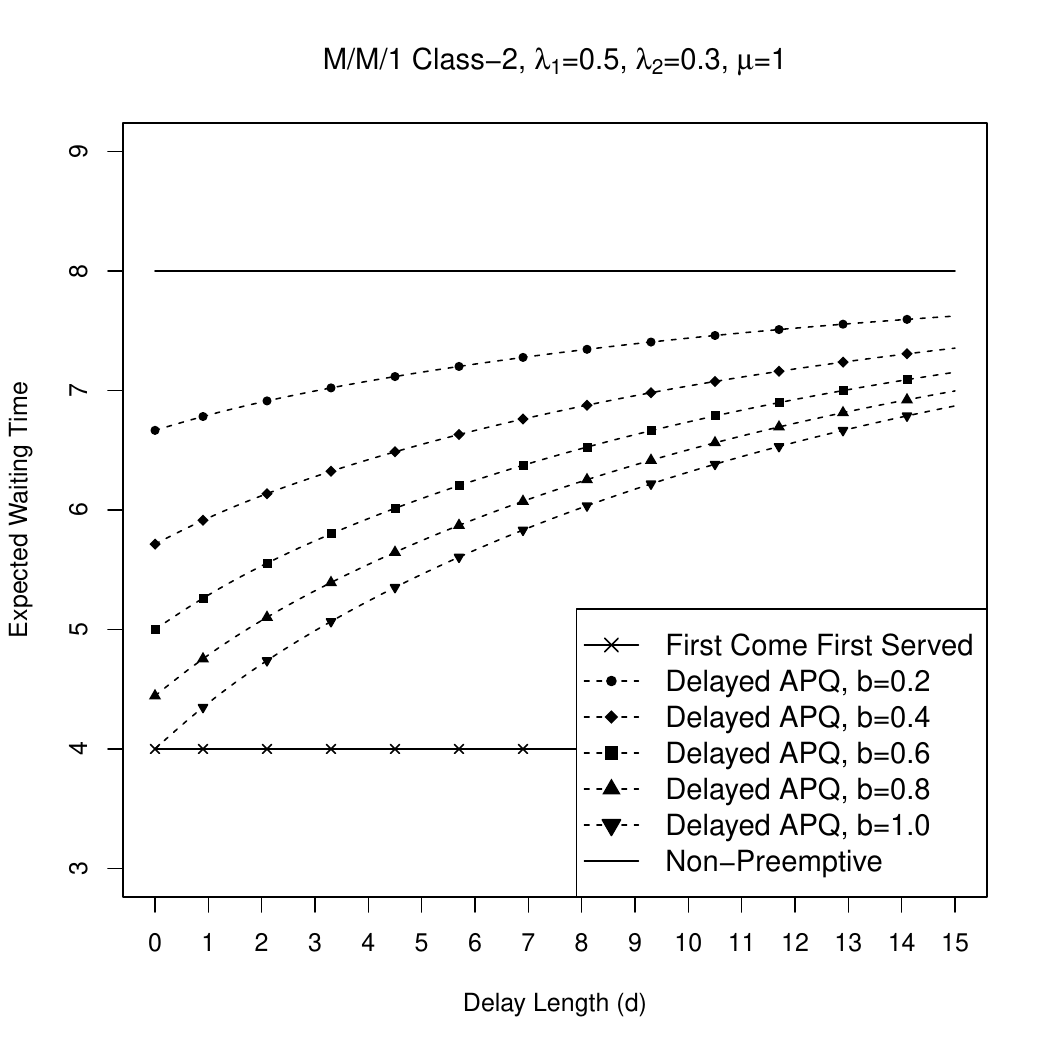}
  }
  \caption{The effect of delay length $d$ on expected waiting time for the M/M/1 \dapqname{} when $\rho = 0.8$.}
  \label{fig:mm1_dvals}
\end{figure}
\begin{figure}[htb]
  \centering
  \subfloat{
    \includegraphics[width=70mm, height=55mm]{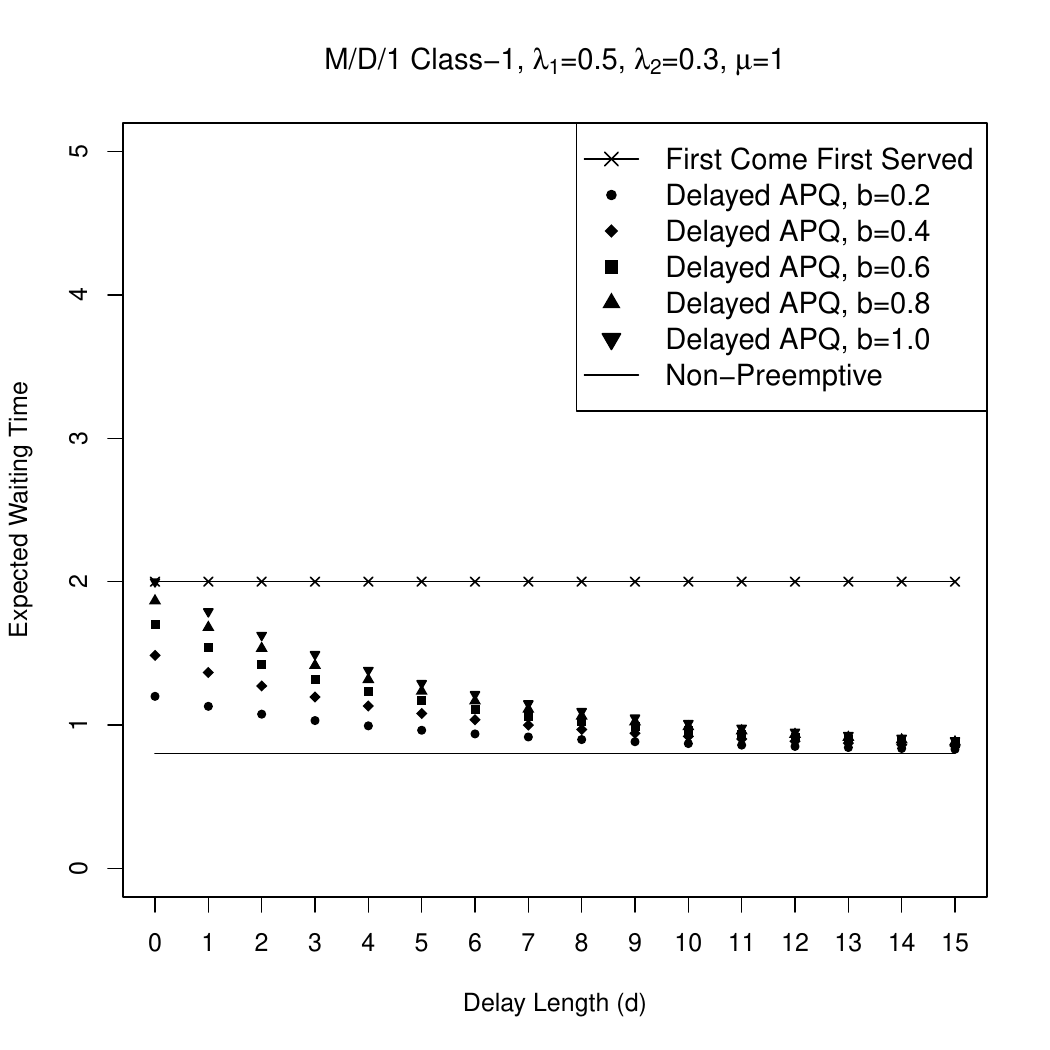}
  }
  \subfloat{
    \includegraphics[width=70mm, height=55mm]{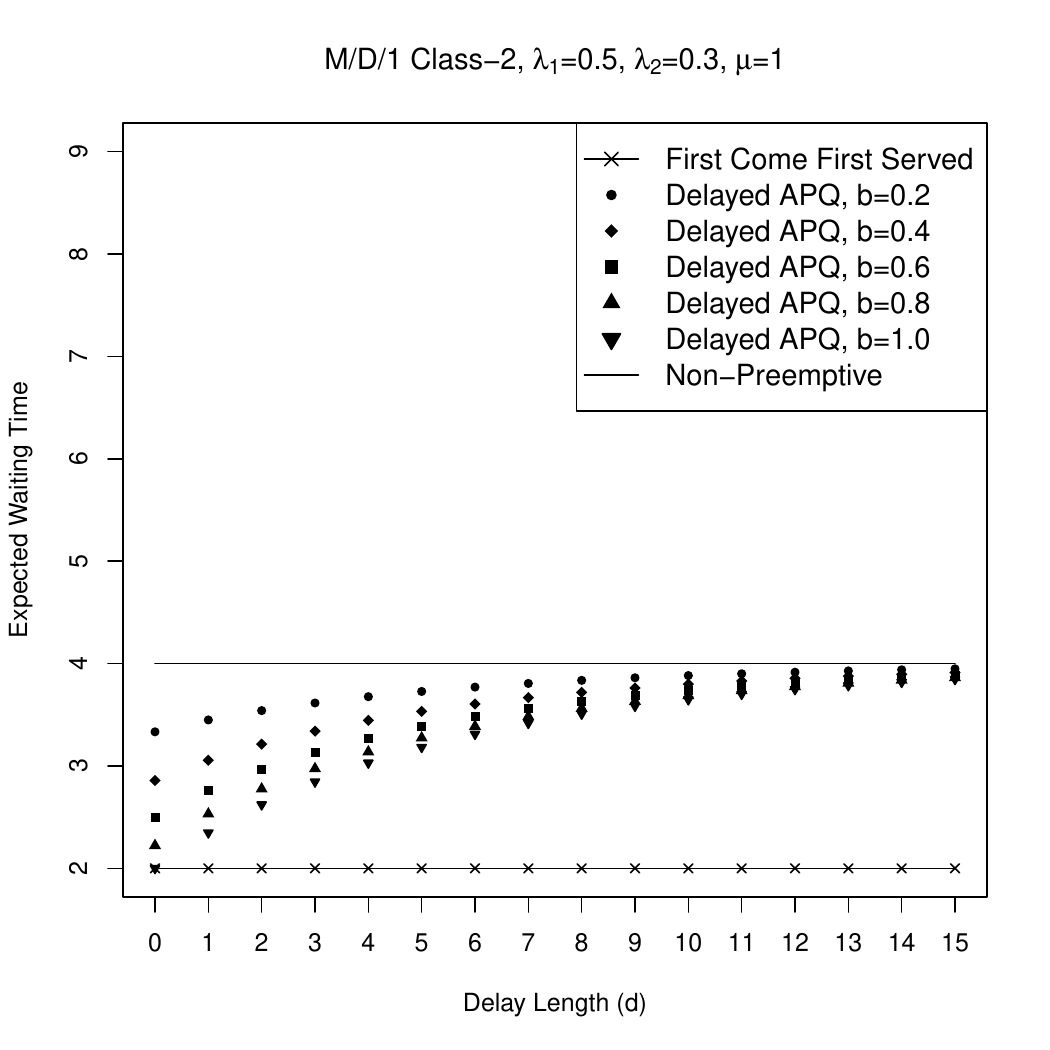}
  }
  \caption{The effect of delay length $d$ on expected waiting time for the M/D/1 \dapqname{} when $\rho = 0.8$.}
  \label{fig:md1_dvals}
\end{figure}}

\section{Optimizing parameters for health care KPIs}\label{sec:optimizing}

In this section, we use our algorithm for the \highname{} expected waiting time to analyze the optimal choice of parameters in a \dapqname{} to meet certain health care KPIs.
In particular, we are interested in the same KPIs studied in \citet{sharif16probability}, \citet{li19kpis}, and \citet{mojalal19dapq}. Using the Canadian Triage and Acuity Scale (CTAS), these papers define the low- and \highname{} customers within an emergency department after excluding the patients who must always be seen immediately and those who have very minor ailments (and make up a negligible proportion of emergency department patients). Then, the prescribed KPIs by \citet{bullard17ctas}, which are unchanged from \citet{bullard08ctas}, are for CTAS-3 (\classname{}-1) patients to wait less than 30 minutes, 90\% of the time, and for CTAS-4 (\classname{}-2) patients to wait less than 60 minutes, 85\% of the time. 
Note that these KPIs correspond to sample proportions since in practice they are evaluated using only observed data, but we study them in the infinite data limit, which corresponds to the actual probabilities under the stationary queueing system.

Throughout this section, we will use $\waitDparams{d}{b}$ and $\cdfDparams{d}{b}$ to explicitly denote dependence on the queueing parameters of both the waiting time and corresponding CDF.  
\citet{mojalal19dapq} incorporate the CTAS KPI by optimizing over the accumulation rate given the desired delay level. That is, given a delay level $d$, a waiting time target $w$, and a compliance probability $p$, they solve for
\[\label{eqn:mojalal-objective}
  b^*(d) = \min \left\{b : b \in [0,1], \cdfDparams{d}{b}_2(w) \geq p \right\}.
\]
The KPI example they explicitly use is $w=4$ and $p = 0.8$, which in the context of a mean service length being 15 minutes \citep{dreyer09workload} corresponds to the smallest accumulation rate for a given delay level such that at least 80\% of CTAS-4 patients wait less than one hour. However, this approach requires one to fix the delay level, and it is unclear what the optimal way to do so is. Fortunately, the additional information of the \classname{}-1 expected waiting time allows us to extend this analysis by optimizing for both $d$ and $b$ together.

Our optimization objective, given $\lam_1, \lam_2 \in (0,1)$ (taking $\mu=1$ for simplicity, and supposing $\lam_1+\lam_2<1$ to ensure the queue is stable), is to choose $d$ and $b$ that minimize the \classname{}-1 expected waiting time subject to the \classname{}-2 constraint of a waiting time target and compliance probability. 
Specifically, we aim to find
\[\label{eqn:main-objective}
  (d^*, b^*) = \argmin \{\EE \waitDparams{d}{b}_1 : d \in [0,\infty], b \in [0,1], \cdfDparams{d}{b}_2(w) \geq p \}.
\]

To do so, we first identify which pairs $(\lam_1, \lam_2)$ have a nontrivial solution to \cref{eqn:main-objective}.
In Figure~9 of \citet{mojalal19dapq}, the authors observe that the range of $d$ values with $b^*(d) < 1$ is quite small for their KPI at various occupancy levels. In \cref{fig:mm1_feasible_bound}, we complete this observation by identifying all values of $\lam_1$ and $\lam_2$ where the optimal pair $(d^*, b^*)$ exists and does not correspond to $d^*=\infty$ or $b^*=0$ for various KPI parameters. 
We refer to this set of values for $(\lam_1, \lam_2)$ as the \emph{feasible region}.
This simplifies the problem by reducing the number of optimizations we need to perform, and demonstrates the restrictive nature of the CTAS KPIs, since most real emergency departments operate at high total levels of occupancy.
 
To find the feasible region, we use two observations.
First, recall that the \npqname{} regime leads to the lowest \classname{}-1 waiting times, so if the constraint $\cdfD_2(w) \geq p$ \emph{can} be achieved by the \npqname{} then there is no further optimization to be done.
Second, the \fcfsname{} regime uniformly results in the lowest \classname{}-2 waiting times, so if the constraint $\cdfD_2(w) \geq p$ \emph{cannot} be achieved by the \fcfsname{} then the occupancy is simply too high for the KPI to be met.
To visualize this, in \cref{fig:mm1_feasible_bound} we plot the lower boundary of when the \npqname{} is strong enough and the upper boundary of when the \fcfsname{} is too weak to achieve the KPI for \classname{}-2 customers for both one hour (left panel) and half hour (right panel) waiting time targets with various compliance probabilities. 
The interpretation of these plots is that for each KPI probability level, the $(\lam_1, \lam_2)$ pairs to the left of the feasible region have sufficiently small occupancy such that the KPI can be achieved by \classname{}-2 customers even under the penalizing \npqname{}, while to the right of the feasible region it is impossible to meet the KPI. Thus, the $(\lam_1, \lam_2)$ pairs that require further optimization are only those that fall within this feasible region.

\preprint{
\begin{figure}[htb]
  \centering
  \subfloat{
    \includegraphics[width=80mm, height=70mm]{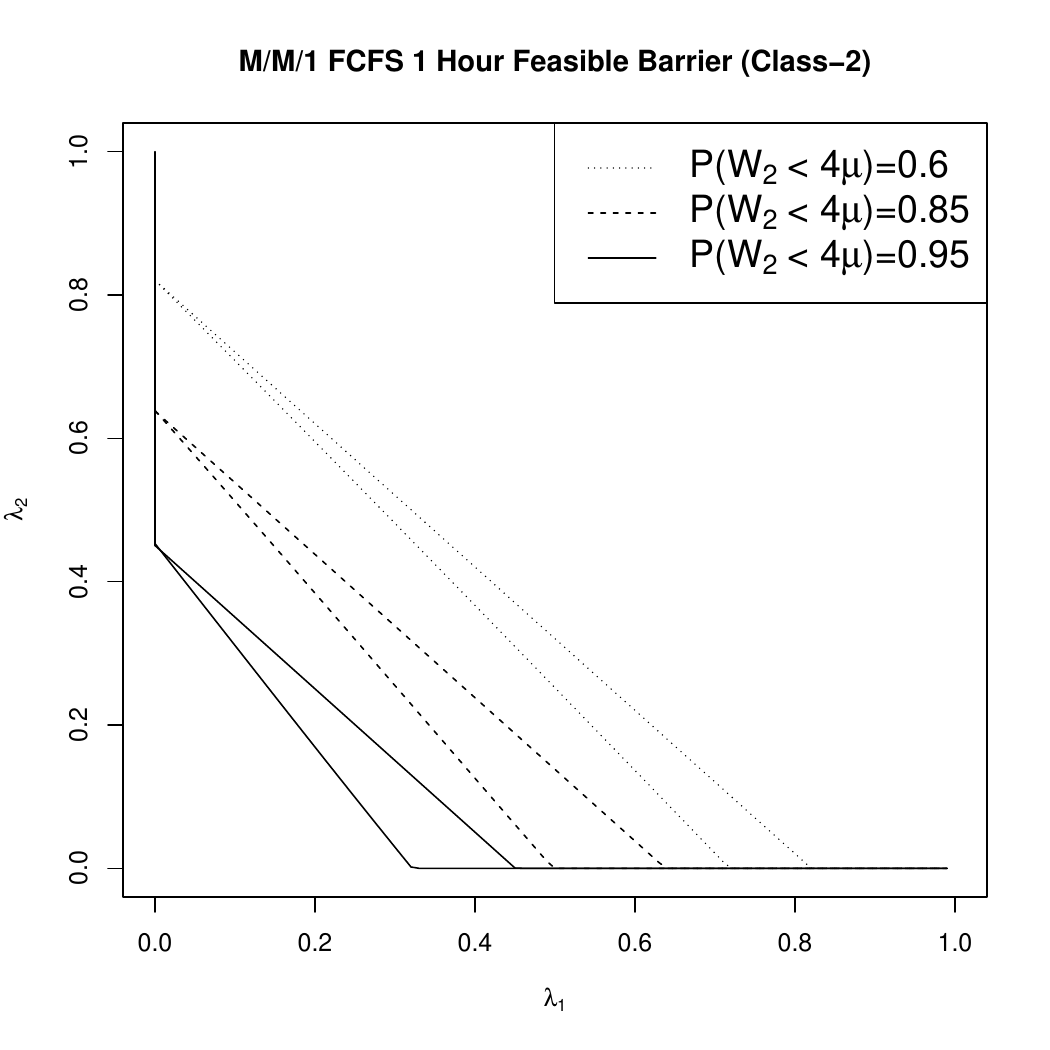}
  }
  \subfloat{
    \includegraphics[width=80mm, height=70mm]{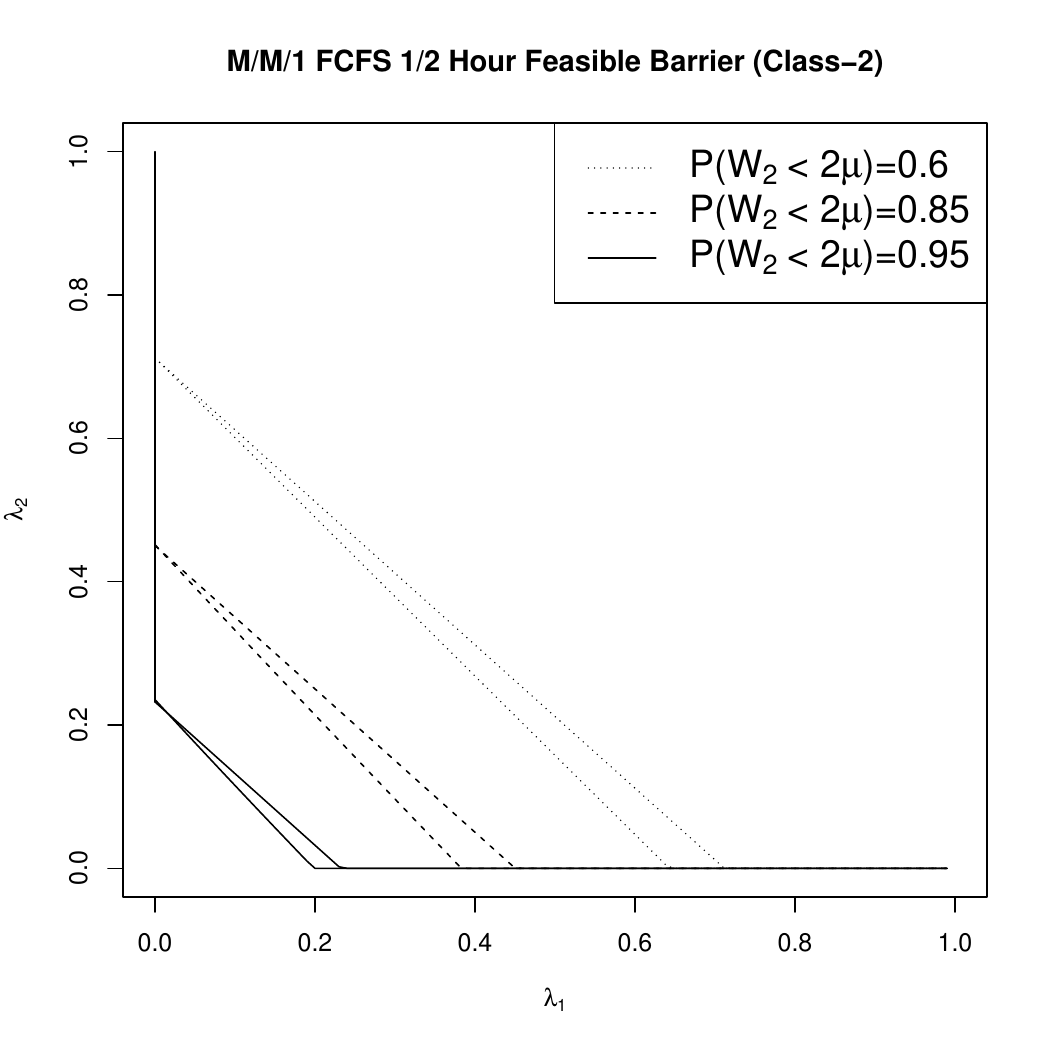}
  }
  \caption{Upper and lower boundaries for $(\lam_1, \lam_2)$ pairs that meet KPI probability for \classname{}-2 waiting time under one hour and half hour and require optimization of $d$.}
  \label{fig:mm1_feasible_bound}
\end{figure}}
\infor{
\begin{figure}[htb]
  \centering
  \subfloat{
    \includegraphics[width=70mm, height=55mm]{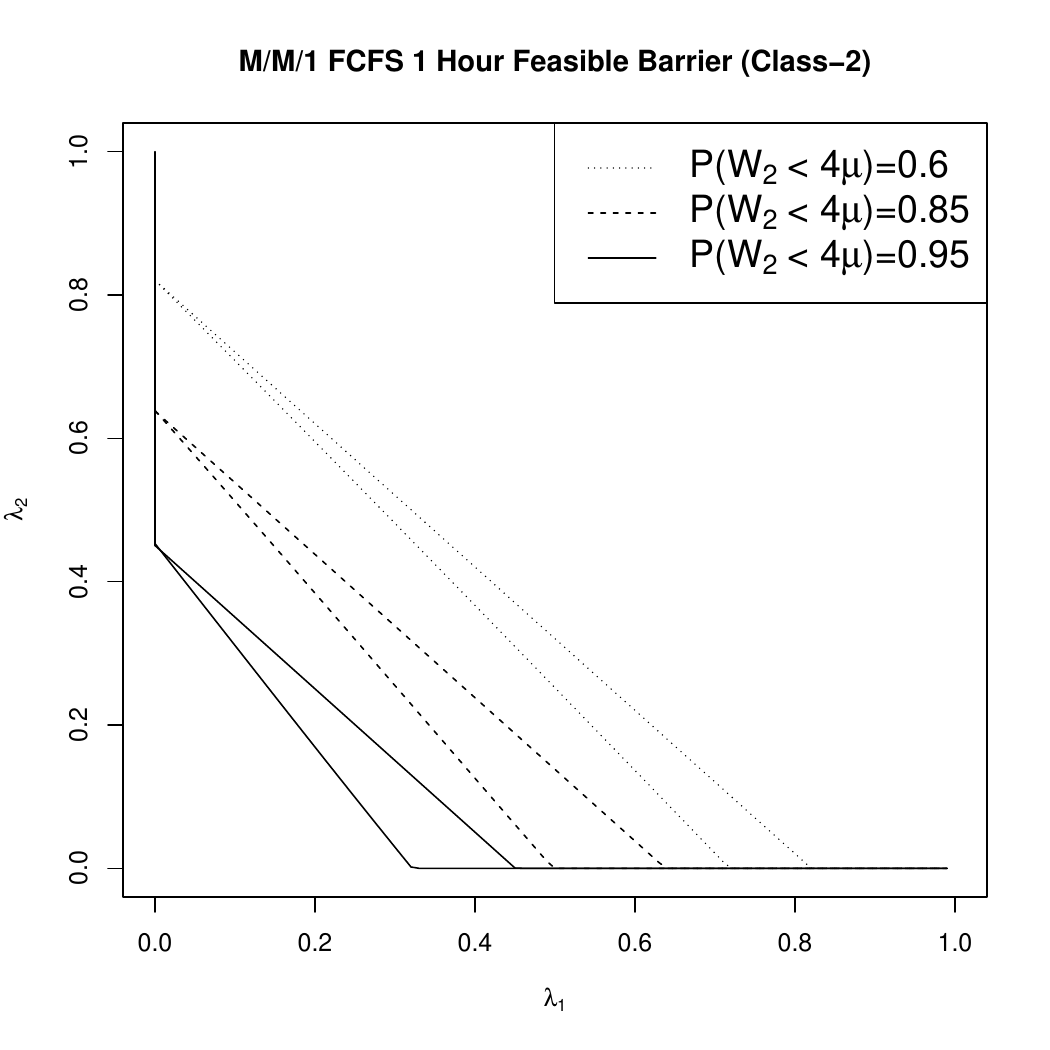}
  }
  \subfloat{
    \includegraphics[width=70mm, height=55mm]{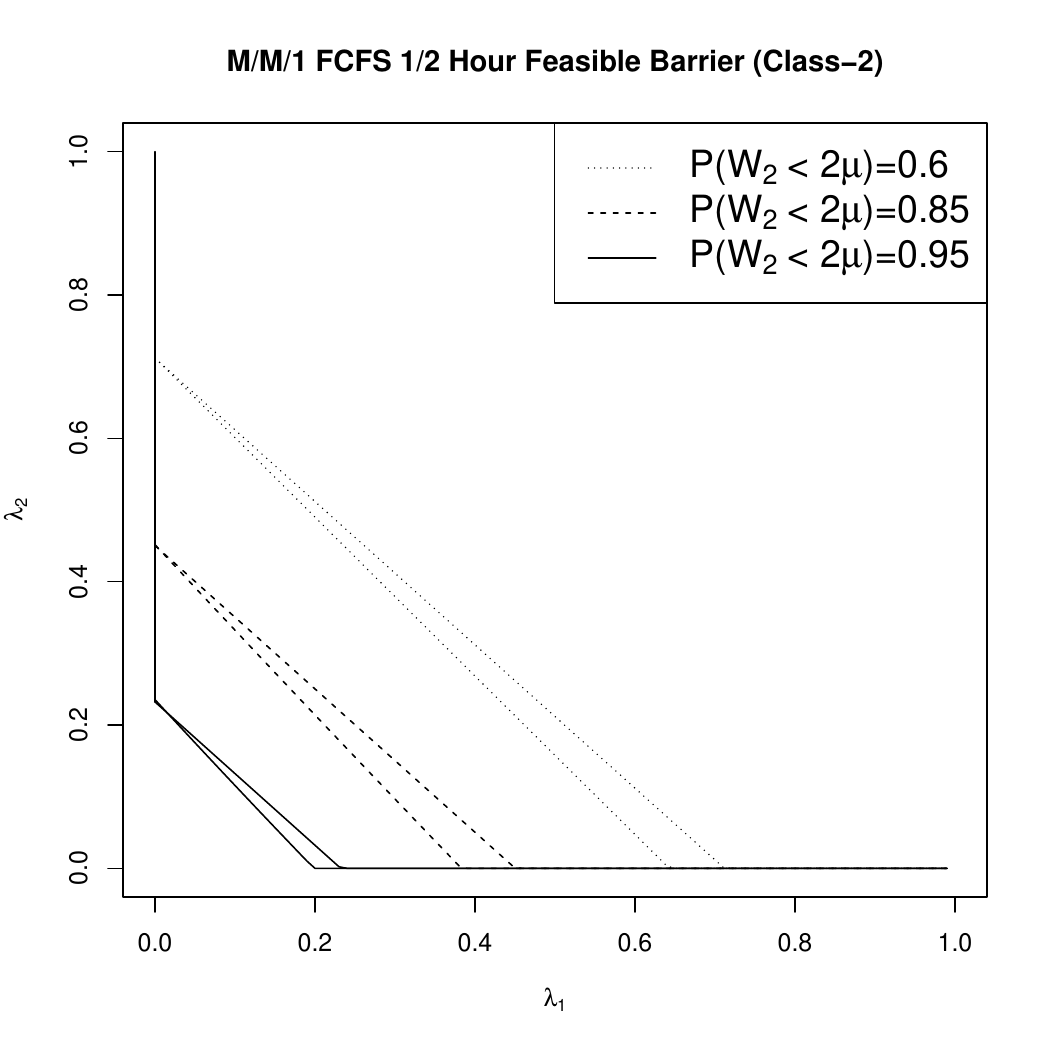}
  }
  \caption{Upper and lower boundaries for $(\lam_1, \lam_2)$ pairs that meet KPI probability for \classname{}-2 waiting time under one hour and half hour and require optimization of $d$.}
  \label{fig:mm1_feasible_bound}
\end{figure}}

Now, for each $(\lam_1, \lam_2)$ pair within the feasible region, there are multiple $(d,b)$ pairs that can be chosen to ensure that \classname{}-2 customers meet the required KPI. To find $(d^*, b^*)$ from these, we observe that for each fixed $d$ value, $\EE \waitDparams{d}{b}_1$ is monotonically increasing with $b$. 
This observation follows by decomposing expected waiting time into the expected number of customers that will be served ahead of a tagged customer multiplied by the expected service length of each of these customers. 
Since $b$ does not affect the number of \classname{}-1 customers served ahead of a tagged \classname{}-1 customer, and an increase in $b$ increases the amount of priority each \classname{}-2 customer has (and hence the number that will be served first), this relationship holds for any arrival and service distributions.
The implication of this observation is that $(d^*, b^*) = (d^*,b^*(d^*))$, and hence \cref{eqn:main-objective} can be reduced to two easier, univariate optimizations.

Next, suppose $(d,b)$ is such that $\cdfDparams{d}{b}_2(w) > p$. Then one can always either increase $d$ or decrease $b$ in order to simultaneously decrease $\EE \waitDparams{d}{b}_1$ and $\cdfDparams{d}{b}_2(w)$.
Thus, the constraint will always be active; that is, $\cdfDparams{d^*}{b^*}_2(w) = p$.
Combined with the argument of the previous paragraph, we have reduced the problem to finding the optimal $d$ out of those for which $\cdfDparams{d}{b^*(d)}_2(w) = p$.
To determine which $d$ value is optimal to choose, we turn to our contribution of the expected \classname{}-1 waiting time, performing a univariate optimization over these $d$ values to determine which $d$ minimizes $\EE \waitDparams{d}{b^*(d)}_1$.
Note that there is no guarantee $\EE \waitDparams{d}{b^*(d)}_1$ is convex as a function of $d$. However, since $(\lam_1, \lam_2)$ is in the feasible region, there is a maximal value of $d$ for which it is possible to achieve $\cdfDparams{d}{b^*(d)}_2(w) \geq p$, and hence we can perform a global univariate optimization.
We formalize the actual computation steps for this procedure in Algorithm~\ref{algo:optimization}.

\begin{algorithm}
  \SetAlgoLined
  \SetKwInput{KwInput}{Inputs}
  \SetKwInput{KwReturn}{Return}
  \KwInput{\\
  $\bullet$ $(\lam_1, \lam_2)$ in the feasible region for KPI determined by $w$ and $p$\;
  $\bullet$ a function $\texttt{argmin} : (f,(a,b))\in(\R\to\R)\times\R^2\to \R$ that returns an $x \in [a,b]$ achieving the global minimum of $f(x)$ on the interval $[a,b]$\;
  $\bullet$ a function $\texttt{root} : (f,(a,b))\in(\R\to\R)\times\R^2\to \R$ that returns an $x \in [a,b]$ with $f(x) = 0$ on the interval $[a,b]$ whenever one exists.}
  \KwResult{Optimal queueing parameters $(d^*,b^*)$ that solve \cref{eqn:main-objective}.}
  \textbackslash* 
  \emph{Find the largest $d$ value such that the KPI can be achieved by some $b<1$} 
  \hfill *\textbackslash \\
  \textbackslash* 
  \emph{In practice, $d_{\max}<5$, so one need not start with \texttt{MAX_FLOAT} for the range} 
  \hfill *\textbackslash \\
  1. Find $d_{\max} = \texttt{root}(d \mapsto \cdfDparams{d}{b=1}_2(w) - p, (0,\texttt{MAX_FLOAT}))$.\\
  \textbackslash* 
  \emph{Define the function to obtain $b^*(d)$}
  \hfill *\textbackslash \\
  2. Define $[0,d_{\max}] \ni d \mapsto b^*(d) = \texttt{root}(b \mapsto \cdfDparams{d}{b}_2(w) - p, (0,1))$.\\
  \textbackslash* 
  \emph{Perform the optimization to solve for $d^*$ using the $b^*(d)$ function} 
  \hfill *\textbackslash \\
  3. Find $d^* = \texttt{argmin}(d \mapsto \EE \waitDparams{d}{b^*(d)}_1, (0,d_{\max}))$.\\
  \KwReturn{$(d^*, b^*(d^*))$}
\caption{Optimization to find $(d^*, b^*)$}
\label{algo:optimization}
\end{algorithm}

To illustrate our approach, we use \cref{fig:mm1_kpi_examples_bvals,fig:mm1_kpi_examples}, which focus on the middle triangle of the left panel of \cref{fig:mm1_feasible_bound}, corresponding to the KPI of $\PP(\waitD_2 < 4) \geq 0.85$. In both figures, each of the three panels correspond to different $(\lam_1,\lam_2)$ pairs that lie in the triangle, with the $x$-axis enumerating the values of $d$ for which $b^*(d)<1$ (that is, $d \in [0,d_{\max}]$ as defined in Algorithm~\ref{algo:optimization}). In \cref{fig:mm1_kpi_examples_bvals}, the $y$-axis plots $b^*(d)$, while in \cref{fig:mm1_kpi_examples}, the $y$-axis plots $\EE \waitDparams{d}{b^*(d)}_1$.
Concretely, Step~3 of Algorithm~\ref{algo:optimization} corresponds to finding the $d$ that is the $\argmin$ of the $y$-axis in \cref{fig:mm1_kpi_examples}.

What we find in each of the panels is quite surprising, since as $d$ increases, $b^*(d)$ increases so much that the net effect is to increase the \classname{}-1 expected waiting time. This suggests that for the purpose of meeting \classname{}-2 KPIs while optimizing \classname{}-1 expected waiting time, one should not prefer a \dapqname{} over a standard \apqname{}. Furthermore, moving left to right through the panels reveals that the detrimental effect of increasing the delay level on the \classname{}-1 expected waiting time becomes more severe as $\lam_2$ controls a high proportion of occupancy. Finally, we note that while these figures only address specific KPI examples and the M/M/1 case, further numerical investigation showed the same conclusions hold for M/D/1 and other KPI levels.

\preprint{
\begin{figure}[htb]
  \centering
  \subfloat{
    \includegraphics[width=52mm, height=55mm]{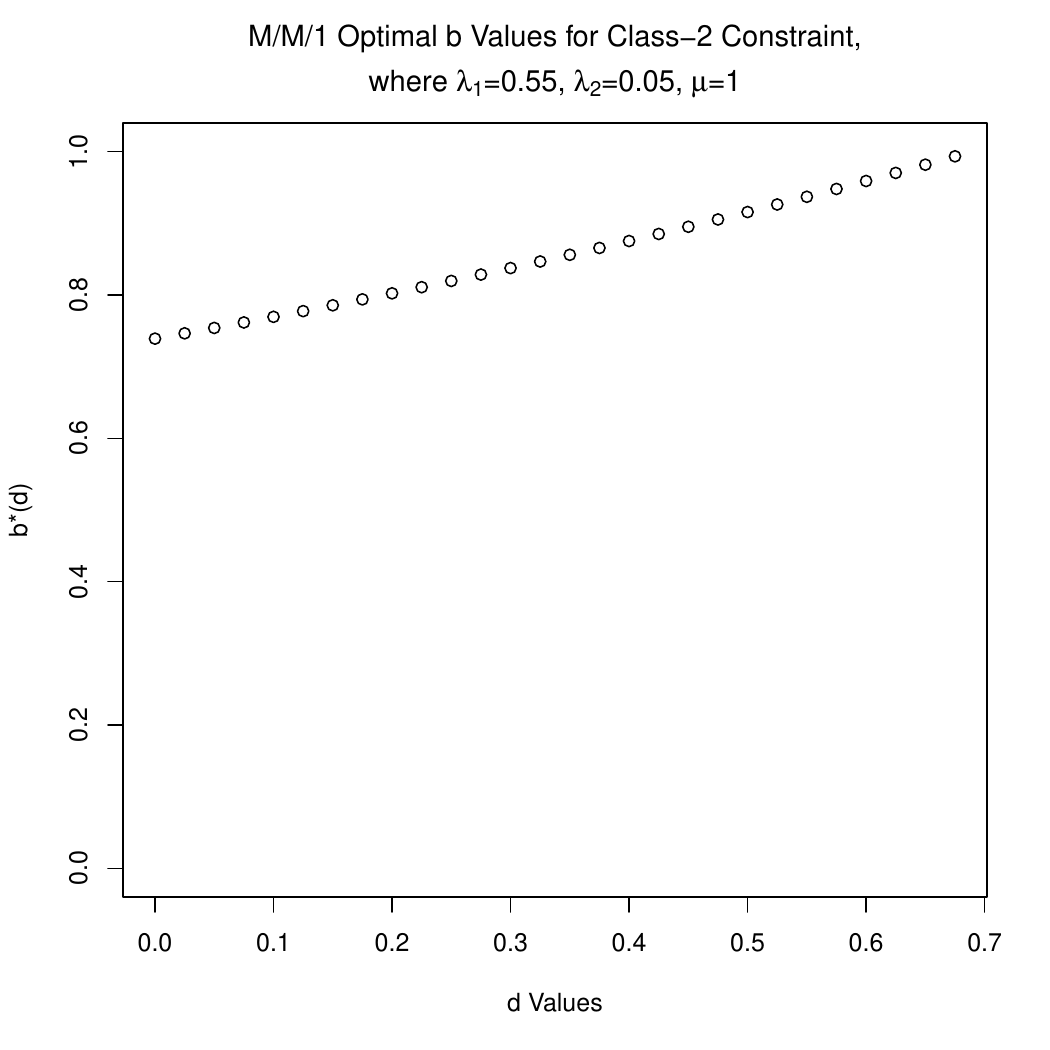}
  }
  \subfloat{
    \includegraphics[width=52mm, height=55mm]{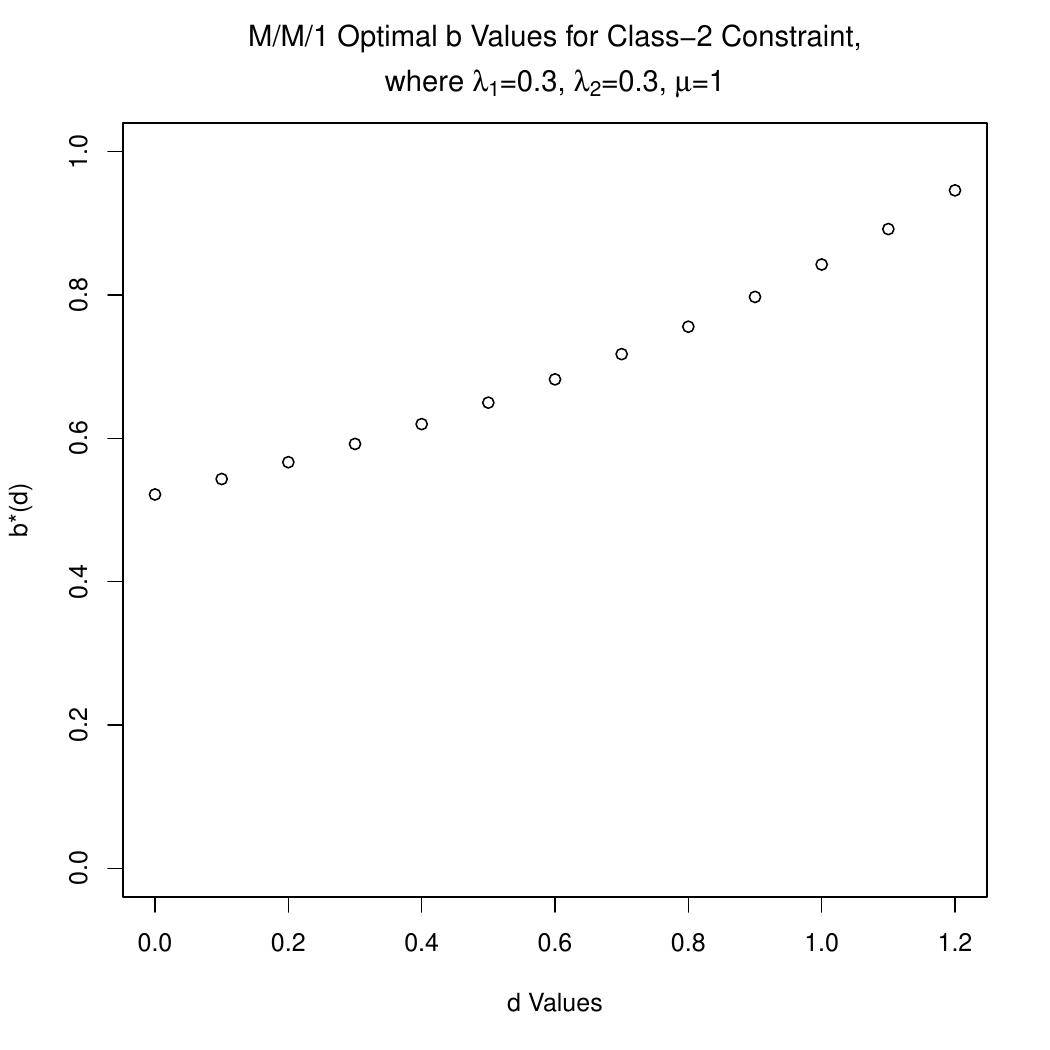}
  }
  \subfloat{
    \includegraphics[width=52mm, height=55mm]{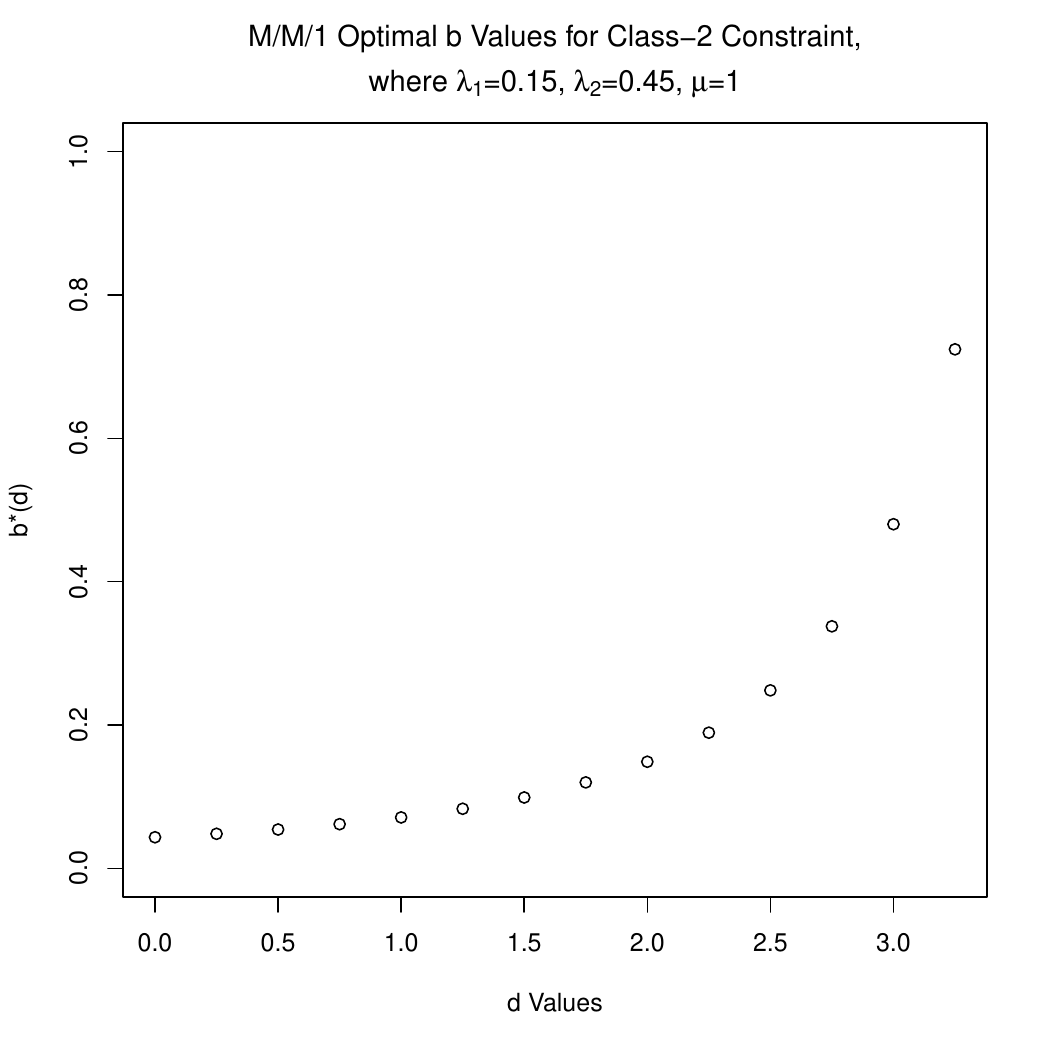}
  }
  \caption{Optimal $b^*(d)$ for delay level $d$ at various occupancy levels for the KPI $\PP(\waitD_2 < 4) \geq 0.85$.}
  \label{fig:mm1_kpi_examples_bvals}
\end{figure}
\begin{figure}[htb]
  \centering
  \subfloat{
    \includegraphics[width=52mm, height=55mm]{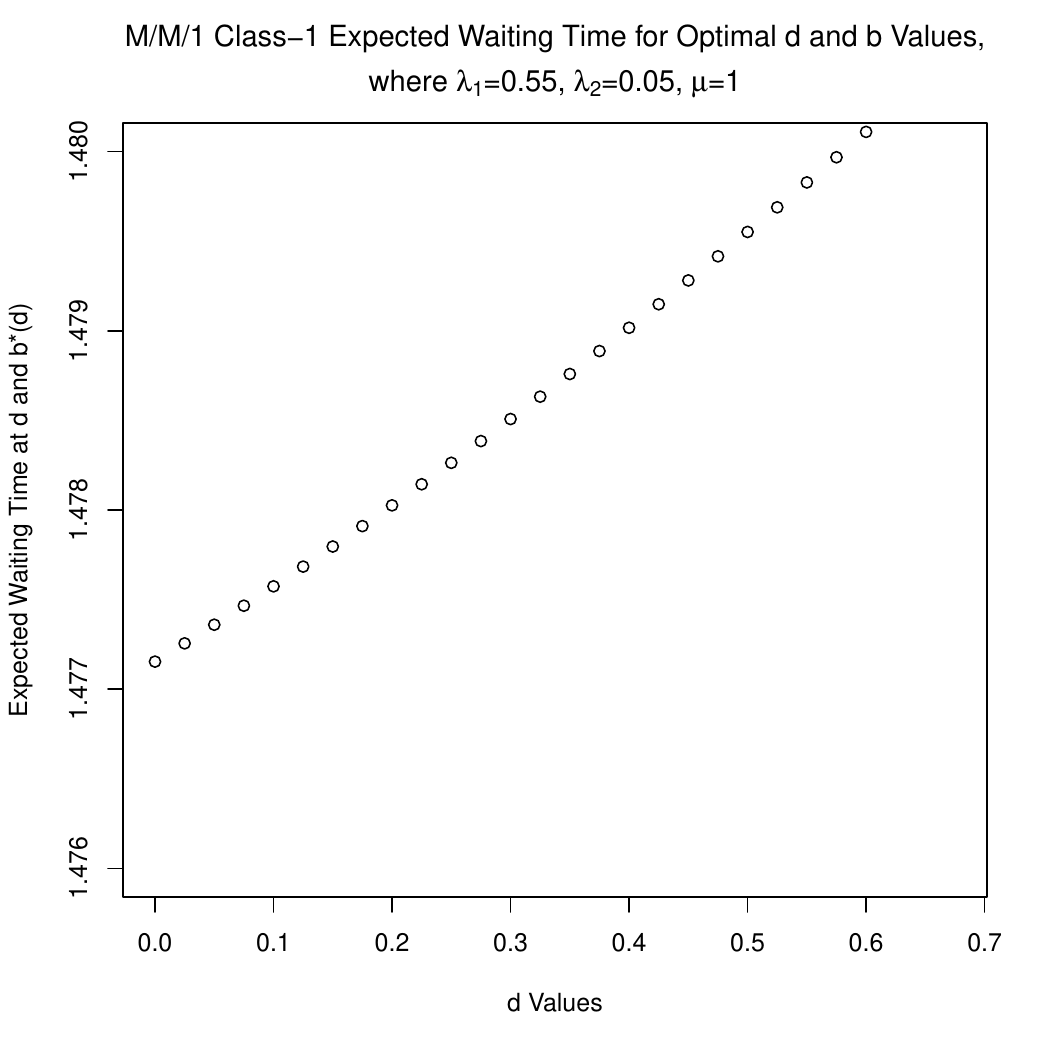}
  }
  \subfloat{
    \includegraphics[width=52mm, height=55mm]{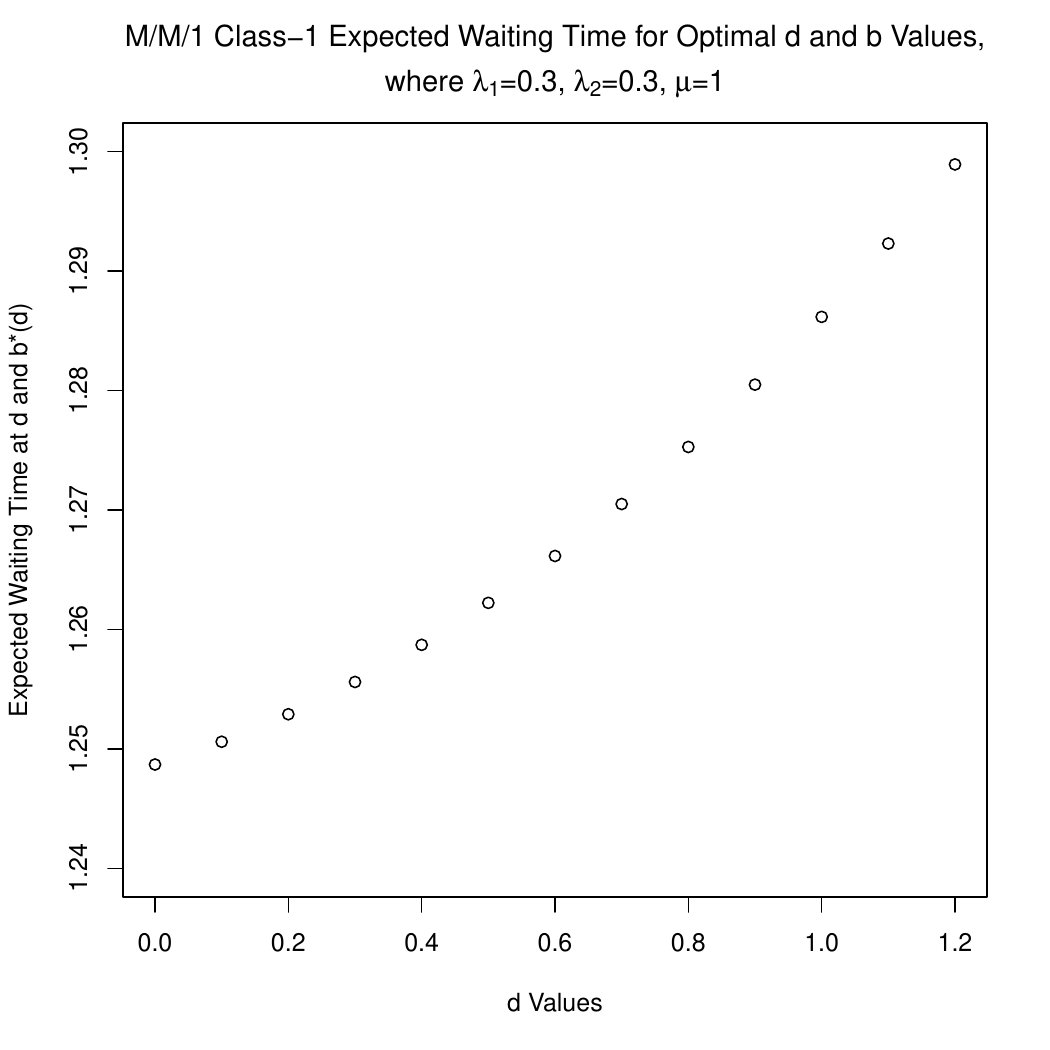}
  }
  \subfloat{
    \includegraphics[width=52mm, height=55mm]{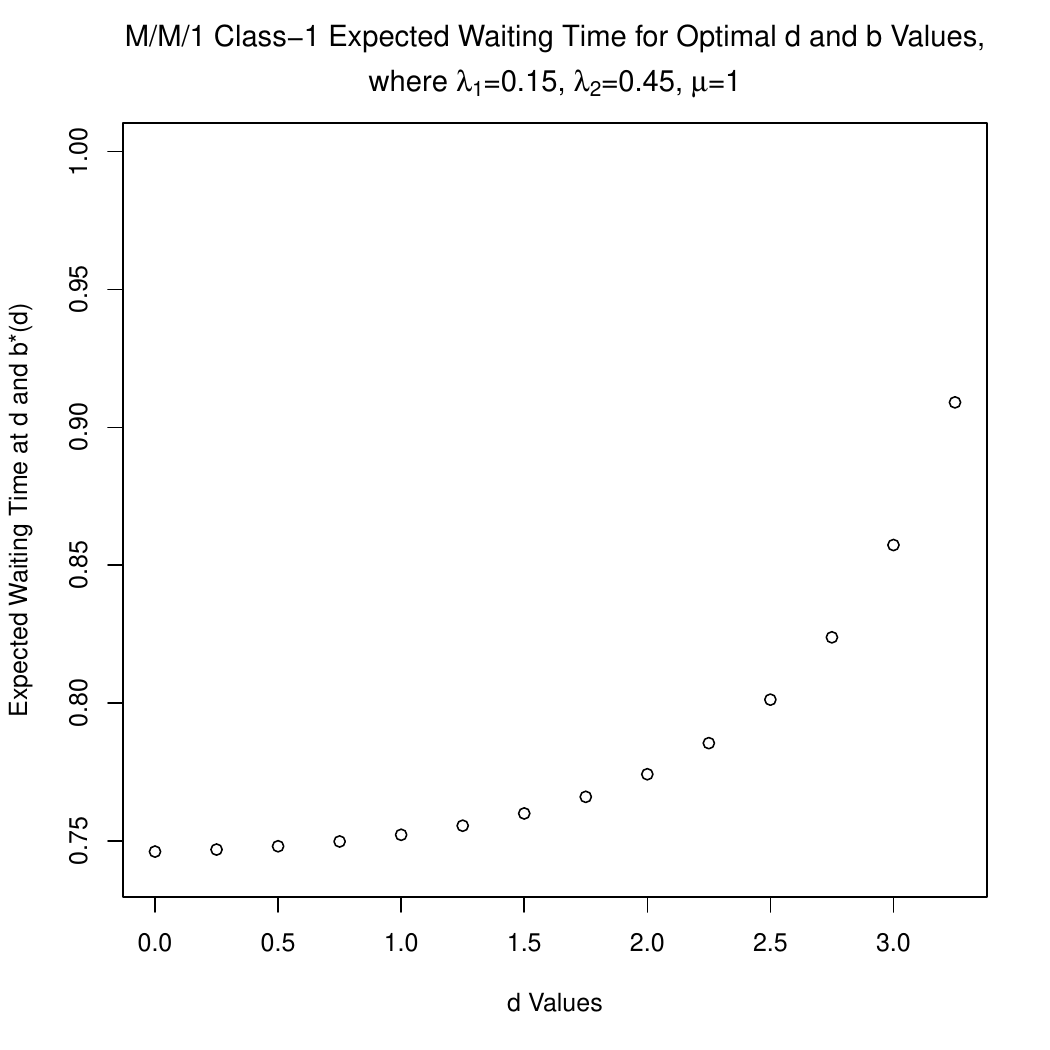}
  }
  \caption{$\EE \waitDparams{d}{b^*(d)}_1$ as a function of $d$ for the KPI $\PP(\waitD_2 < 4) \geq 0.85$.}
  \label{fig:mm1_kpi_examples}
\end{figure}}

\infor{
\begin{figure}[htb]
  \centering
  \subfloat{
    \includegraphics[width=45mm, height=45mm]{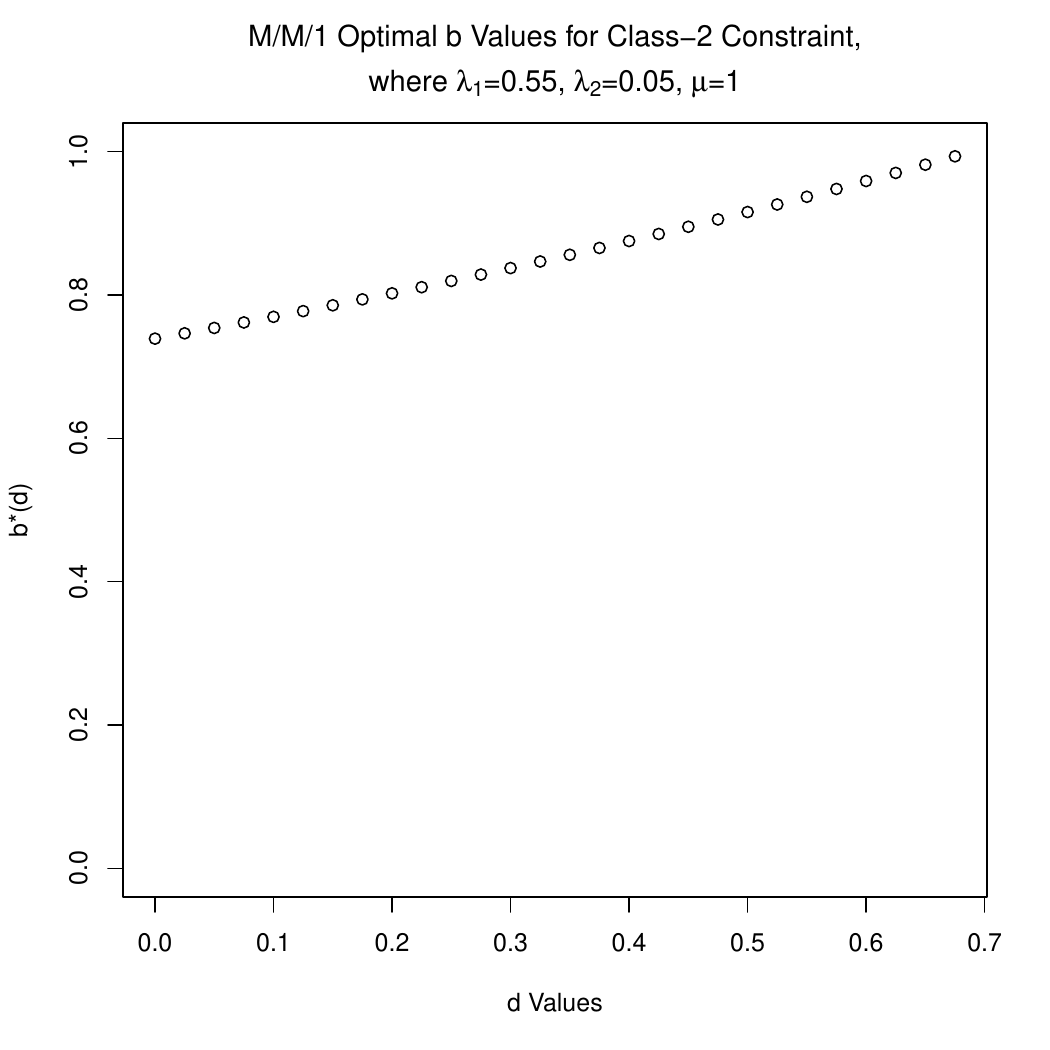}
  }
  \subfloat{
    \includegraphics[width=45mm, height=45mm]{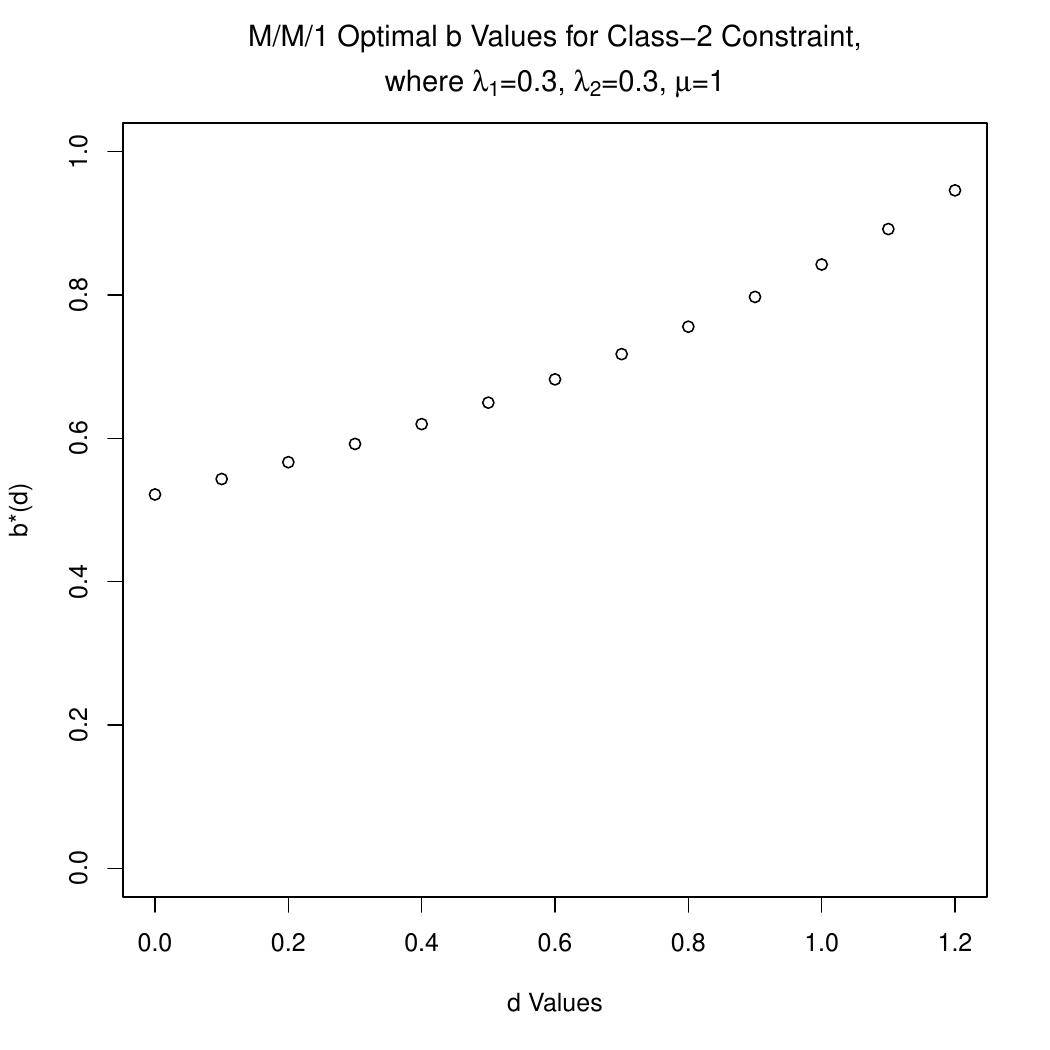}
  }
  \subfloat{
    \includegraphics[width=45mm, height=45mm]{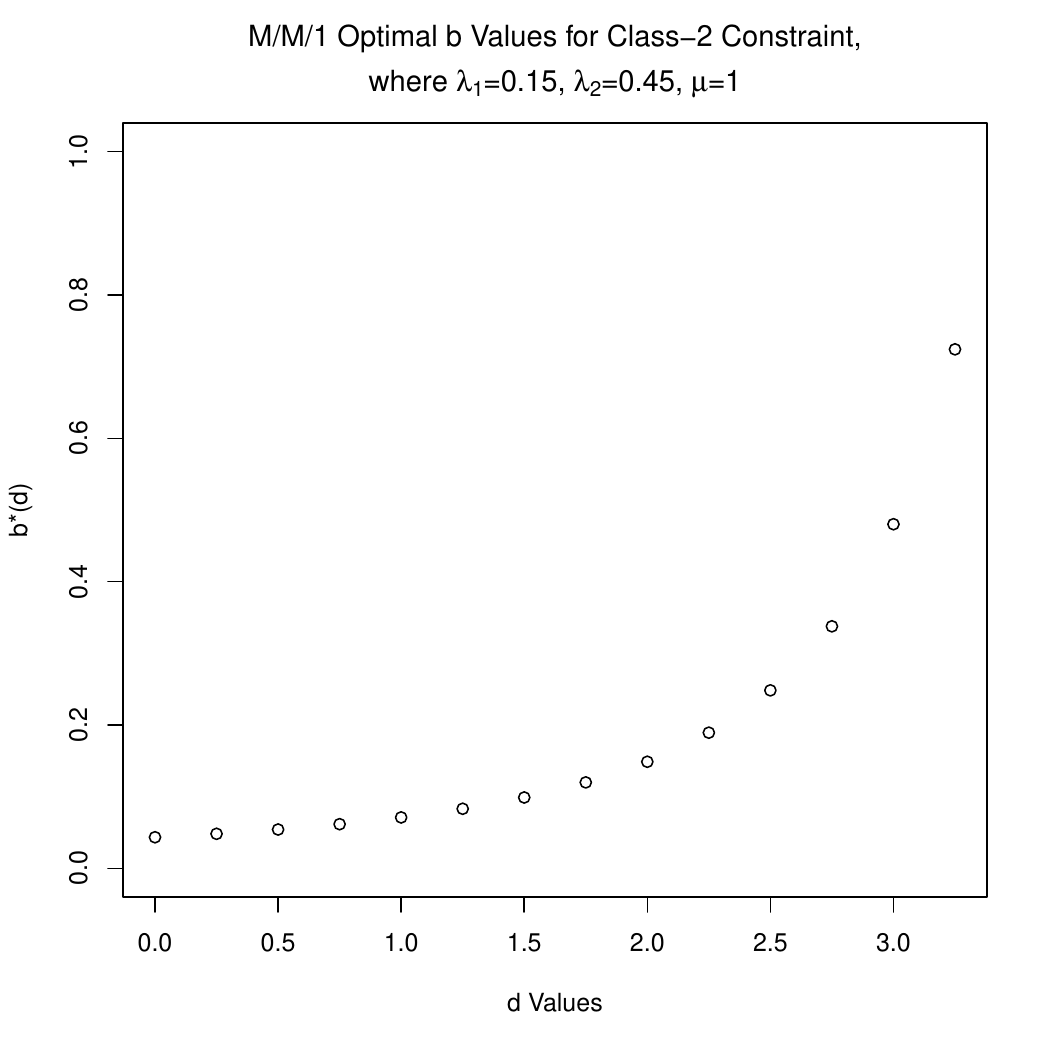}
  }
  \caption{Optimal $b^*(d)$ for delay level $d$ at various occupancy levels for the KPI $\PP(\waitD_2 < 4) \geq 0.85$.}
  \label{fig:mm1_kpi_examples_bvals}
\end{figure}
\begin{figure}[htb]
  \centering
  \subfloat{
    \includegraphics[width=45mm, height=45mm]{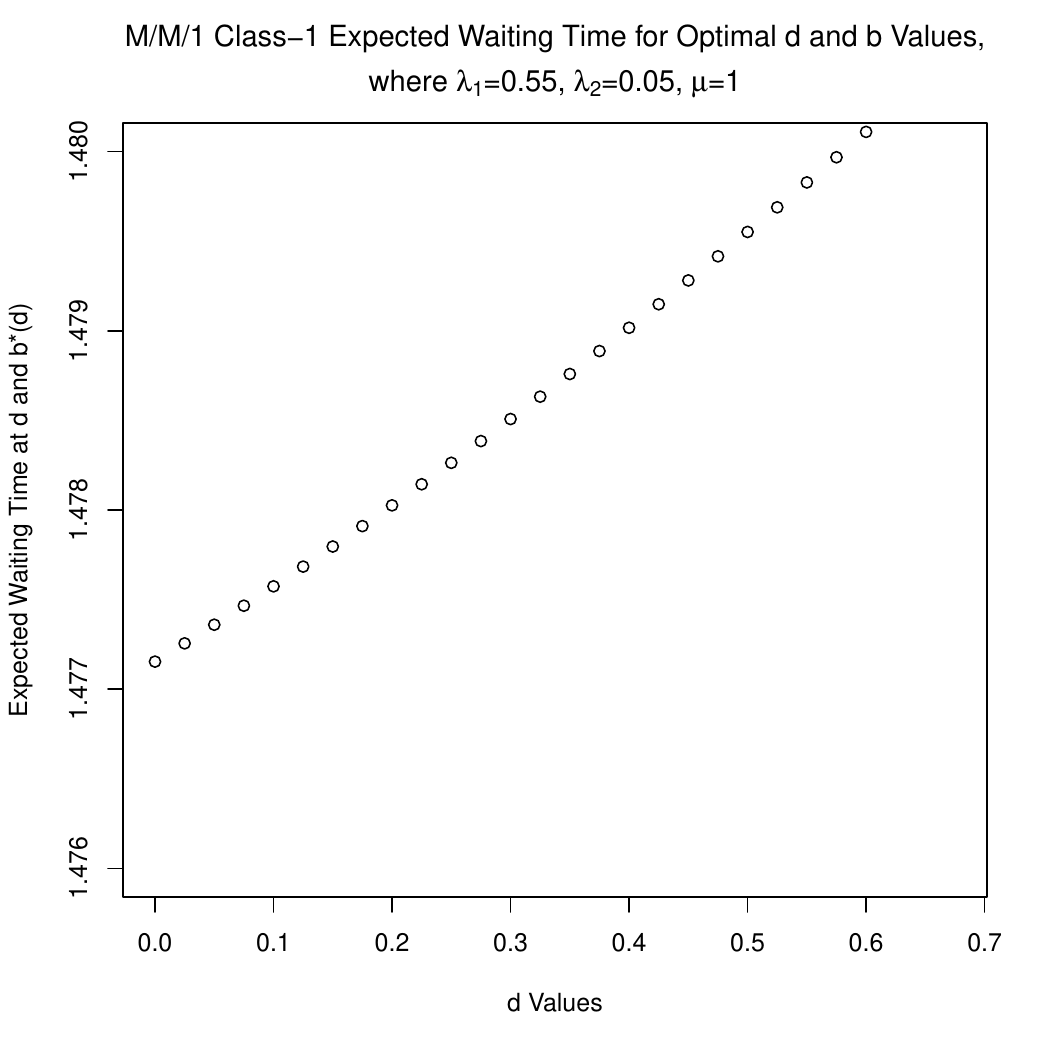}
  }
  \subfloat{
    \includegraphics[width=45mm, height=45mm]{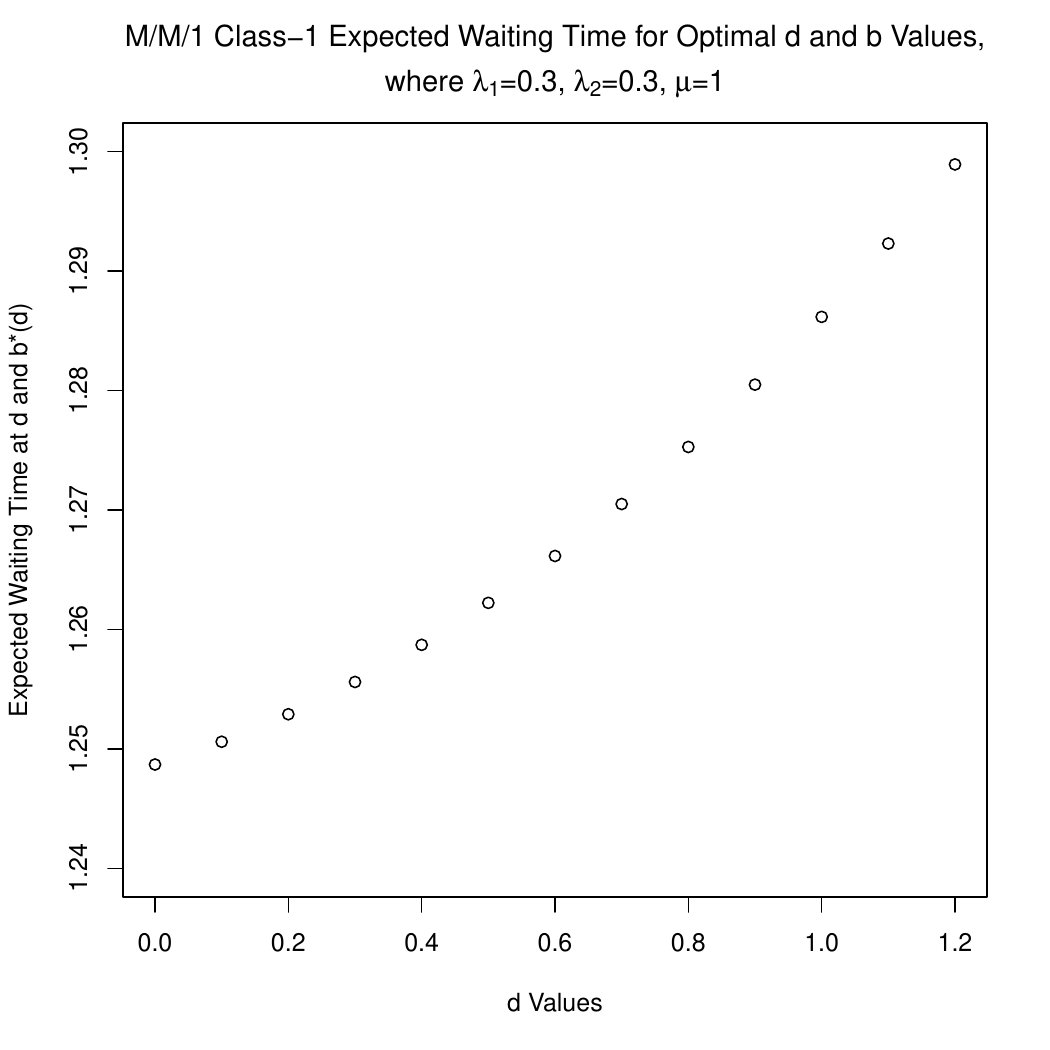}
  }
  \subfloat{
    \includegraphics[width=45mm, height=45mm]{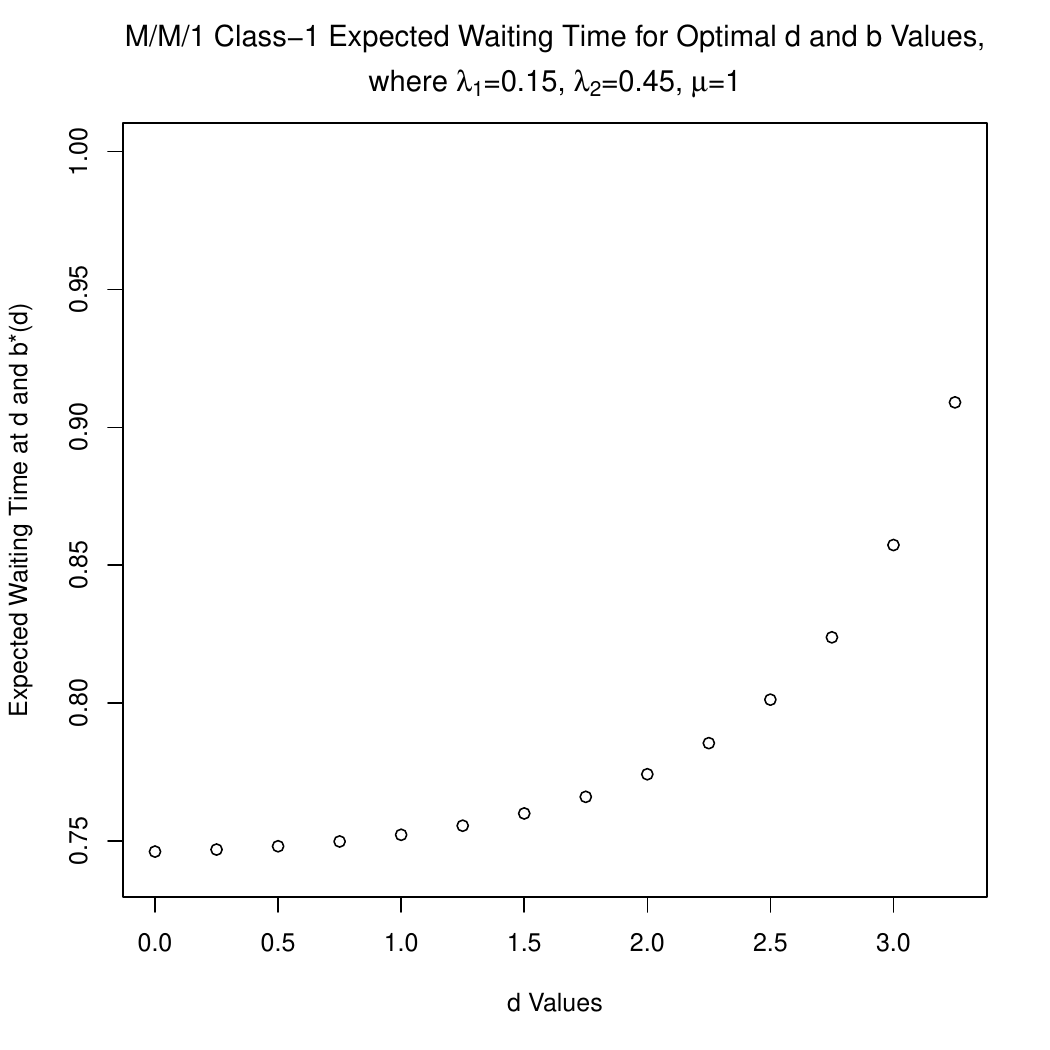}
  }
  \caption{$\EE \waitDparams{d}{b^*(d)}_1$ as a function of $d$ for the KPI $\PP(\waitD_2 < 4) \geq 0.85$.}
  \label{fig:mm1_kpi_examples}
\end{figure}}

\section{Exponential approximations of \classname{}-1 waiting times} \label{sec:approximation}

Beyond computing expected waiting times for \classname{}-1 customers, it is of great interest to characterize the entire waiting time distribution.
Currently, the theoretical tools available are insufficient for capturing the recursive dependence structure inherent to the \dapqname{}, which differs from the \apqname{} primarily by having \classname{}-1 customers experience different accreditation rates depending on the status of the queue, which breaks the necessary independence assumptions used in the analysis of the latter.
Instead, we turn to finding an \emph{analytic approximation} of the \classname{}-1 waiting time, which is a well-studied strategy for earlier versions of queues, but was previously inapplicable without the summary statistic computation we provide in this work.
In particular, we employ a zero-inflated exponential approximation, and measure its quality via both exact numerical validation and simulation procedures. 
We note that the purpose of this work is to identify computational procedures for summary statistics of the \dapqname{} that \emph{do not} require simulation, however, we still feel it is a reasonable tool to use for validation and justification purposes.
We focus on an exponential approximation in the M/M/1 case for simplicity, although based on the related work, it is plausible that the approximation would also be suitable in the M/G/1 case.

The driving motivation towards using an exponential approximation is that the \classname{}-1 waiting time in both the \apqname{} and the \dapqname{} interpolates between the FCFS (longest waiting times) and the \npqname{} (shortest waiting times). At both of these end points, for an M/M/1 queue, \classname{}-1 waiting times are distributed according to a zero-inflated exponential. We denote such a random variable by $Z \sim \zexp(\rho, \alpha)$, which for $\rho,\alpha>0$ has CDF defined for $t \geq 0$ by
\[\label{eqn:zexp-def}
	\PP(Z \leq t) = 1 - \rho e^{-\alpha t}.
\]

It is easy to see from \cref{eqn:zexp-def} that $\PP(Z = 0) = 1 - \rho$. Fortunately, we know that $\rho = \lam/\mu$, so there is only one parameter to optimize. Then, by integrating $\PP(Z > t)$, we obtain $\alpha = \rho / \EE Z$. In other words, if we want $Z = \wait_1$, we can define the zero-inflated exponential approximation with only the occupancy ratio $\rho$ and the expected waiting time $\EE \wait_1$. Fortunately, this is \emph{exactly} what we have available to characterize $\waitD_1$.

Approximating queueing dynamics using exponential models has a long history in the literature. Most relevant to this work is \citet{abate95approx}, who also approximate waiting times with a zero-inflated exponential in the M/G/1 queue. See the references therein for other historical uses of various exponential approximations.
Despite this lengthy literature, to the best of our knowledge, exponential approximations of queues have only ever been used to deal with intractability due to the \emph{service and arrival distributions}. 
Instead, we propose an approximation to overcome intractability due to the \emph{queueing discipline}.

\subsection{Analytical approximation error for \apqname{} ($d=0$)}

To justify our use of an exponential approximation, we first compare its accuracy in the simpler \apqname{}, where we also have access to the exact waiting time CDF in order to compare. In \cref{fig:approx_high_rho}, we plot the exact CDFs against the approximate CDFs for various levels of $\rho$. In \cref{fig:diff_high_rho}, we present this same information in another way, plotting the absolute difference between the exact CDFs and the approximate CDFs. 

The approximation seems to work well when $b$ is large, $\lam_1 > \lam_2$, or $b$ is very small. For extreme $b$ (near 0 or 1), the queueing discipline is close to what one would observe in the NPQ and FCFS cases, respectively, for which the approximation is exact.
When $\lam_1 > \lam_2$, the lower priority class is not as impactful on the service dynamics as the higher priority class, and as $\lam_2 \to 0$ the approximation once more tends towards the exact solution.
We also consider the same analysis from a different perspective in \cref{fig:diff_high_rho}, comparing the absolute difference between the true and approximate waiting time CDFs. Here we see nontrivial error for small $t$ (corresponding to 1-5 average service lengths), although when $\rho$ is large (as we expect in a health care setting), this error caps out around 5\%, and similarly the error appears to be smaller when $\lam_1 > \lam_2$.

\preprint{
\begin{figure}[htb]
  \centering
  \subfloat{
    \includegraphics[width=80mm, height=70mm]{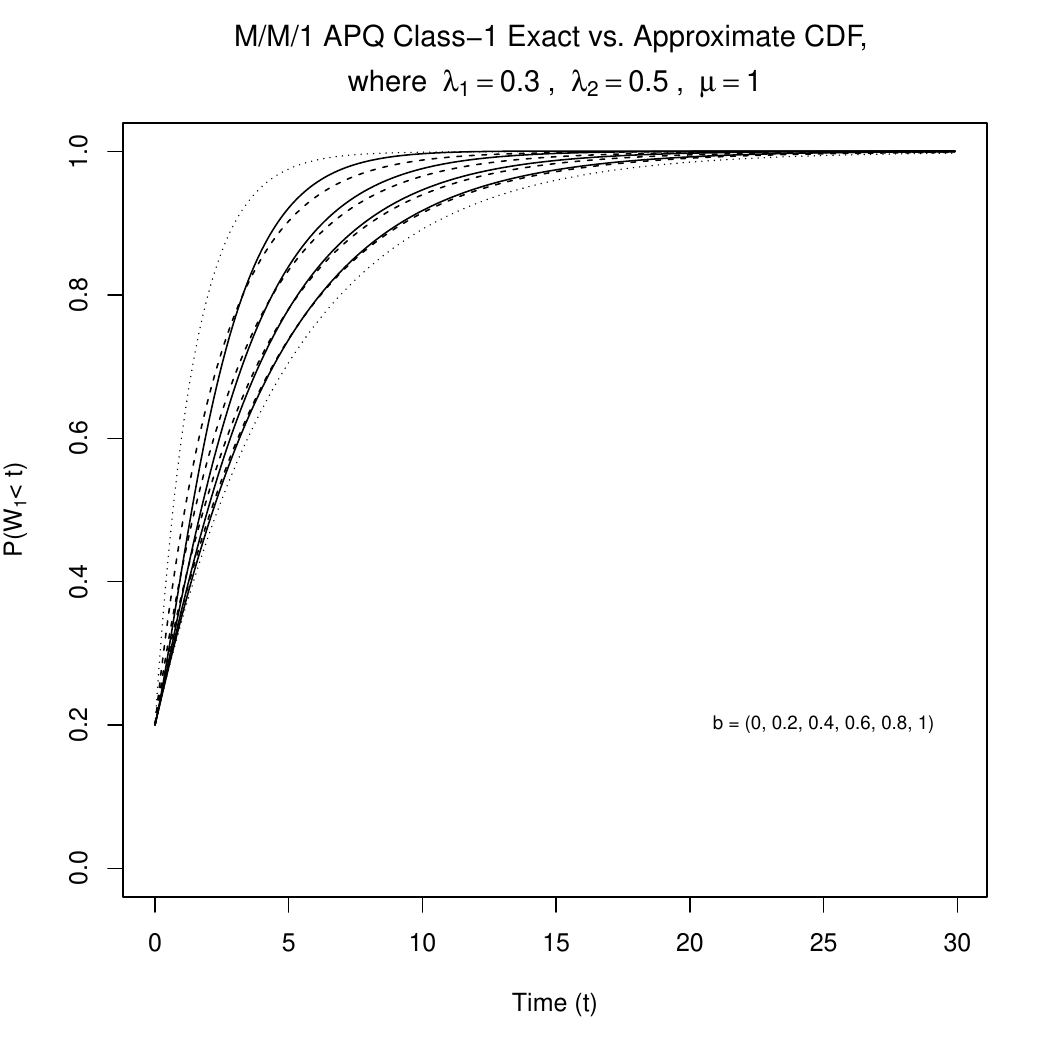}
  }
  \subfloat{
    \includegraphics[width=80mm, height=70mm]{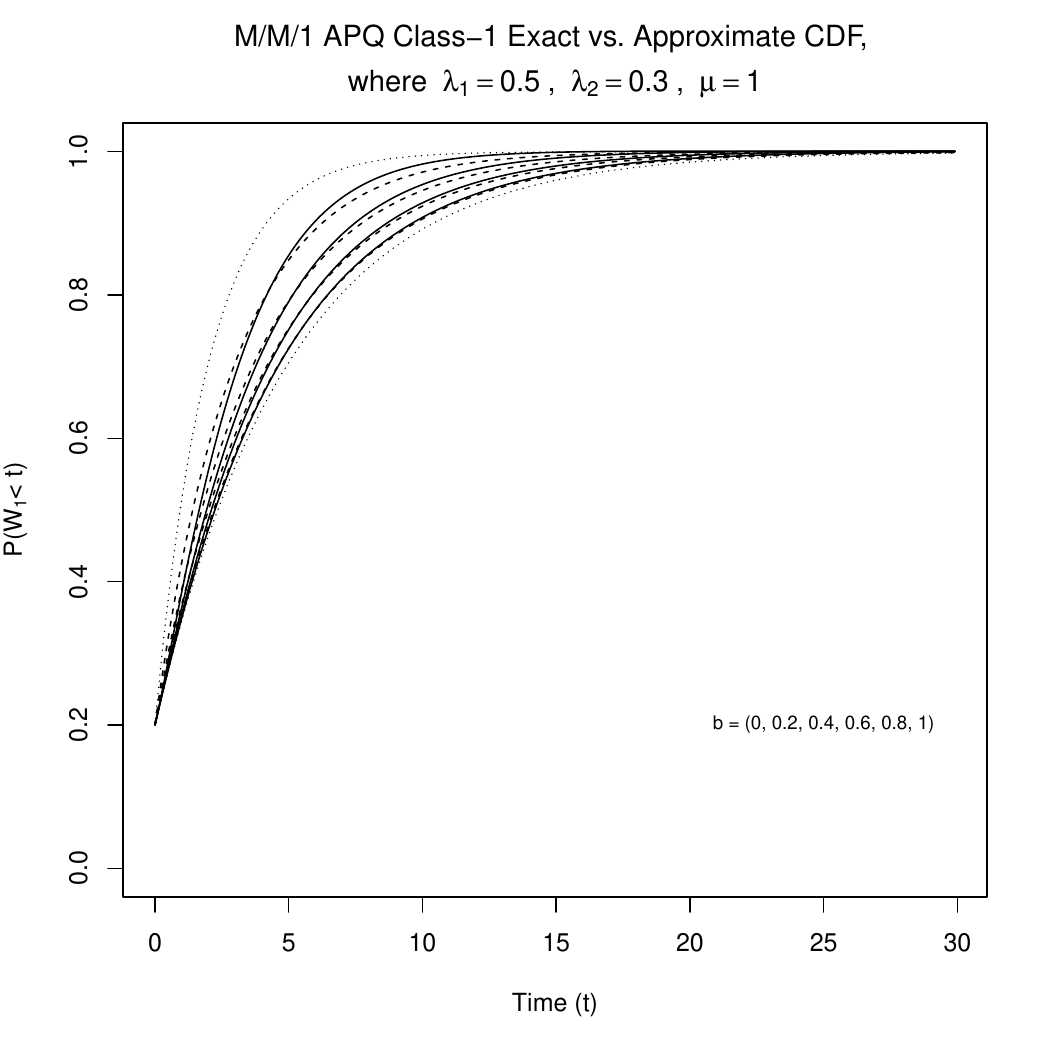}
  }
\\
  \subfloat{
    \includegraphics[width=80mm, height=70mm]{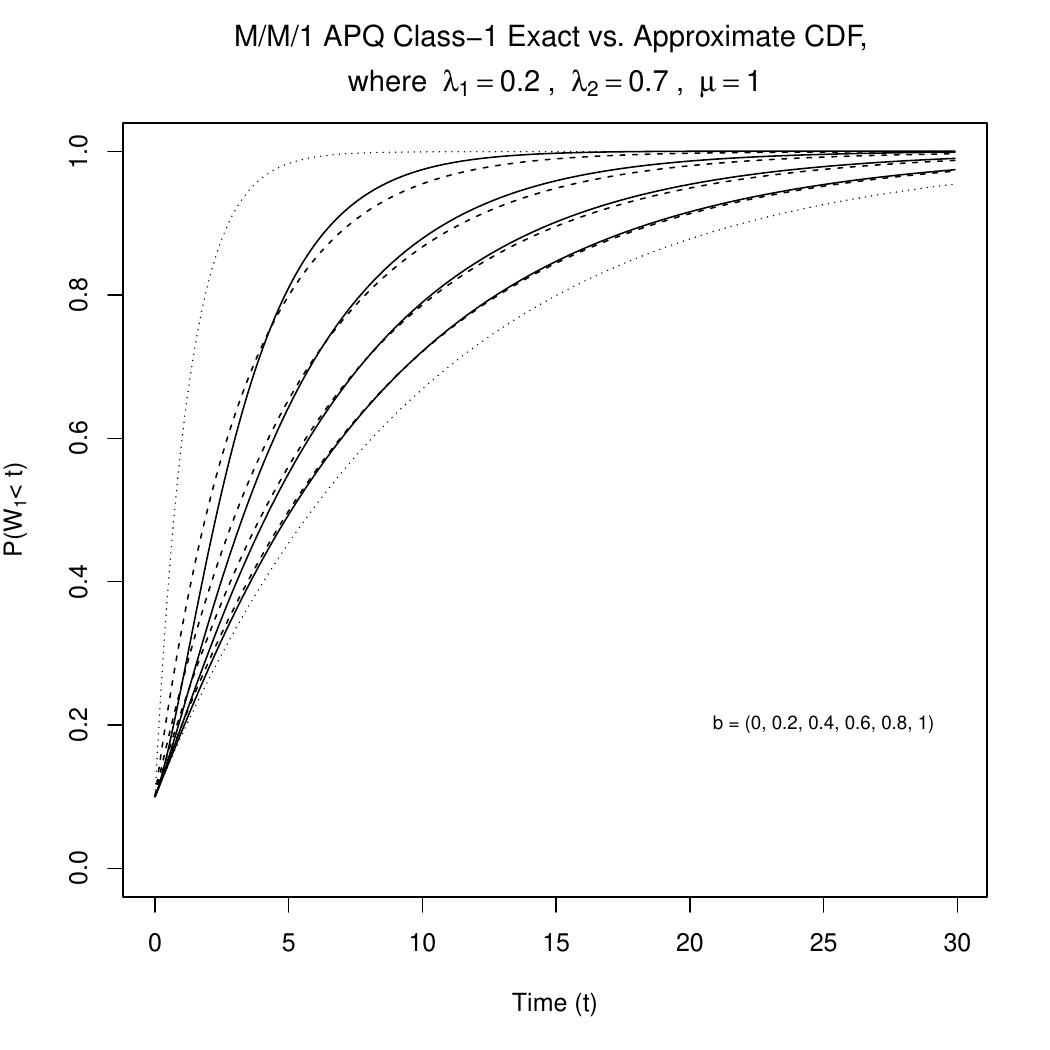}
  }
  \subfloat{
    \includegraphics[width=80mm, height=70mm]{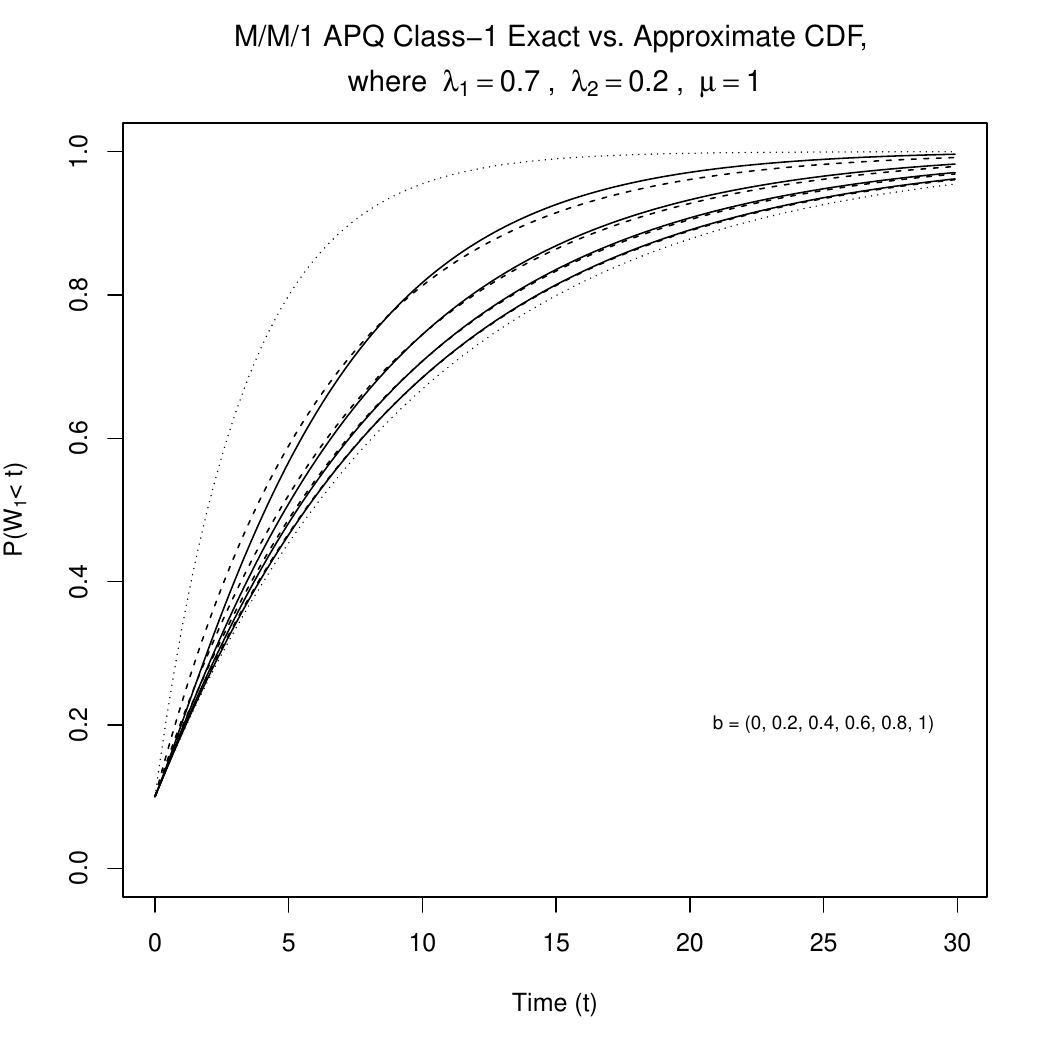}
  }
  \caption{Exact (\full) v.s.\ exponential approximations (\dashed) of \highname{} waiting time CDFs. Outer lines (\dotted) correspond to \npqname{} ($b=0$) and FCFS ($b=1$).}
  \label{fig:approx_high_rho}
\end{figure}
\begin{figure}[htb]
  \centering
  \subfloat{
    \includegraphics[width=80mm, height=70mm]{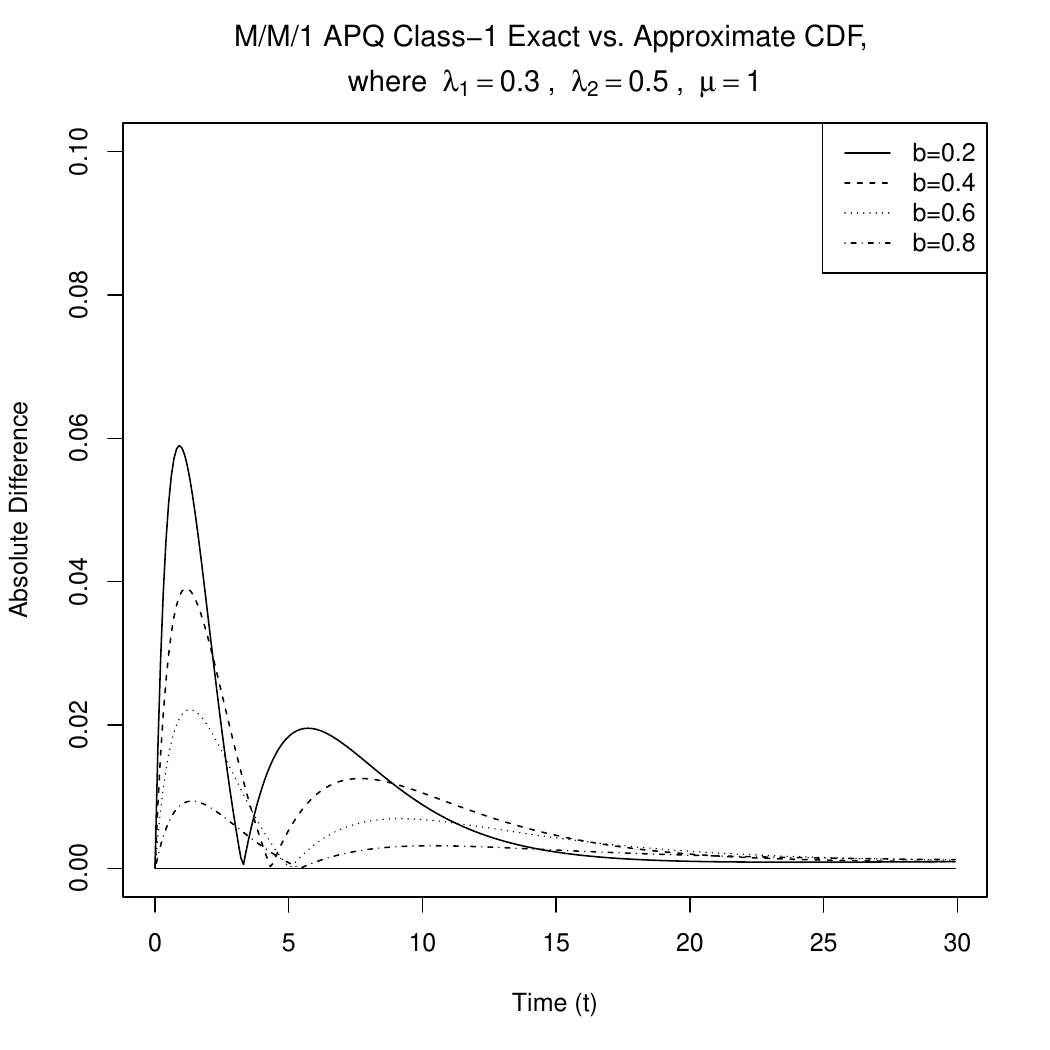}
  }
  \subfloat{
    \includegraphics[width=80mm, height=70mm]{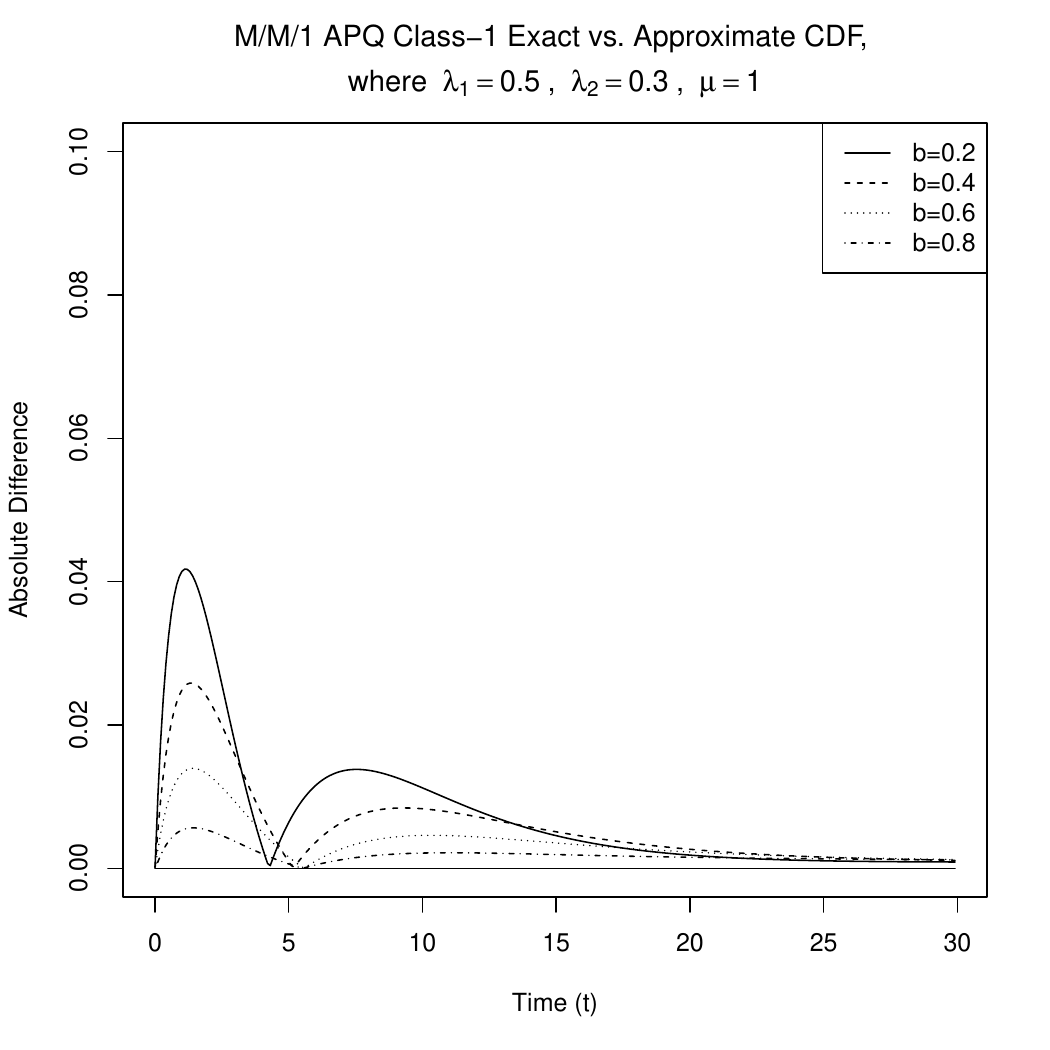}
  }
\\
  \subfloat{
    \includegraphics[width=80mm, height=70mm]{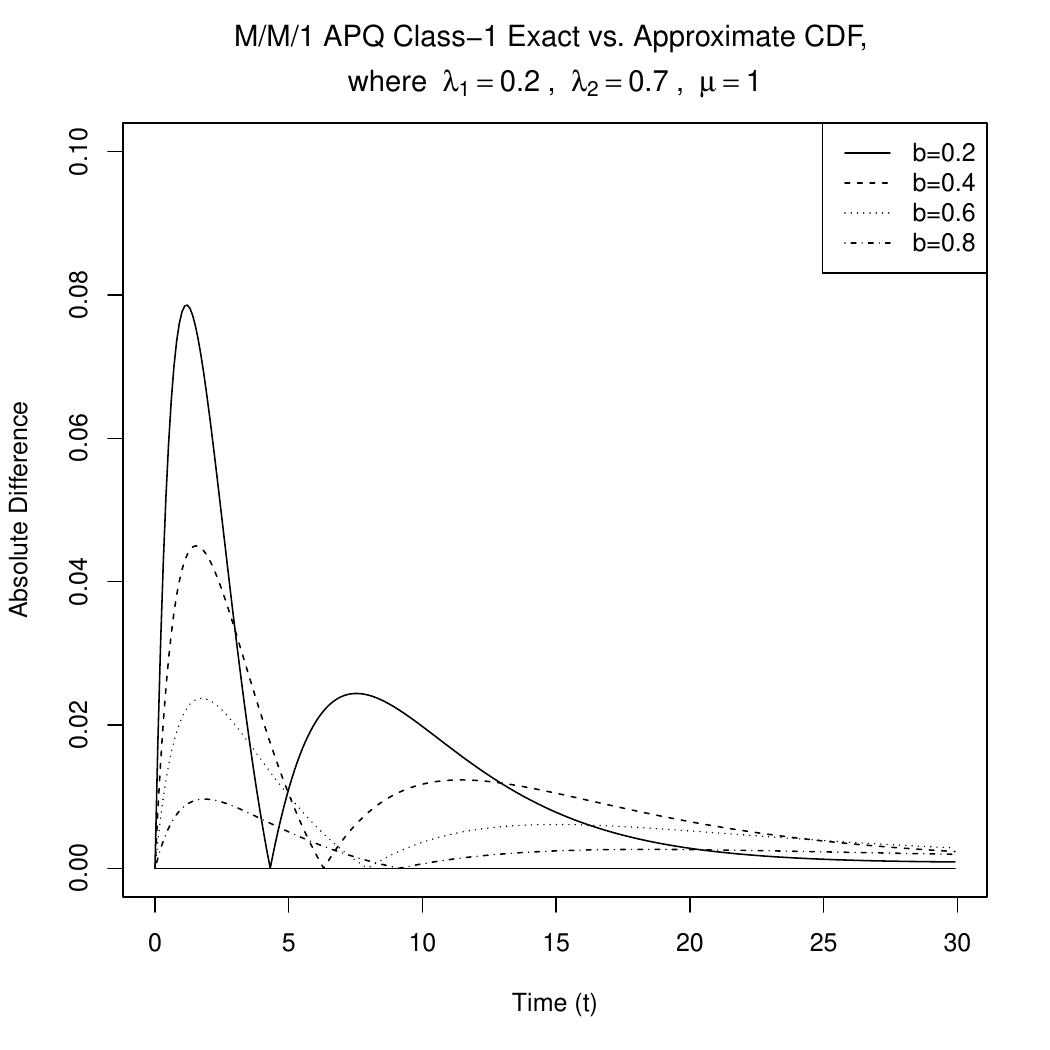}
  }
  \subfloat{
    \includegraphics[width=80mm, height=70mm]{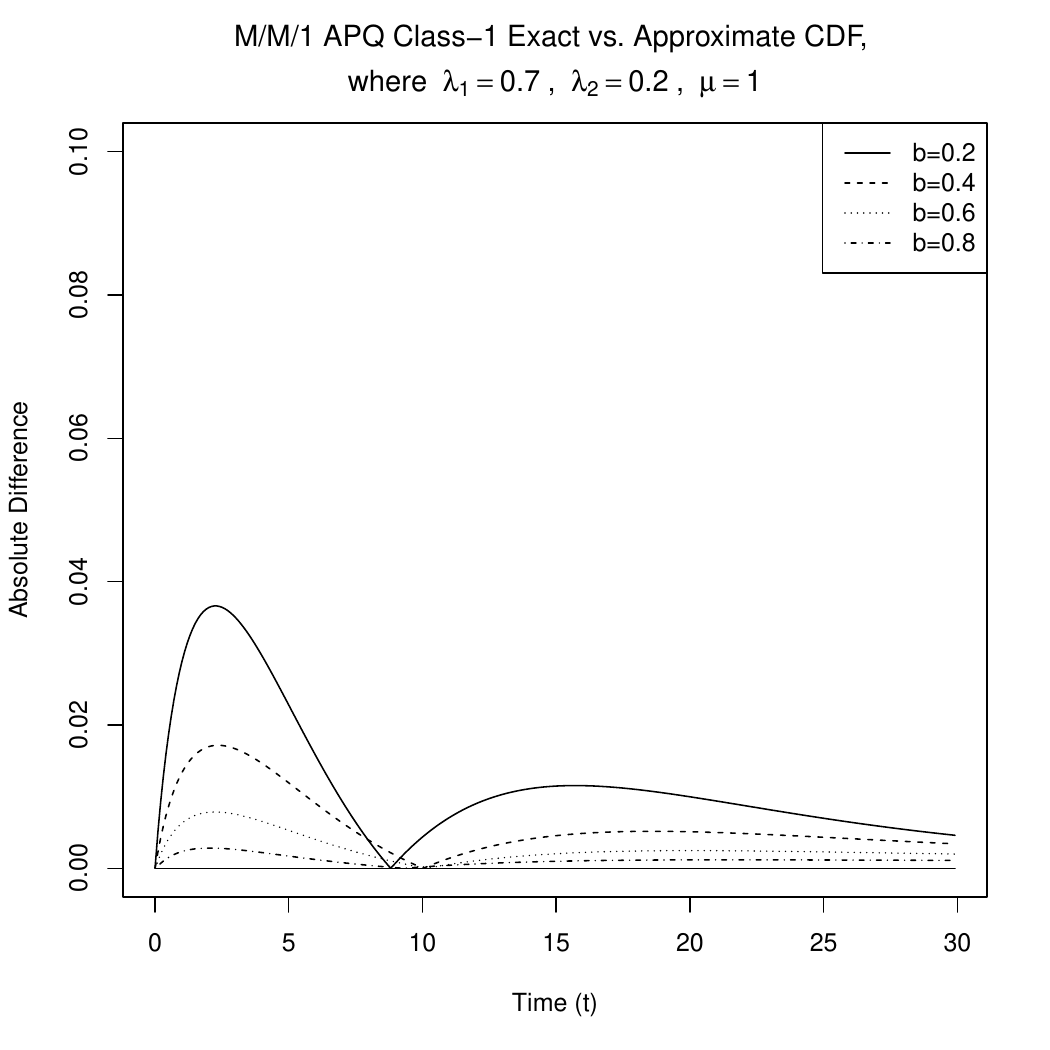}
  }
  \caption{Absolute difference between exact and exponential approximations of \highname{} waiting time CDFs.}
  \label{fig:diff_high_rho}
\end{figure}}

\infor{
\begin{figure}[htb]
  \centering
  \subfloat{
    \includegraphics[width=70mm, height=55mm]{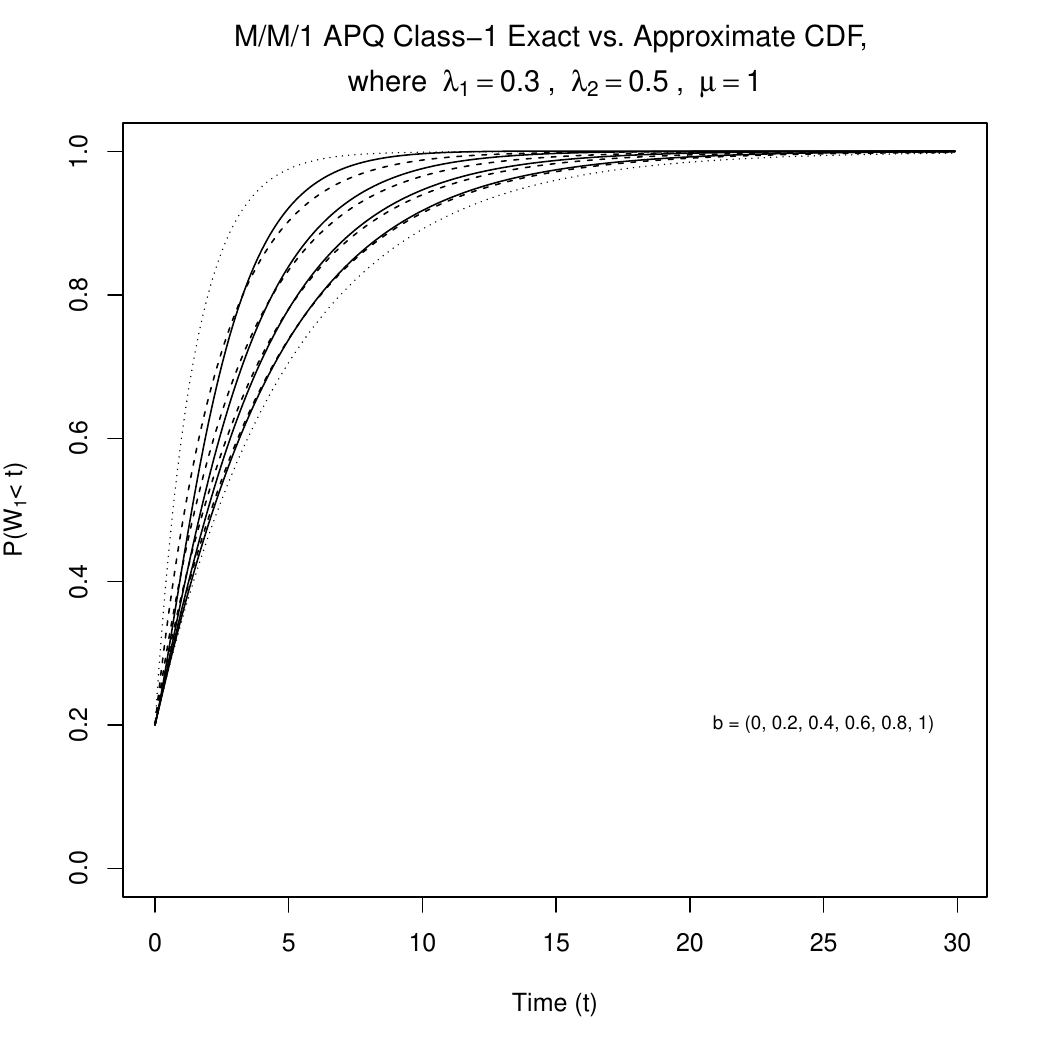}
  }
  \subfloat{
    \includegraphics[width=70mm, height=55mm]{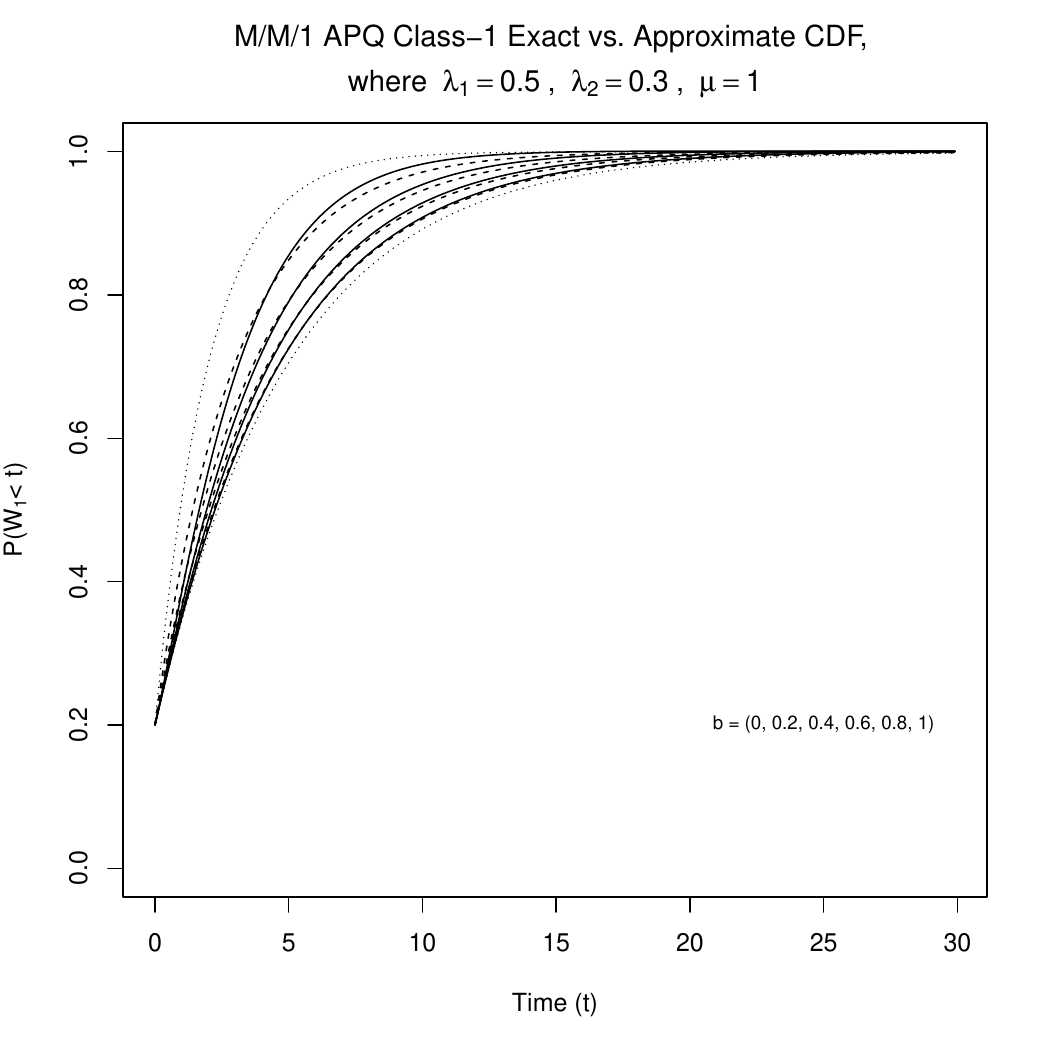}
  }
\\
  \subfloat{
    \includegraphics[width=70mm, height=55mm]{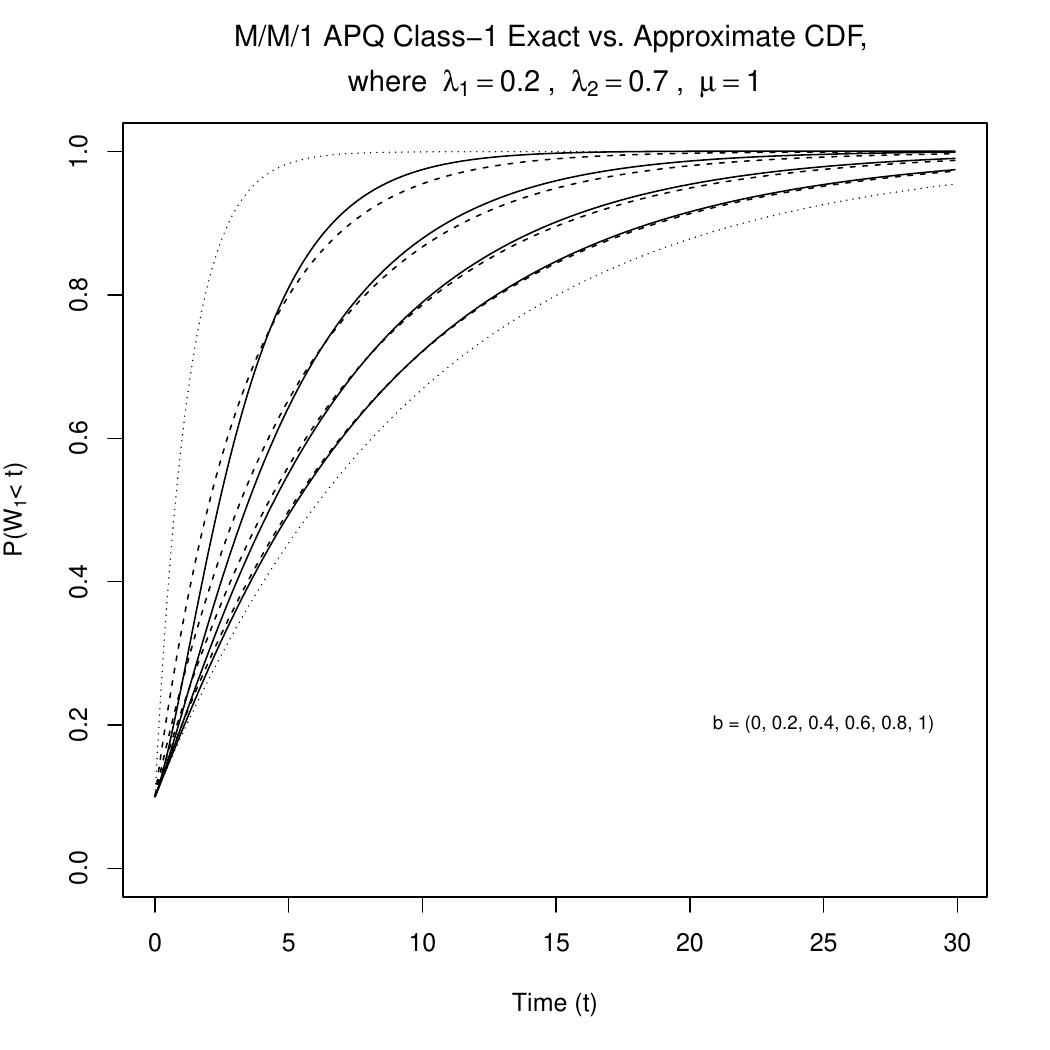}
  }
  \subfloat{
    \includegraphics[width=70mm, height=55mm]{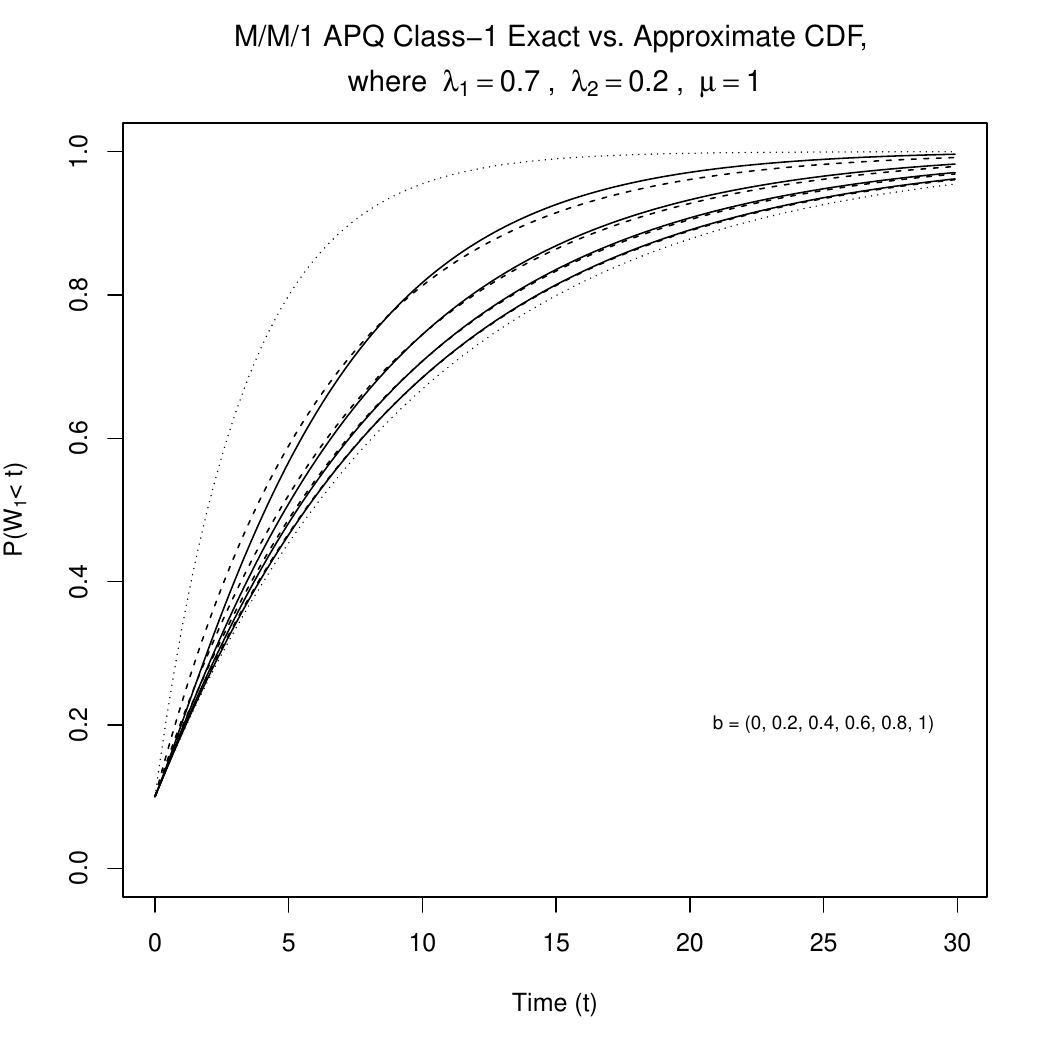}
  }
  \caption{Exact (\full) v.s.\ exponential approximations (\dashed) of \highname{} waiting time CDFs. Outer lines (\dotted) correspond to \npqname{} ($b=0$) and FCFS ($b=1$).}
  \label{fig:approx_high_rho}
\end{figure}
\begin{figure}[htb]
  \centering
  \subfloat{
    \includegraphics[width=70mm, height=55mm]{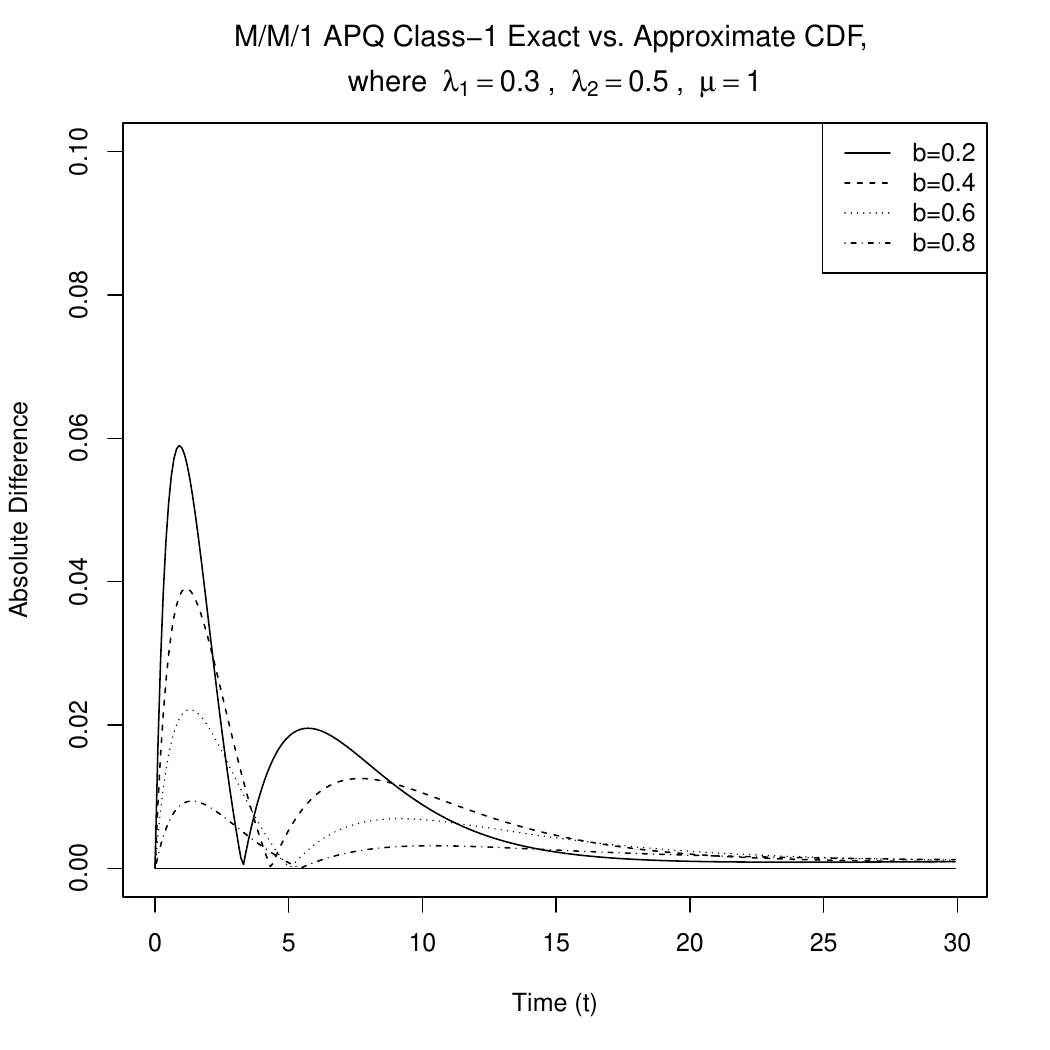}
  }
  \subfloat{
    \includegraphics[width=70mm, height=55mm]{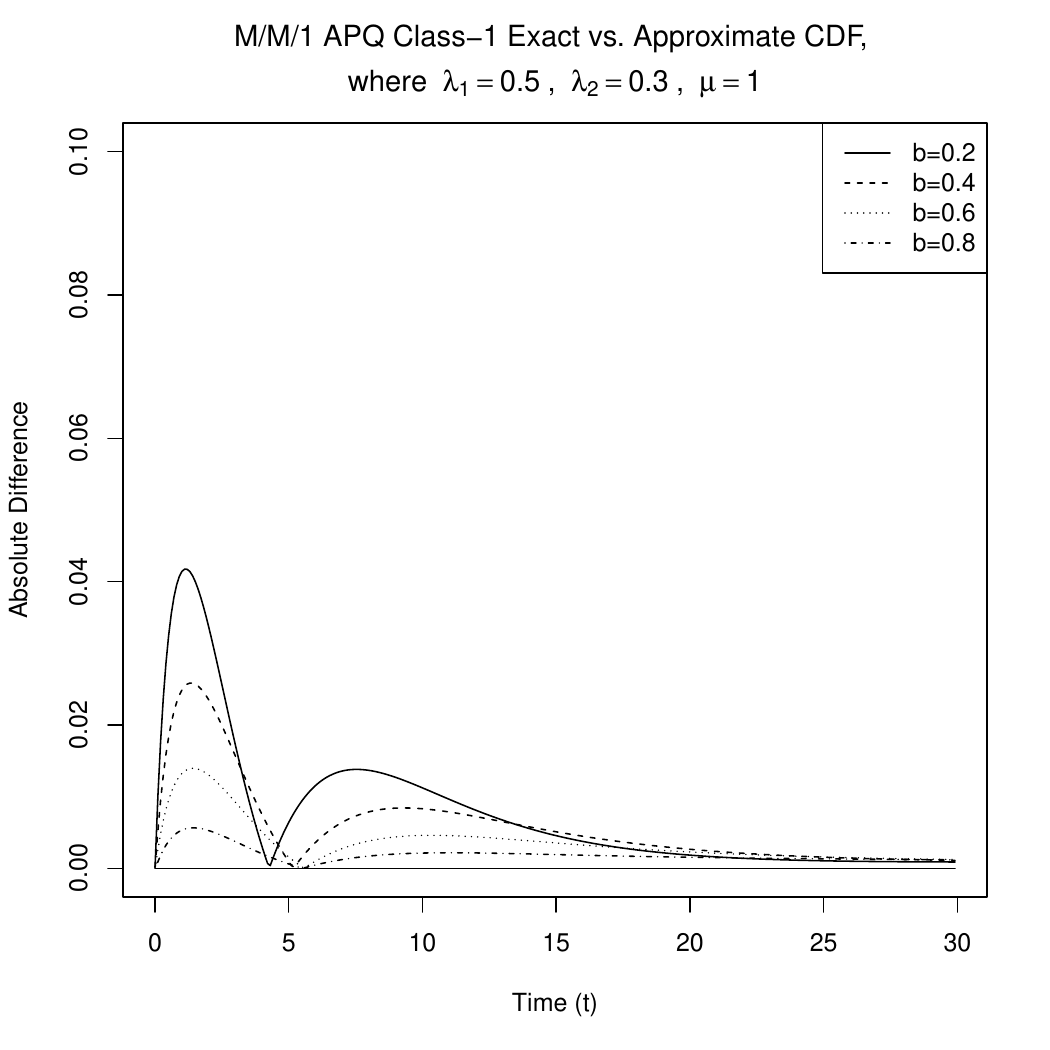}
  }
\\
  \subfloat{
    \includegraphics[width=70mm, height=55mm]{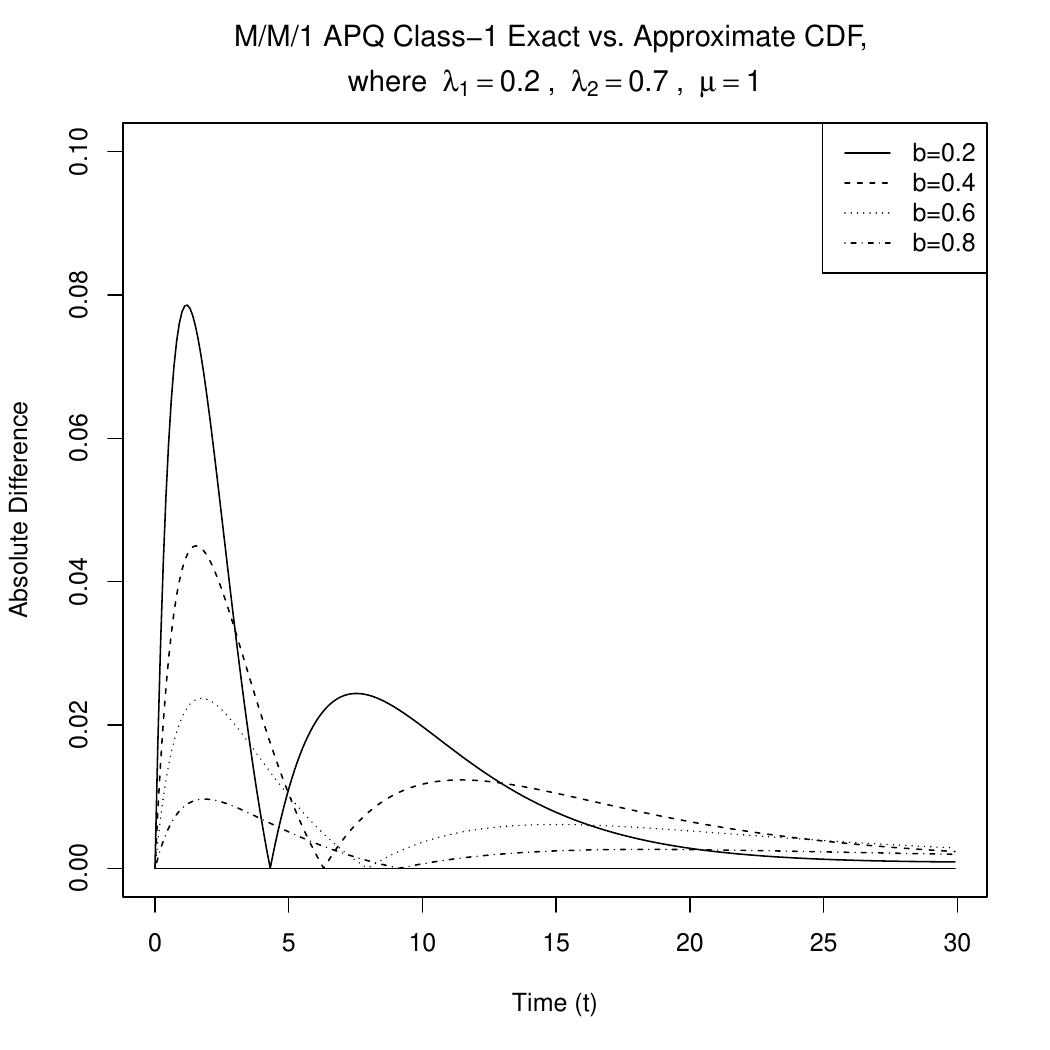}
  }
  \subfloat{
    \includegraphics[width=70mm, height=55mm]{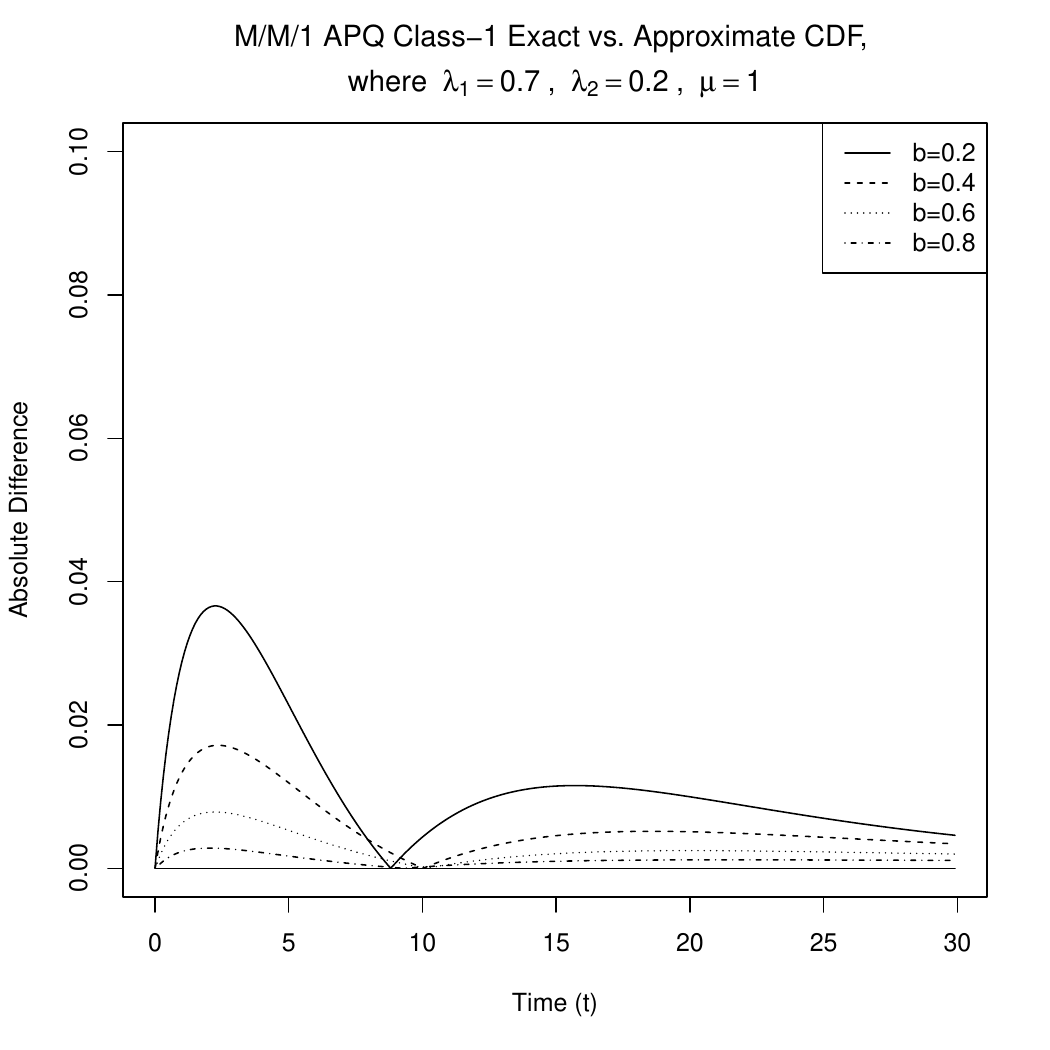}
  }
  \caption{Absolute difference between exact and exponential approximations of \highname{} waiting time CDFs.}
  \label{fig:diff_high_rho}
\end{figure}}

\subsection{Simulated approximation error for $d>0$}

We now perform a similar analysis of the zero-inflated exponential approximation from the last section, but for the actual queueing discipline of interest (\dapqname). Since the true \classname{}-1 waiting time CDF is unknown (hence the approximation), we instead compare to the CDF computed via simulation. 
We avoid using these simulations elsewhere in the paper, as the main focus is on analytical expressions for the quantities of interest, but feel that in this section they are a warranted tool for justification.

The simulations are run using a simple Python script to brute force reproduce a \dapqname{} via discrete-event simulation (i.e., every ``customer'' is generated and explicitly moves through the queue). Each set of parameters was simulated for $n=4000$ customers following a burn-in period of $1500$ customers, from which empirical CDFs were computed, and then these empirical CDFs were averaged over $50$ runs. The averaged, empirical CDFs were compared to the known CDFs for \fcfsname{}, \apqname{}, and \dapqname{} \classname{}-2 as validation, and found numerically indistinguishable (see \cref{fig:simulated_high_rho}).

\preprint{
\begin{figure}[htb]
  \centering
  \subfloat{
    \includegraphics[width=80mm, height=70mm]{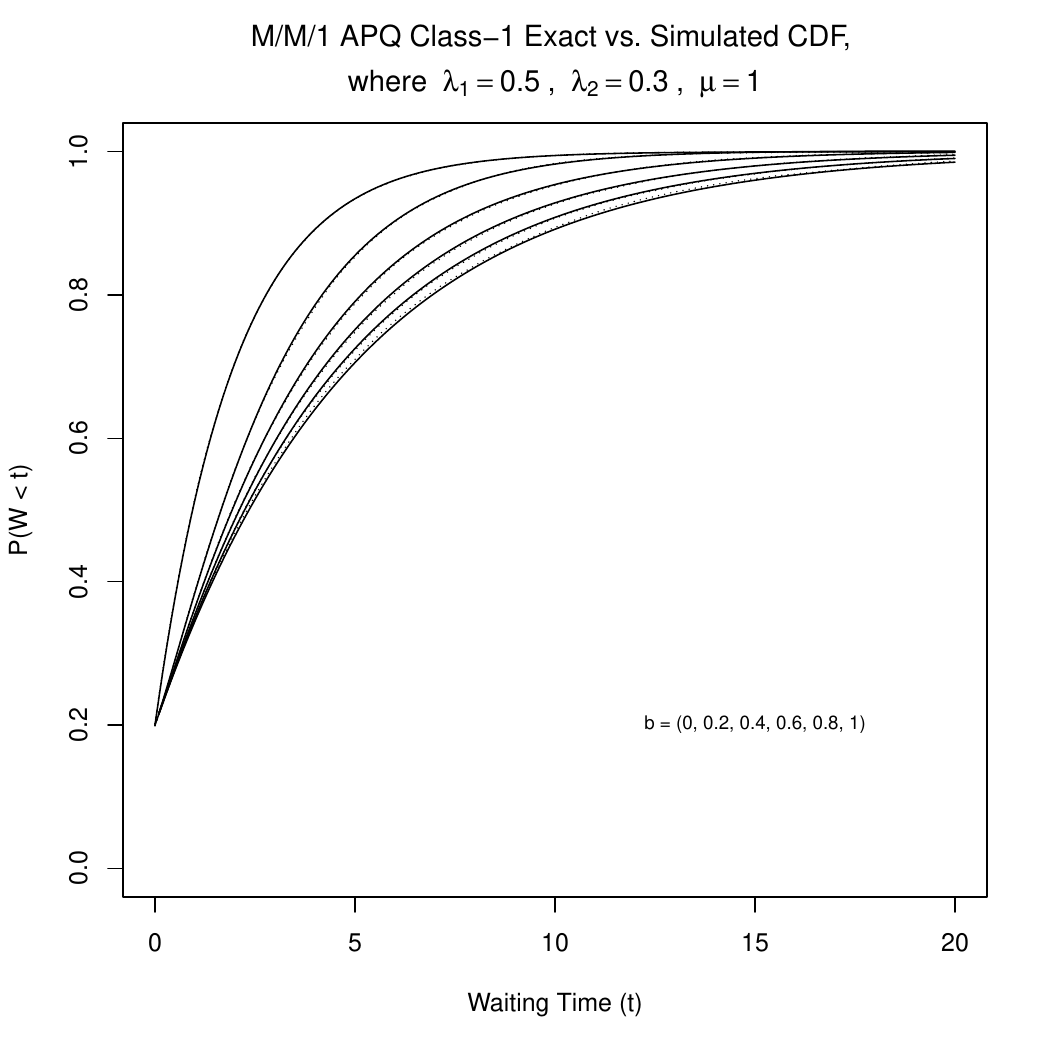}
  }
  \subfloat{
    \includegraphics[width=80mm, height=70mm]{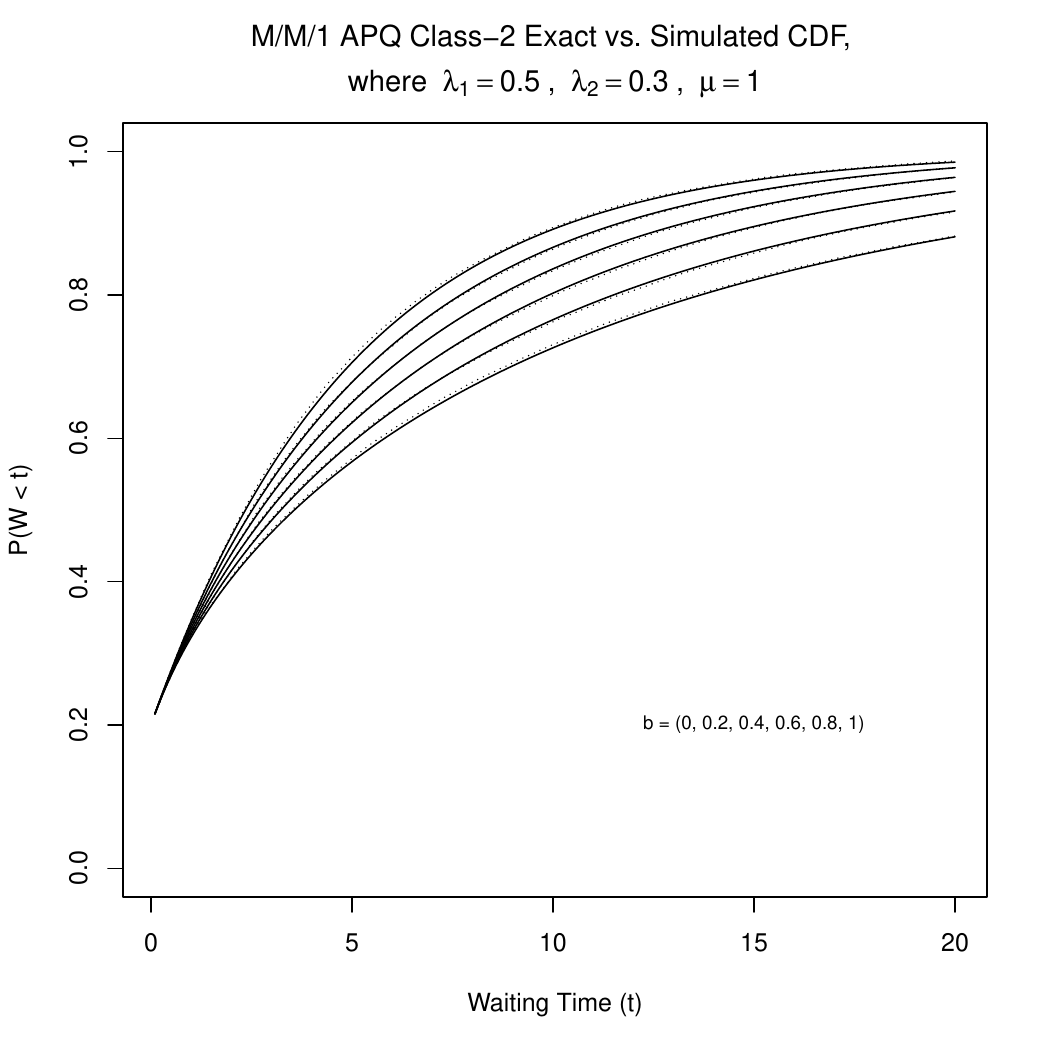}
  }
\\
  \subfloat{
    \includegraphics[width=80mm, height=70mm]{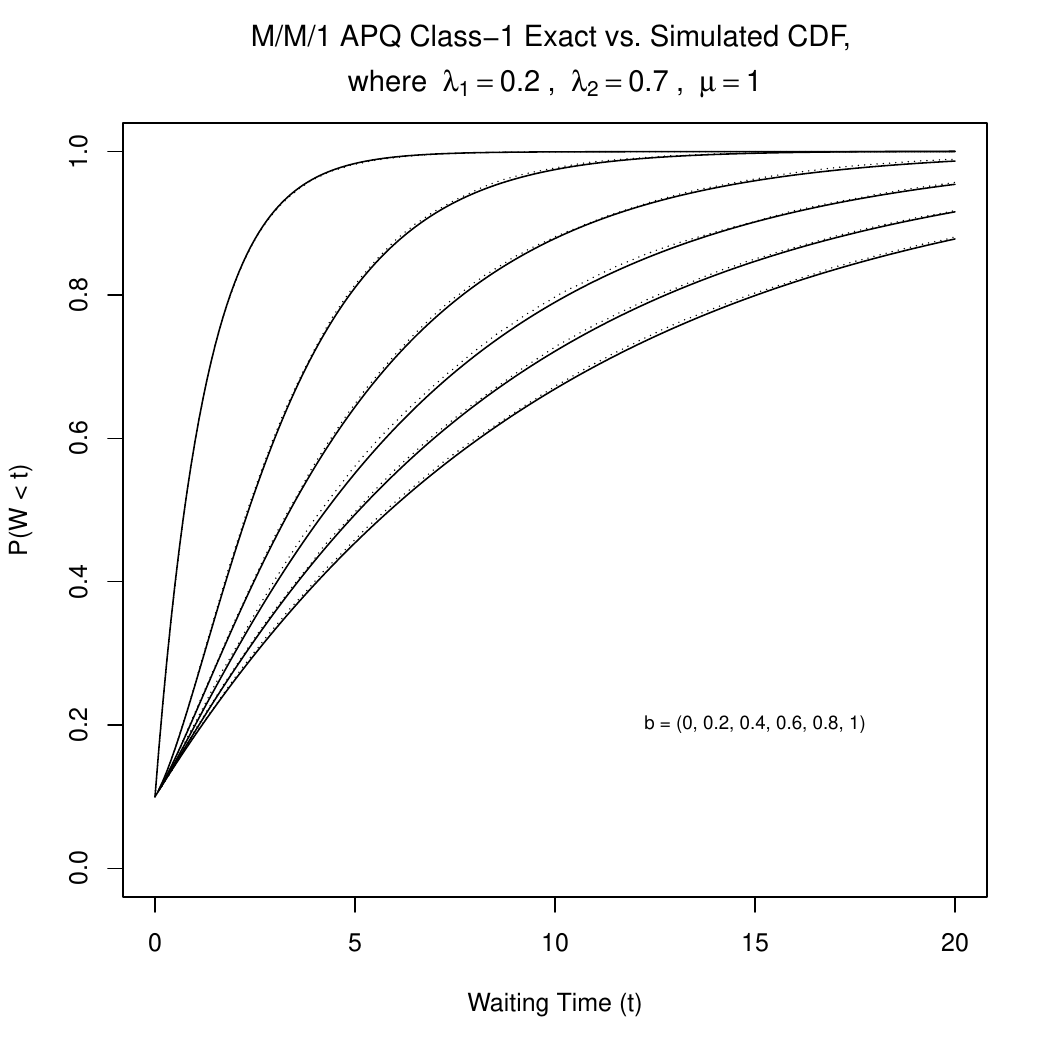}
  }
  \subfloat{
    \includegraphics[width=80mm, height=70mm]{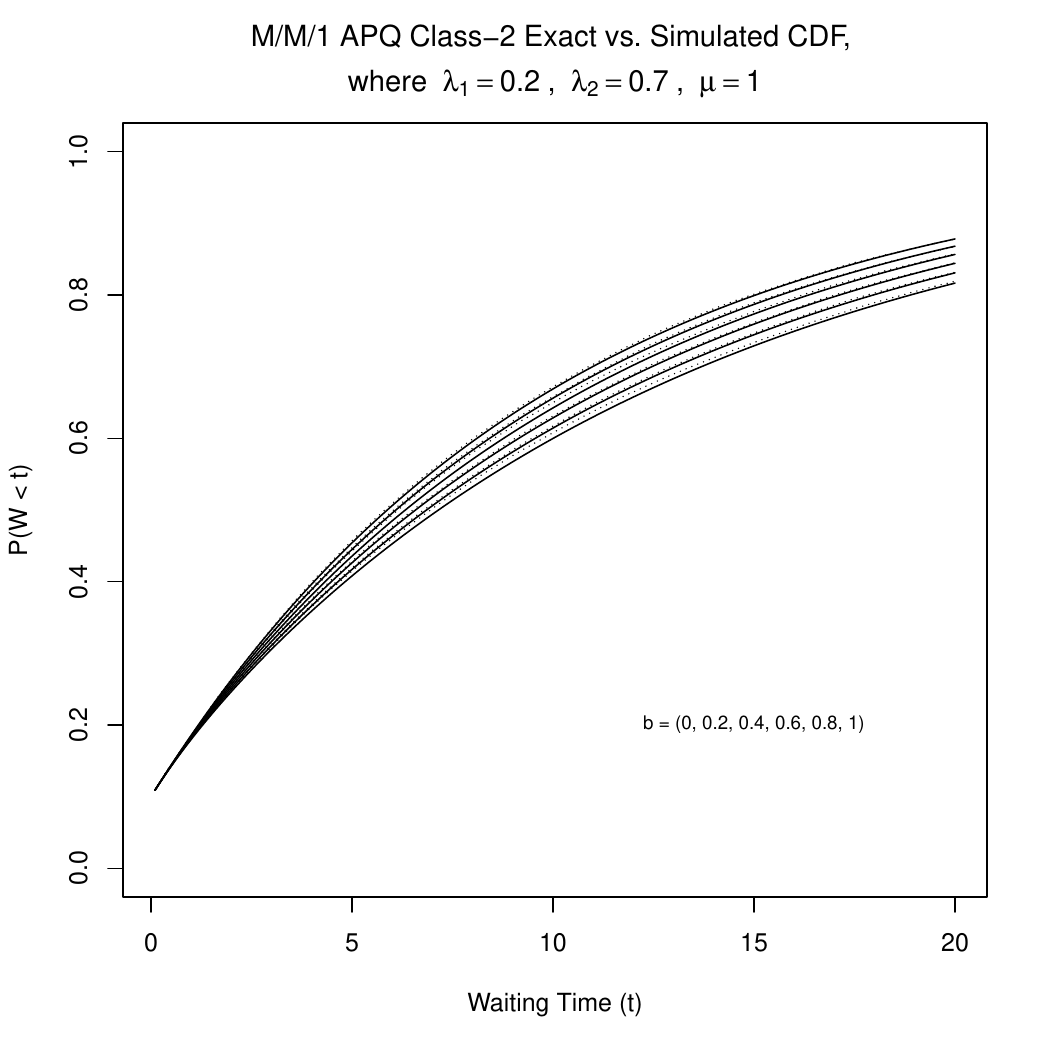}
  }
  \caption{Exact (\full) v.s.\ simulated (\dashed) waiting time CDFs. 
  For \classname{}-1, the leftmost curve is \npqname{} $(b=0)$, while for \classname{}-2, the leftmost curve is \fcfsname{} $(b=1)$.}
  \label{fig:simulated_high_rho}
\end{figure}}

\infor{
\begin{figure}[htb]
  \centering
  \subfloat{
    \includegraphics[width=70mm, height=55mm]{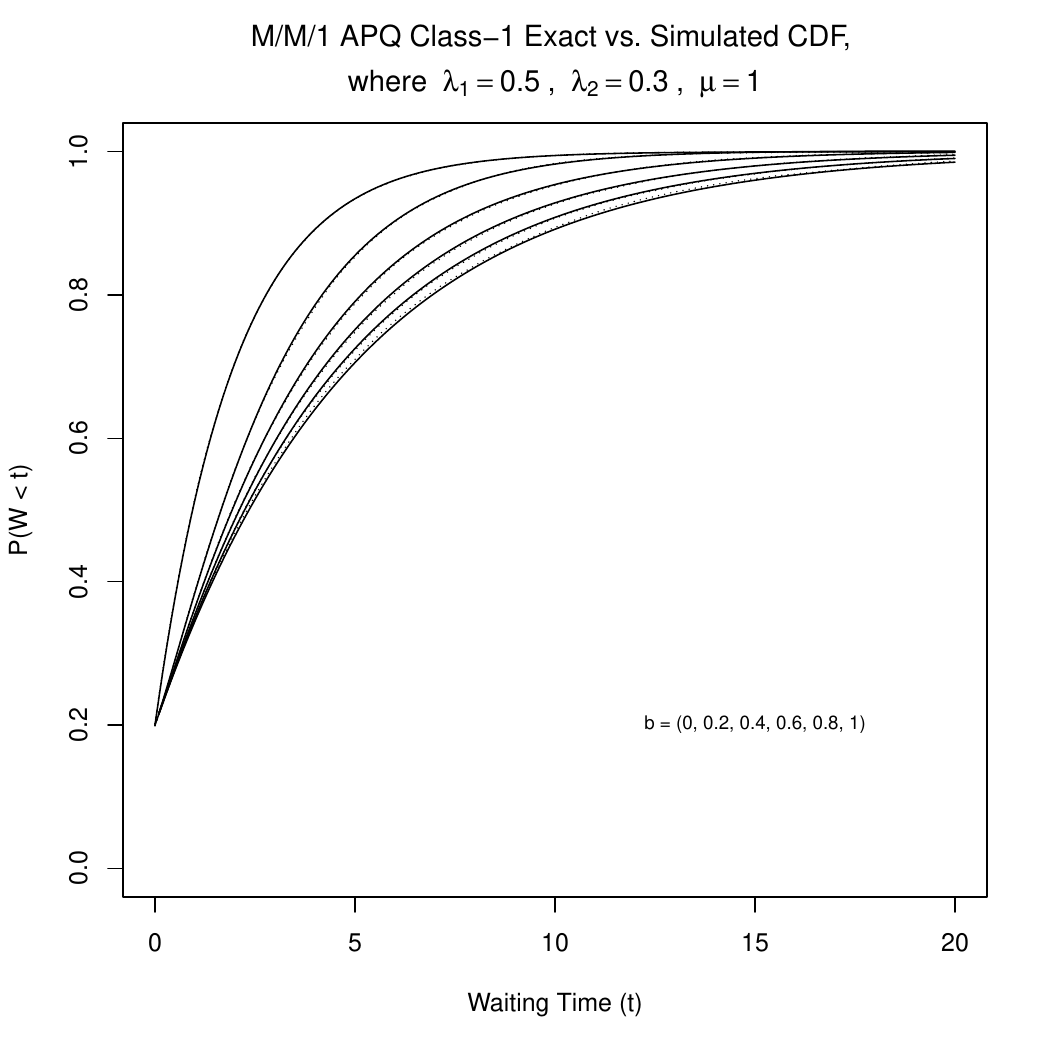}
  }
  \subfloat{
    \includegraphics[width=70mm, height=55mm]{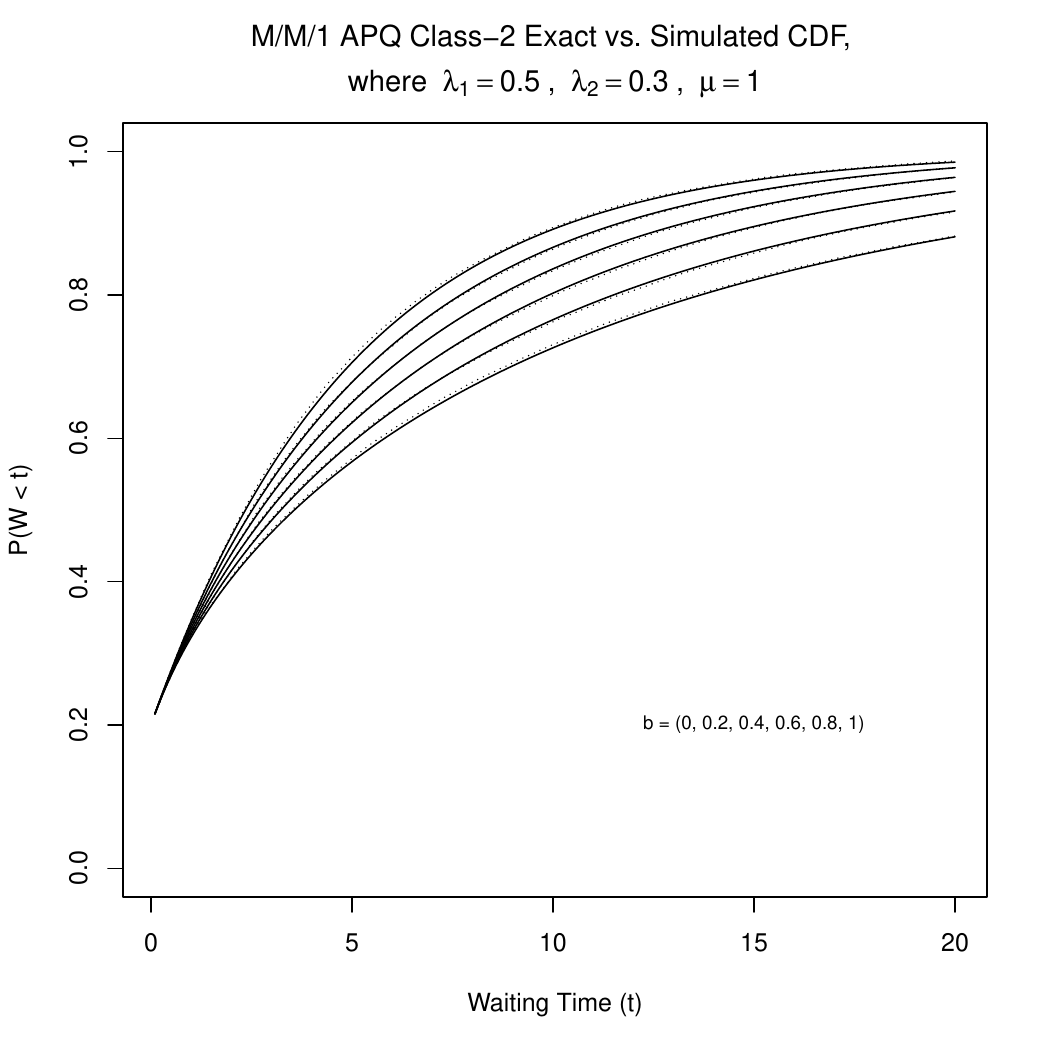}
  }
\\
  \subfloat{
    \includegraphics[width=70mm, height=55mm]{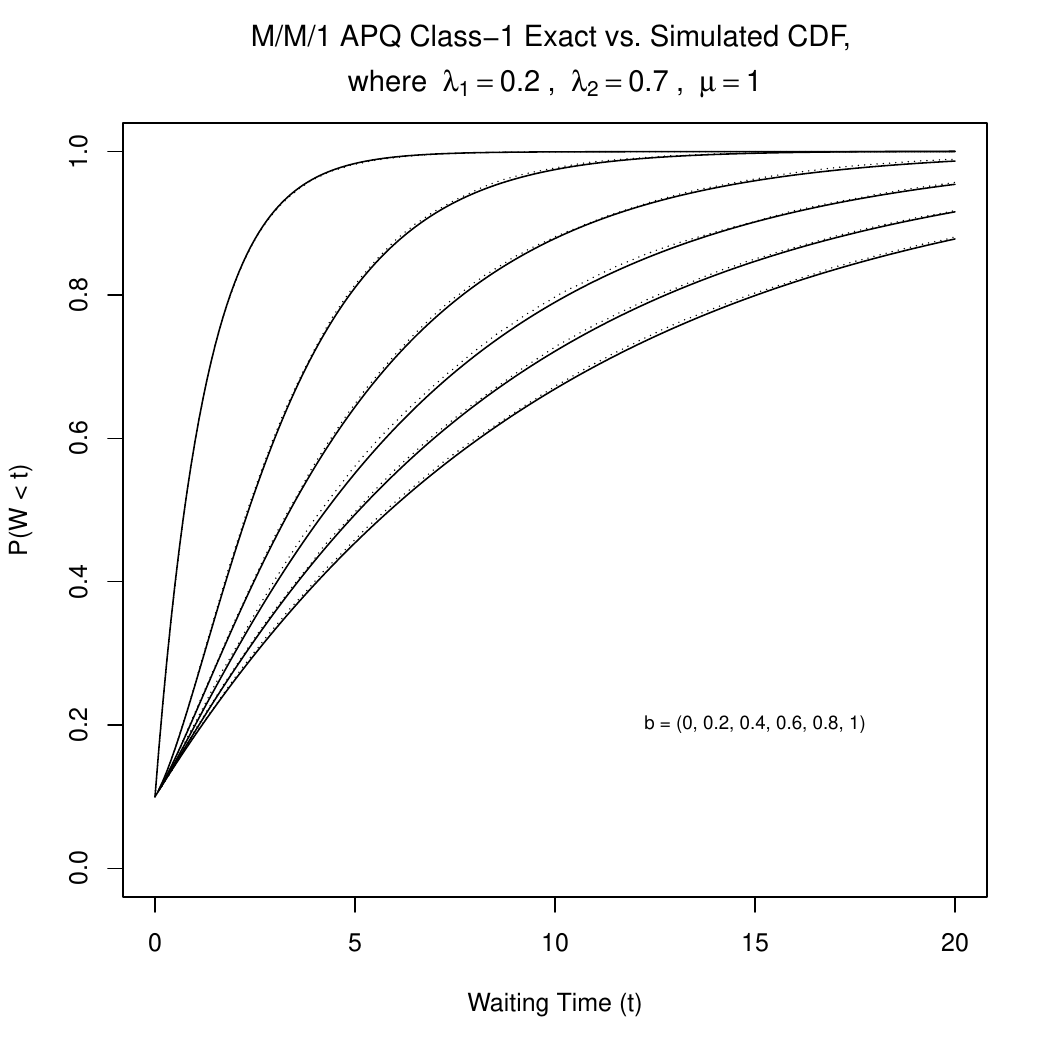}
  }
  \subfloat{
    \includegraphics[width=70mm, height=55mm]{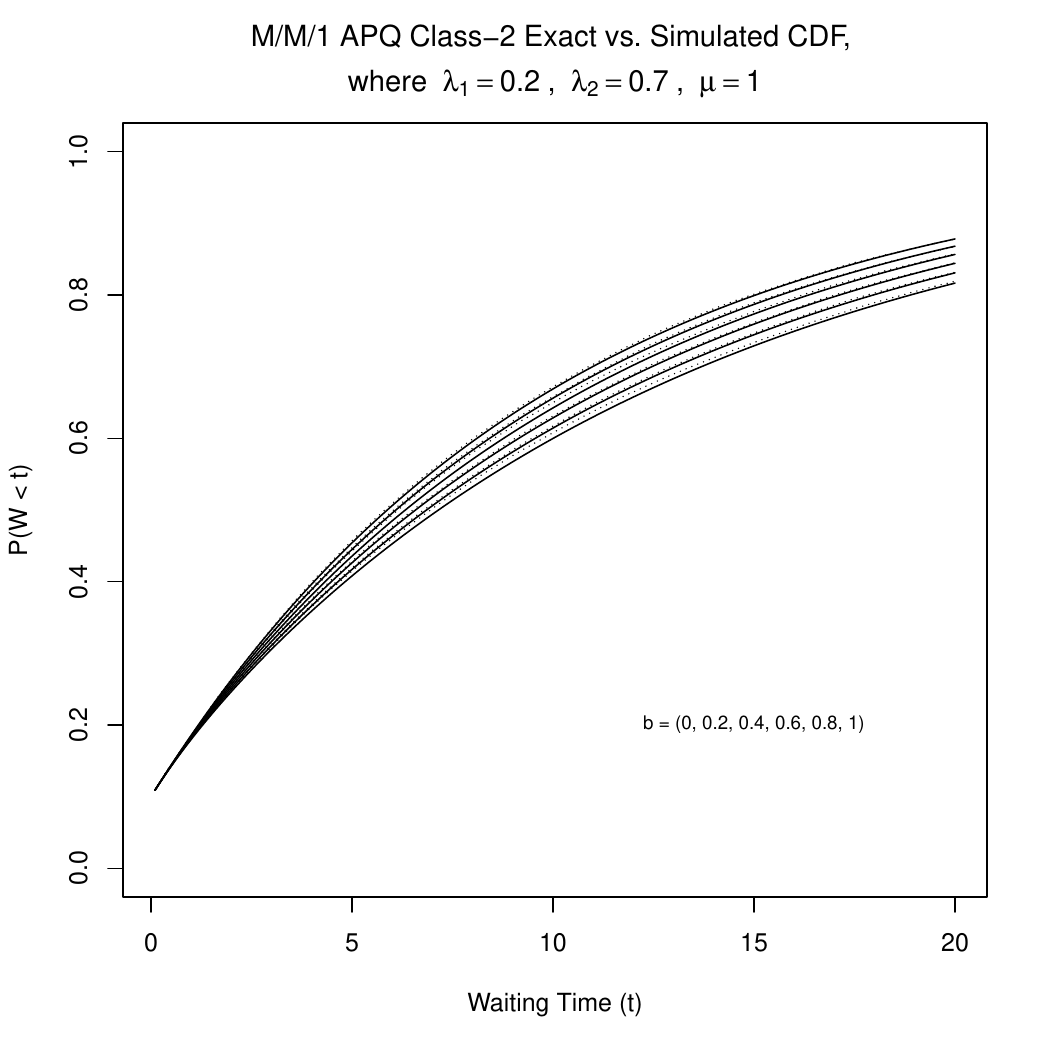}
  }
  \caption{Exact (\full) v.s.\ simulated (\dashed) waiting time CDFs. 
  For \classname{}-1, the leftmost curve is \npqname{} $(b=0)$, while for \classname{}-2, the leftmost curve is \fcfsname{} $(b=1)$.}
  \label{fig:simulated_high_rho}
\end{figure}}

We now reproduce \cref{fig:diff_high_rho} for the simulated \dapqname{} to evaluate the zero-inflated exponential approximation in this setting. \cref{fig:diff_simulated_high_rho} demonstrates that the accuracy is mostly unchanged from the \apqname{} setting $(d=0)$, with the exception that the curves are  less smooth due to the stochasticity of the simulations. 

\preprint{
\begin{figure}[htb]
  \centering
  \subfloat{
    \includegraphics[width=80mm, height=70mm]{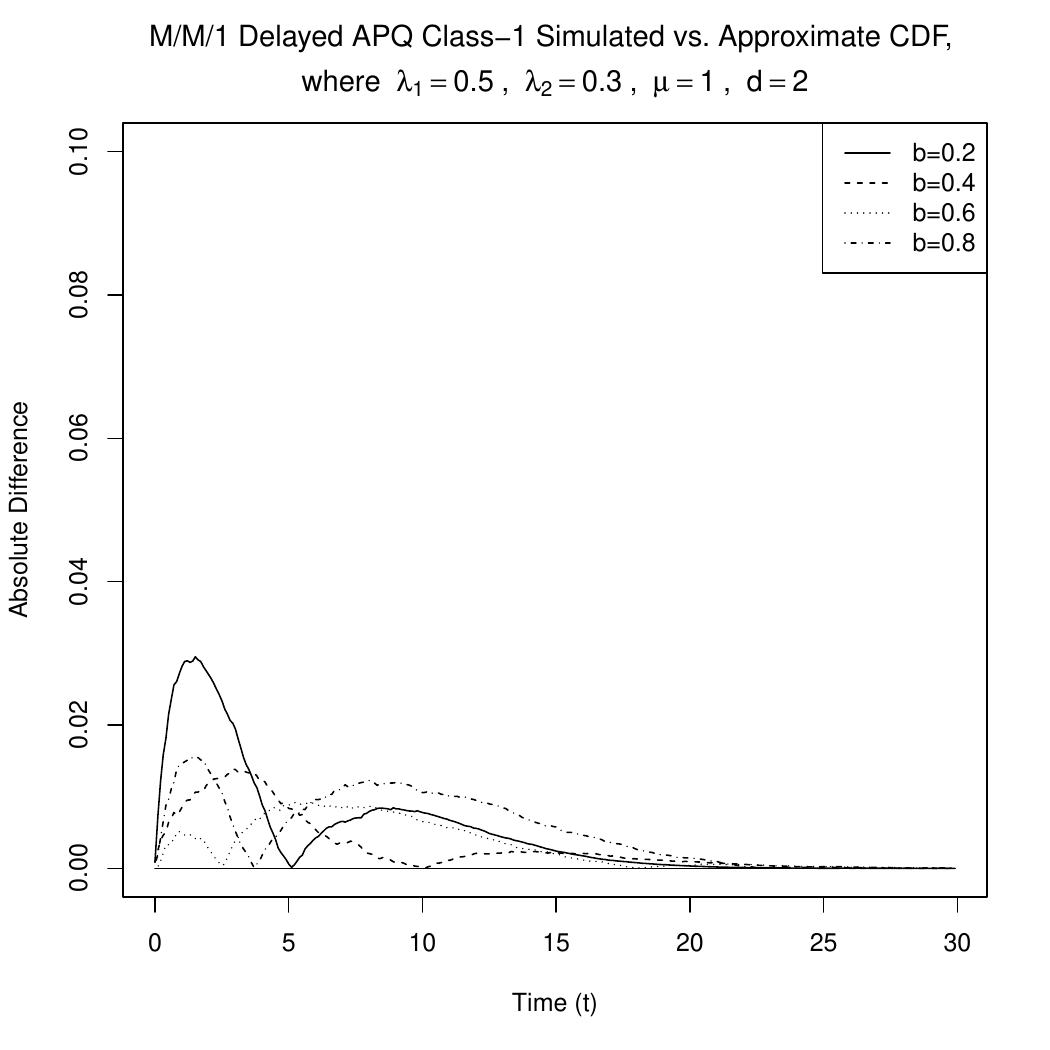}
  }
  \subfloat{
    \includegraphics[width=80mm, height=70mm]{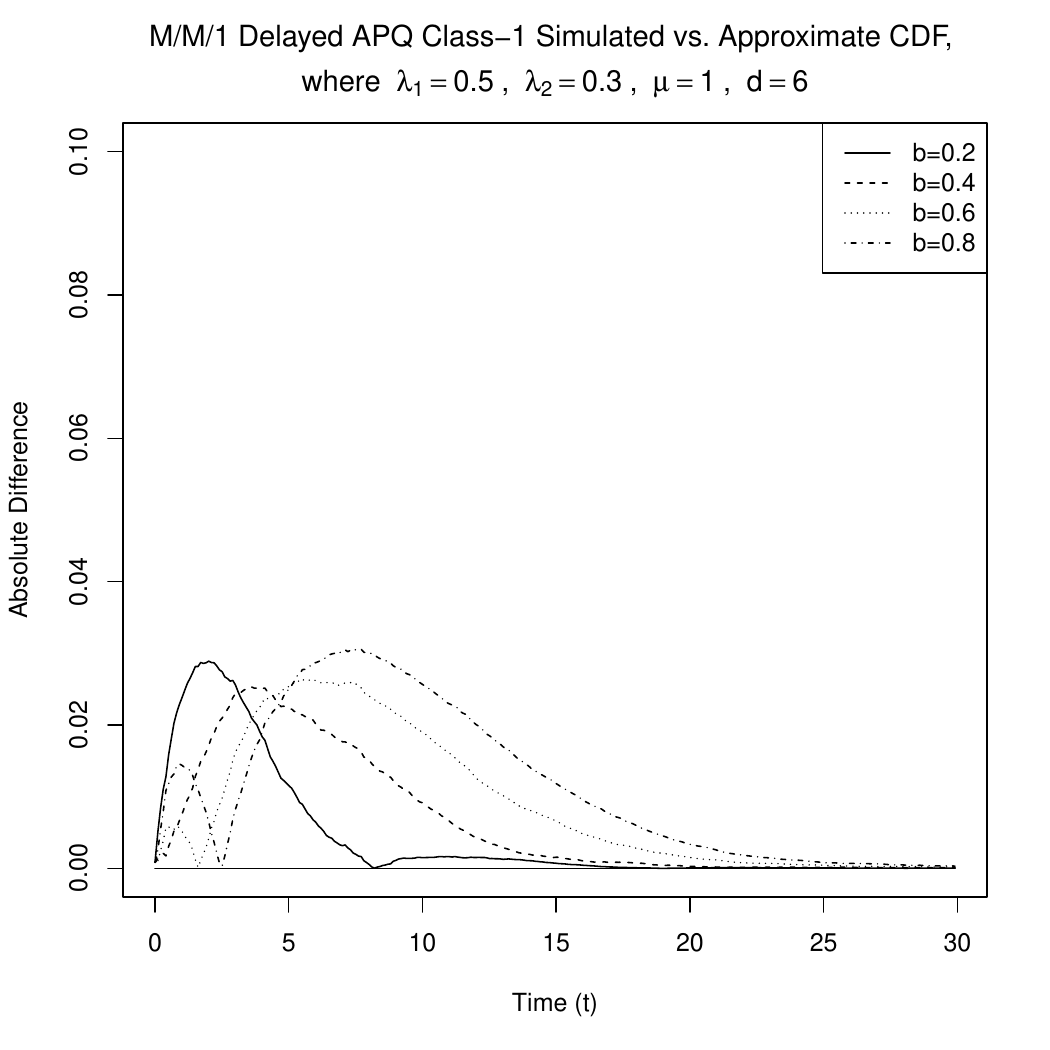}
  }
\\
  \subfloat{
    \includegraphics[width=80mm, height=70mm]{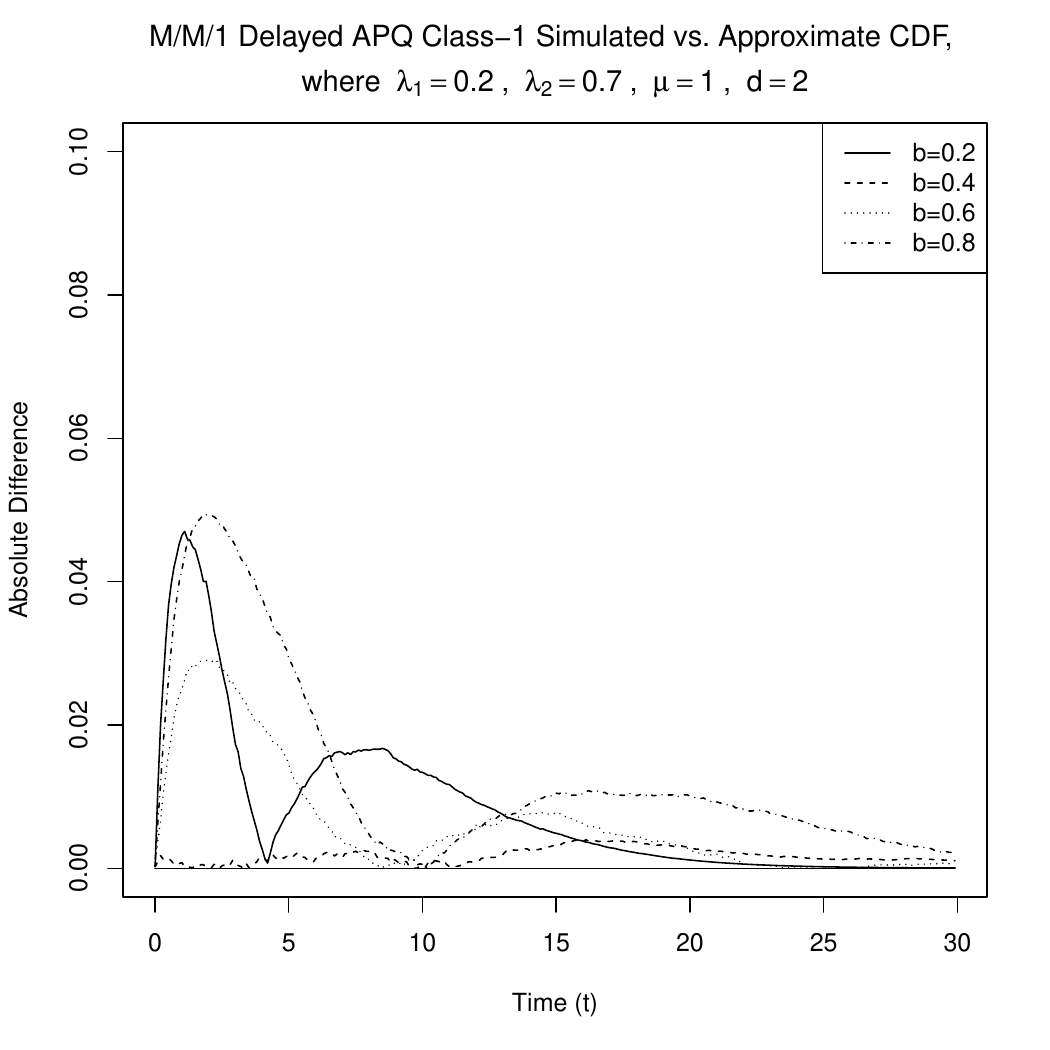}
  }
  \subfloat{
    \includegraphics[width=80mm, height=70mm]{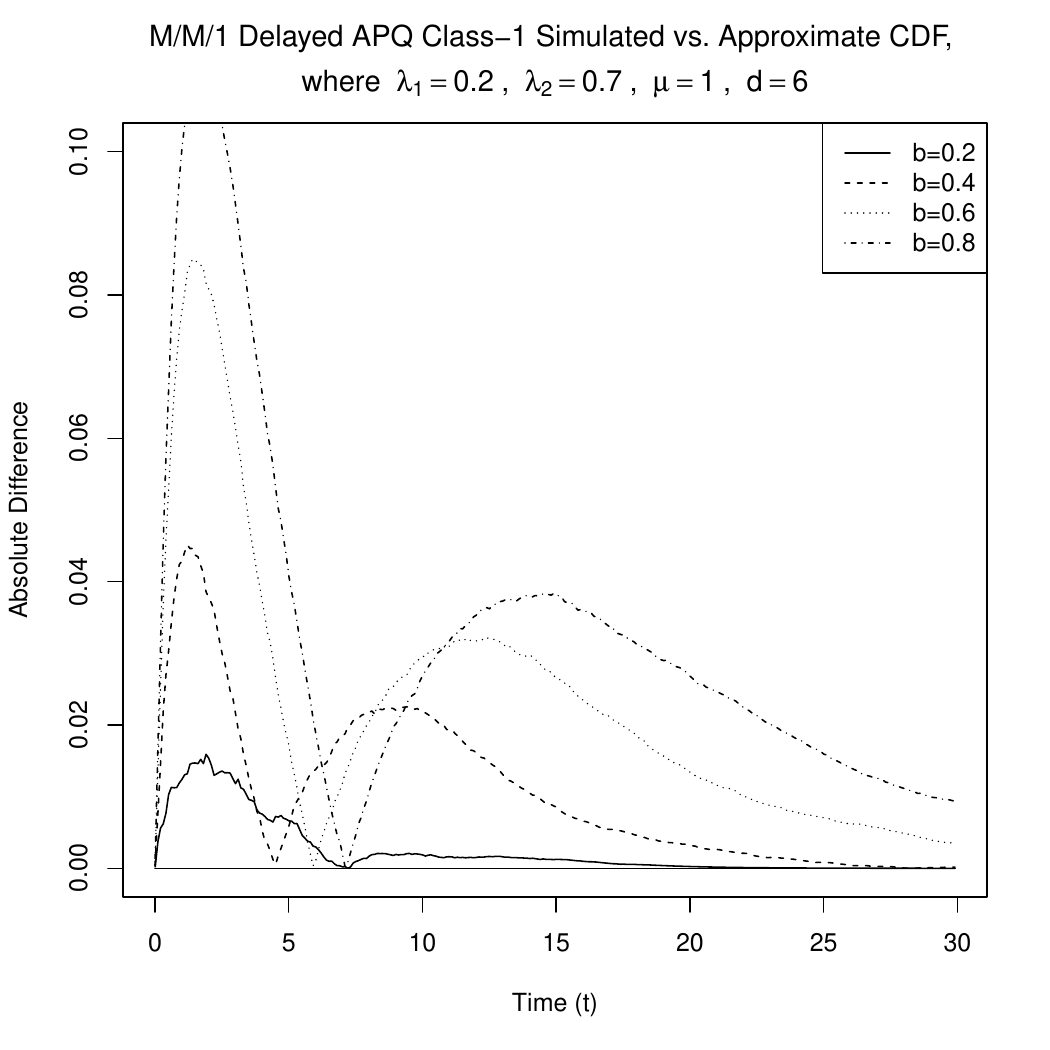}
  }
  \caption{Absolute difference between simulated and exponential approximations of \highname{} waiting time CDFs.}
  \label{fig:diff_simulated_high_rho}
\end{figure}}

\infor{
\begin{figure}[htb]
  \centering
  \subfloat{
    \includegraphics[width=70mm, height=55mm]{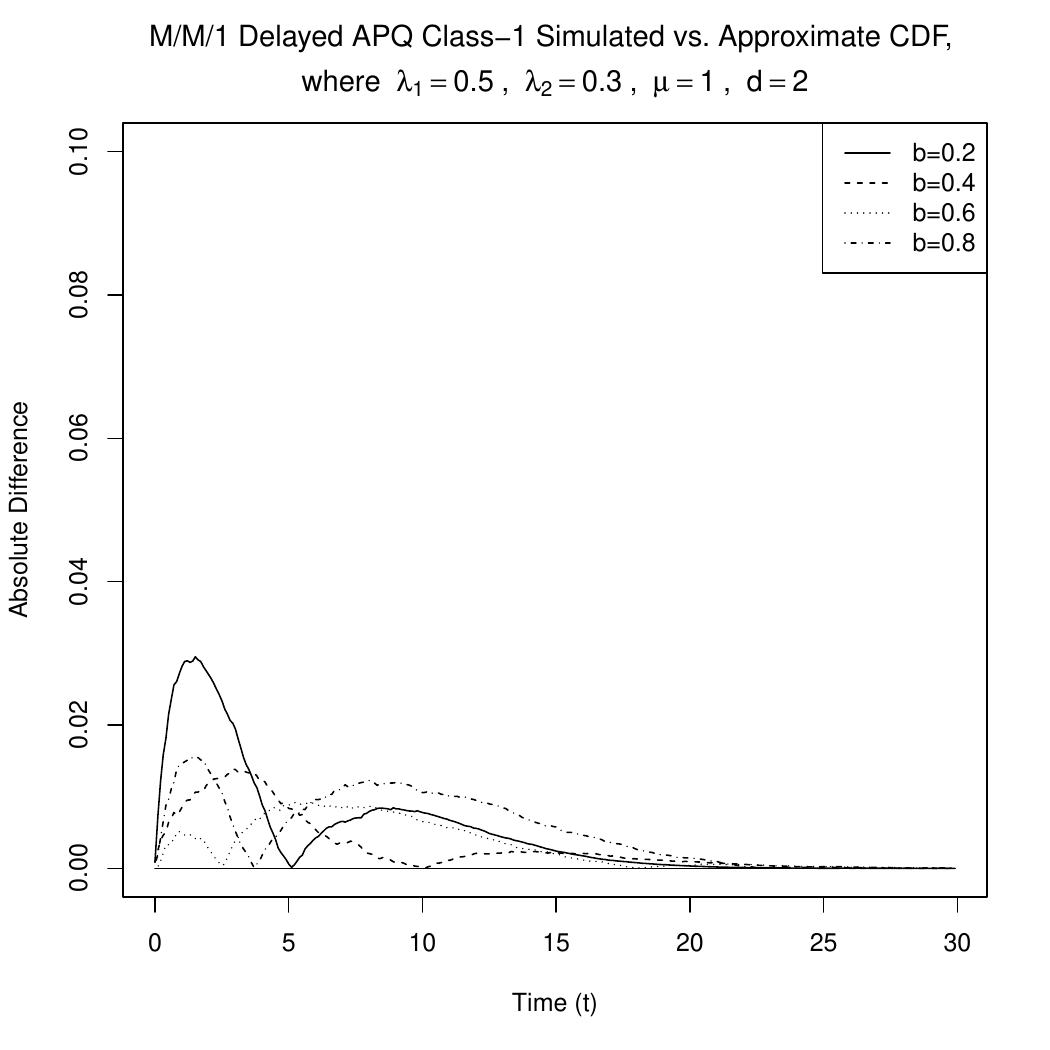}
  }
  \subfloat{
    \includegraphics[width=70mm, height=55mm]{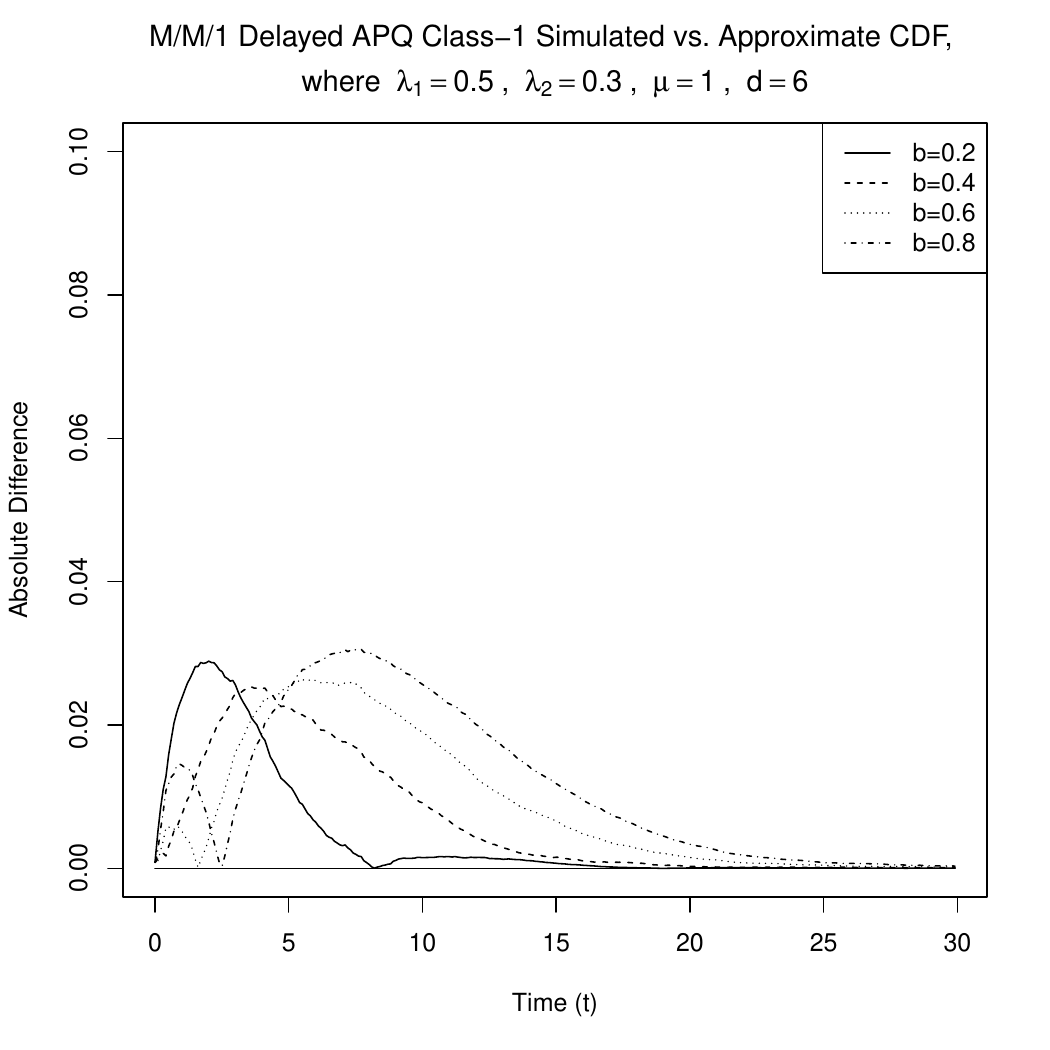}
  }
\\
  \subfloat{
    \includegraphics[width=70mm, height=55mm]{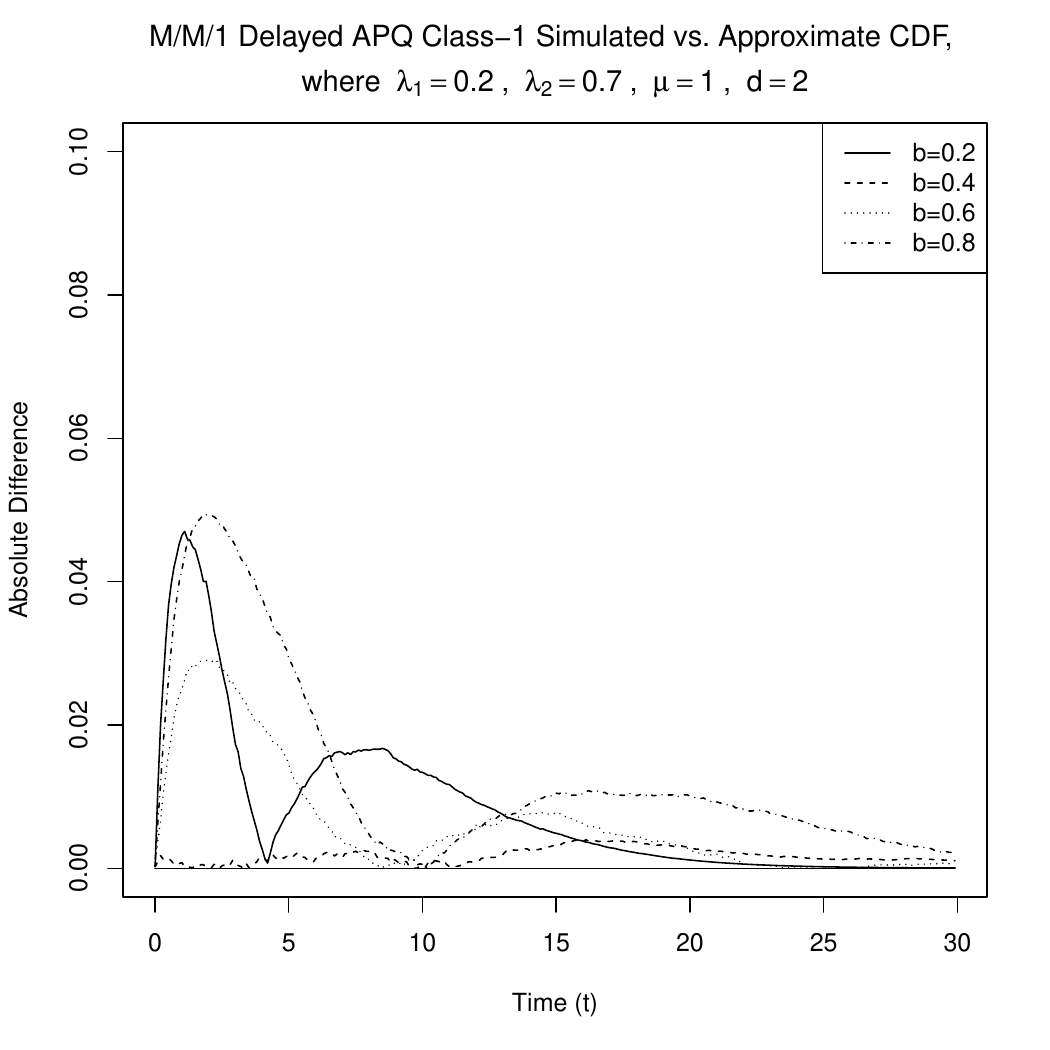}
  }
  \subfloat{
    \includegraphics[width=70mm, height=55mm]{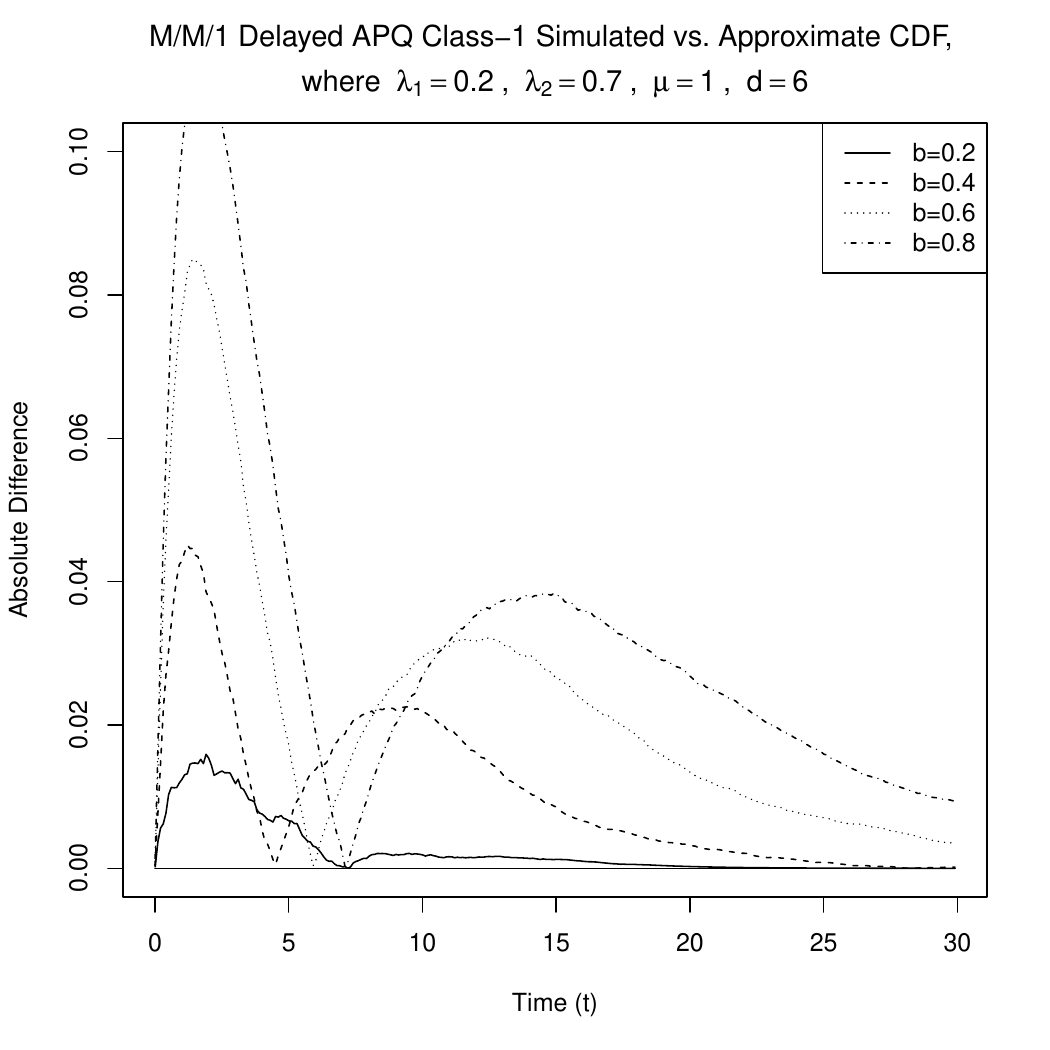}
  }
  \caption{Absolute difference between simulated and exponential approximations of \highname{} waiting time CDFs.}
  \label{fig:diff_simulated_high_rho}
\end{figure}}

\subsection{Optimizing parameters using approximate CDFs}

In this section, we mirror the optimization procedure carried out in \cref{sec:optimizing}, but using the entire (approximate) \highname{} waiting time CDF rather than only the expected value. In \cref{fig:mm1_feasible_bound_class1}, we plot the feasibility region for $(\lam_1,\lam_2)$ pairs that require tuning of $d$ to satisfy a \highname{} KPI constraint, in contrast with the \lowname{} constraint considered in \cref{fig:mm1_feasible_bound}. This feasible region is \emph{exact}, since we do not require an exponential approximation to compute the \fcfsname{} and \npqname{} CDF of \highname{} customers.
Note that for certain \highname{} KPIs, the feasible region takes a diamond shape rather than a simple triangle shape.
This is because we are considering essentially the same targets and compliance probabilities as for \classname{}-2, but these are naturally easier to achieve for \classname{}-1, and thus a larger feasible region becomes visible.

\preprint{
\begin{figure}[htb]
  \centering
  \subfloat{
    \includegraphics[width=52mm, height=55mm]{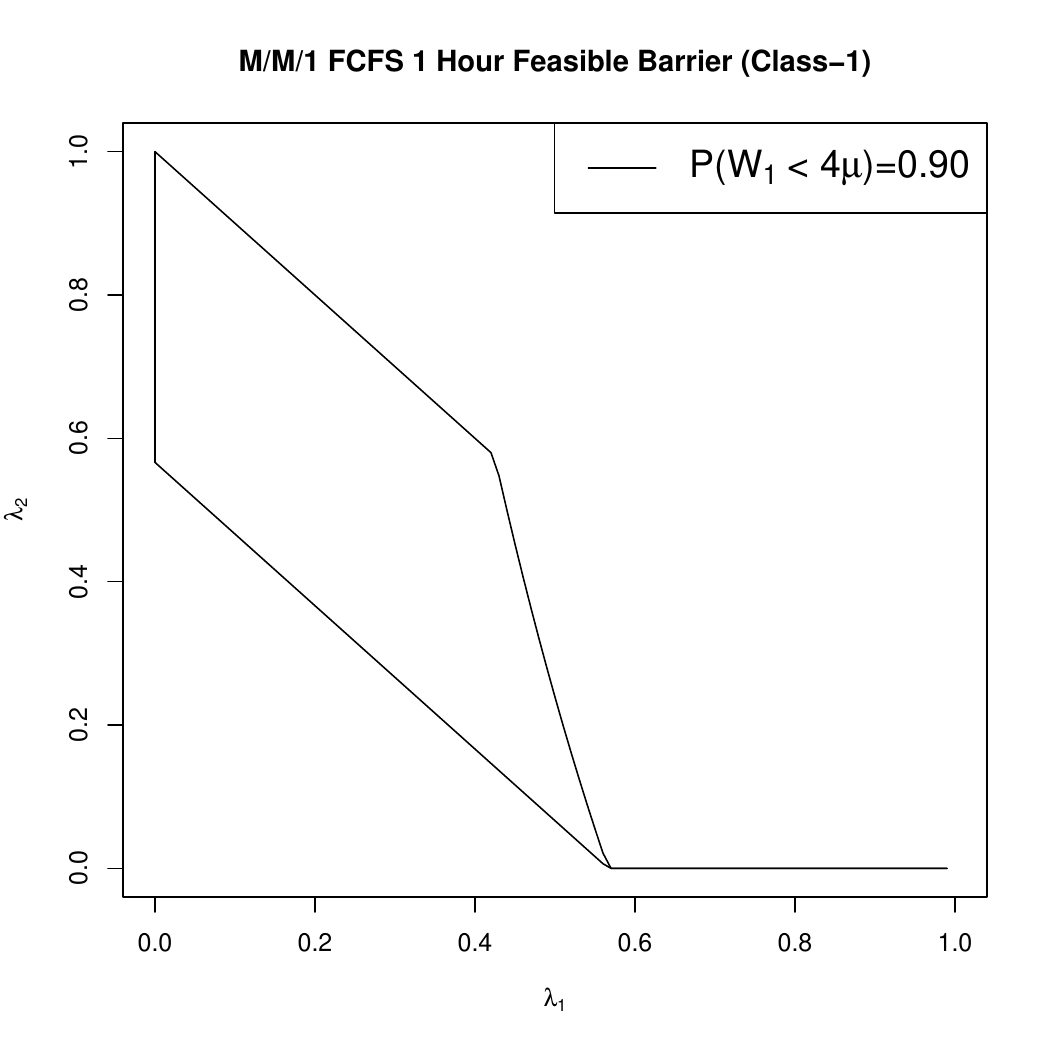}
  }
  \subfloat{
    \includegraphics[width=52mm, height=55mm]{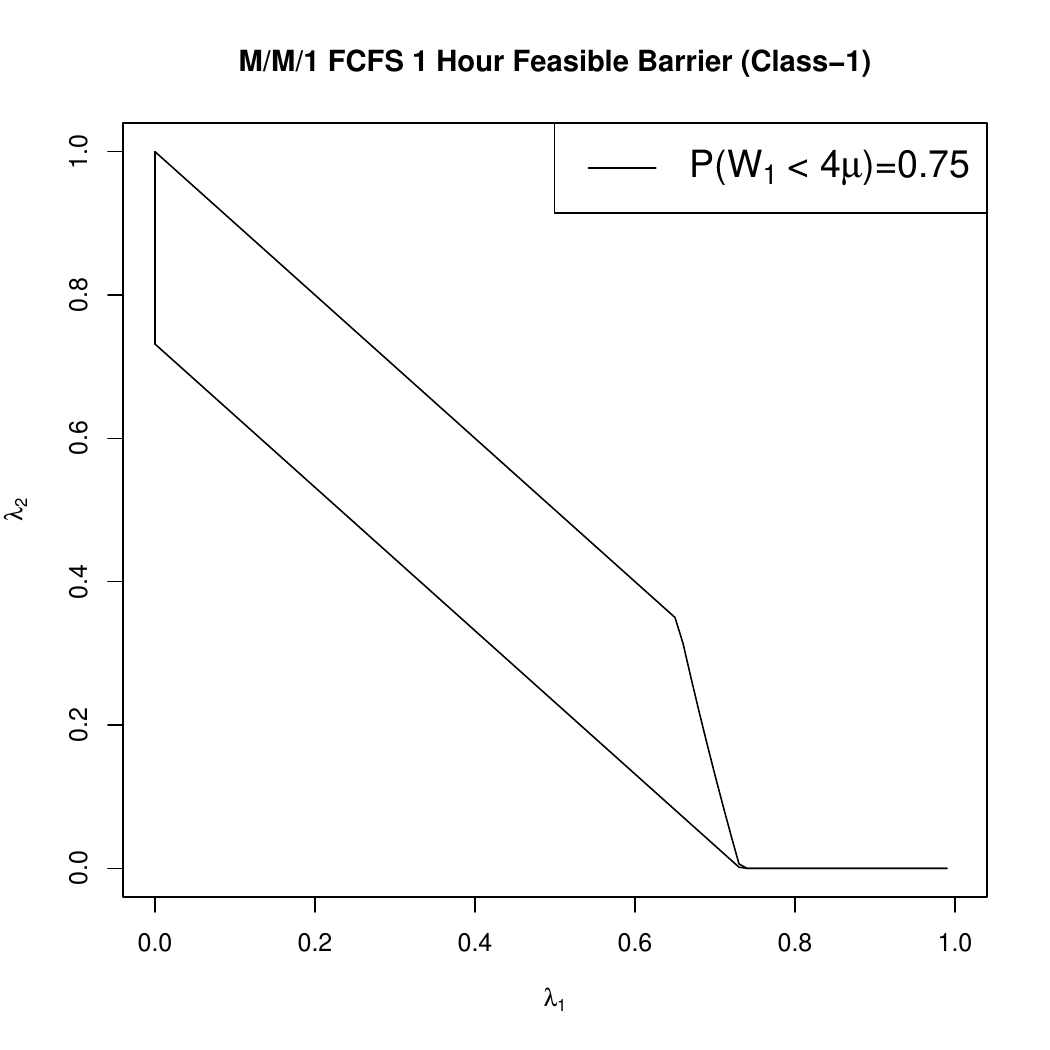}
  }
  \subfloat{
    \includegraphics[width=52mm, height=55mm]{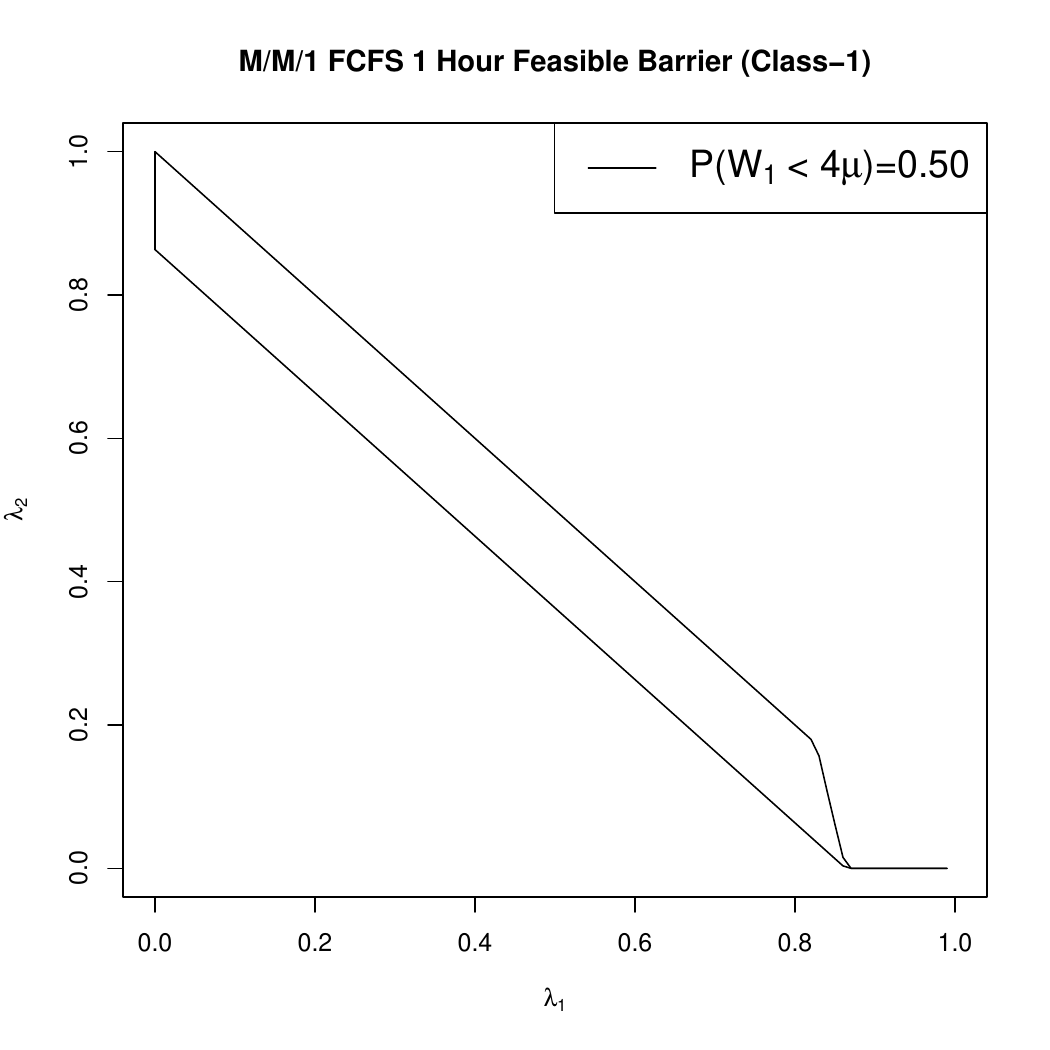}
  }
\\
  \subfloat{
    \includegraphics[width=52mm, height=55mm]{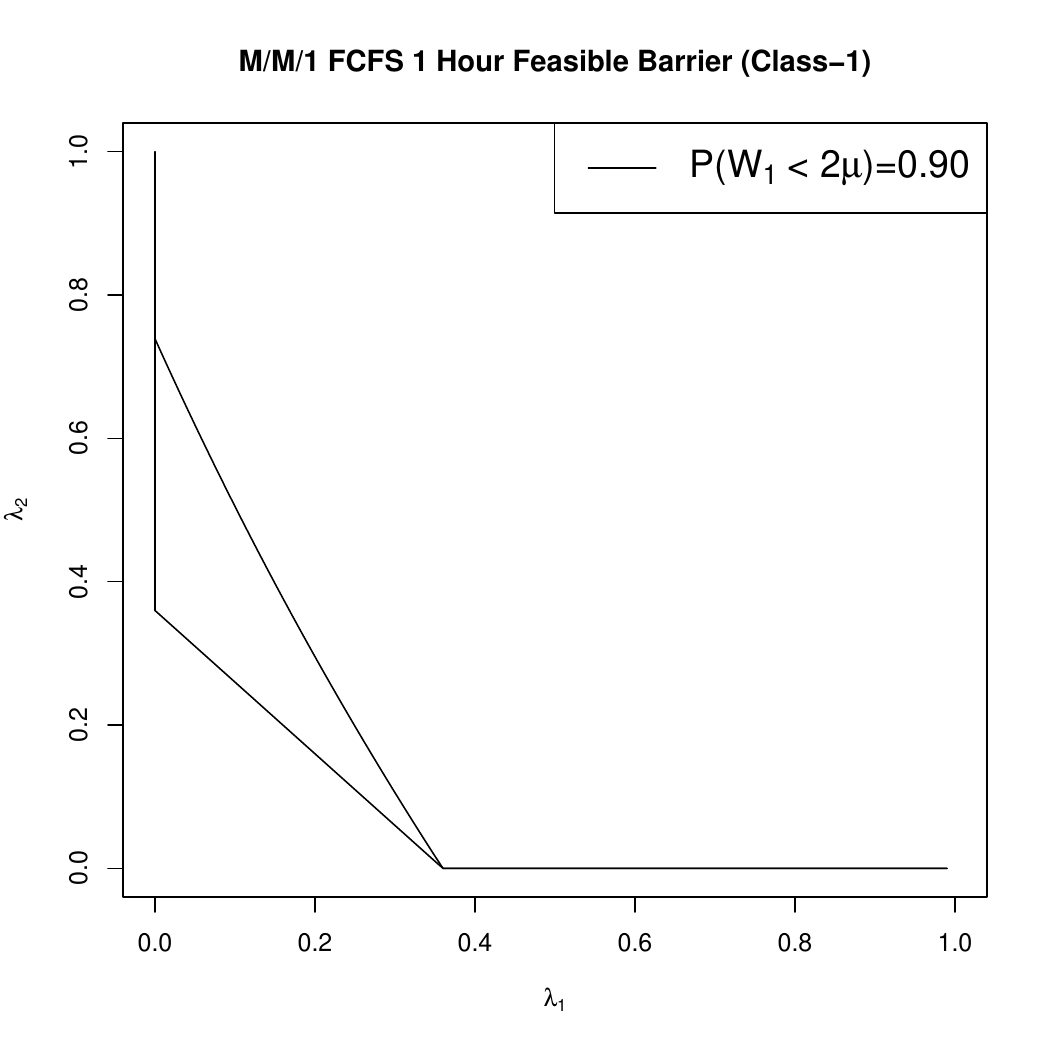}
  }
  \subfloat{
    \includegraphics[width=52mm, height=55mm]{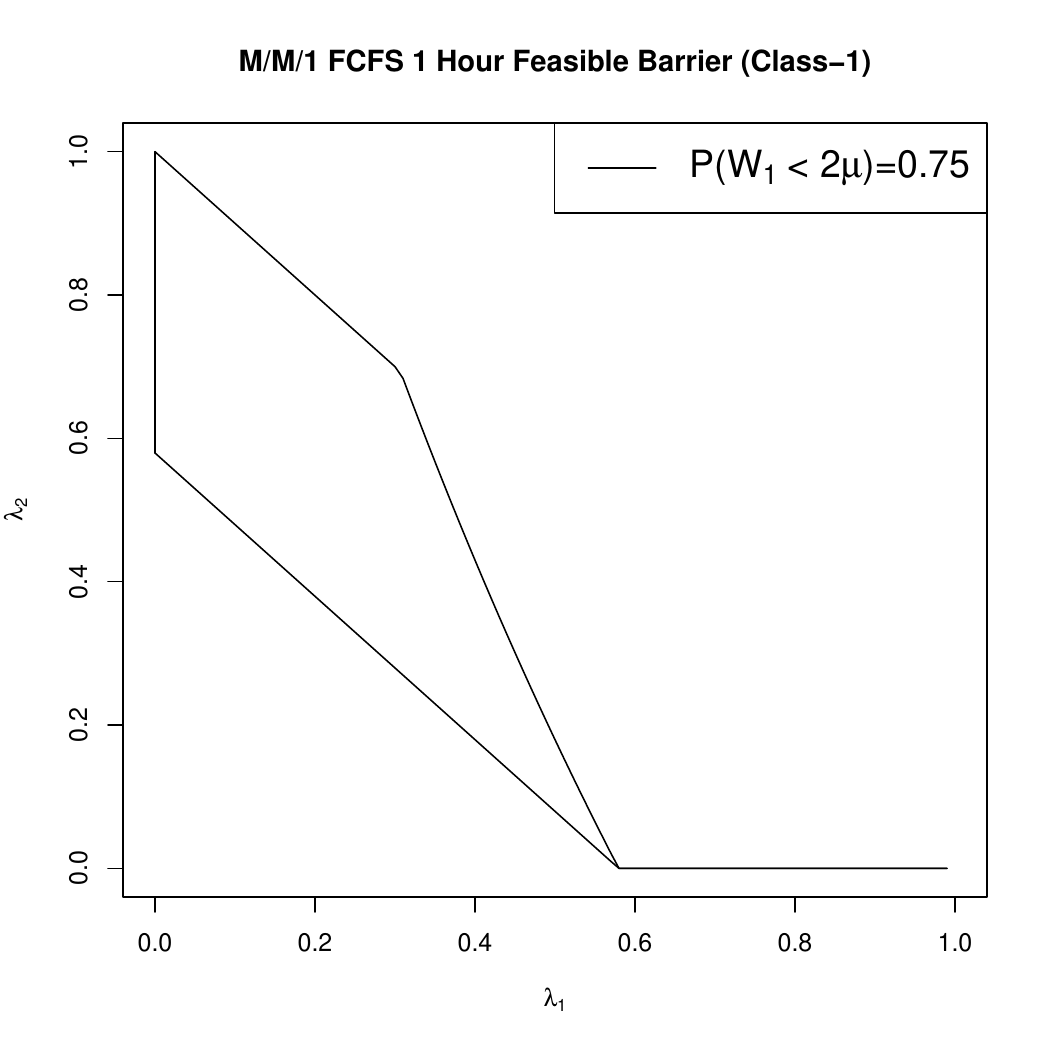}
  }
  \subfloat{
    \includegraphics[width=52mm, height=55mm]{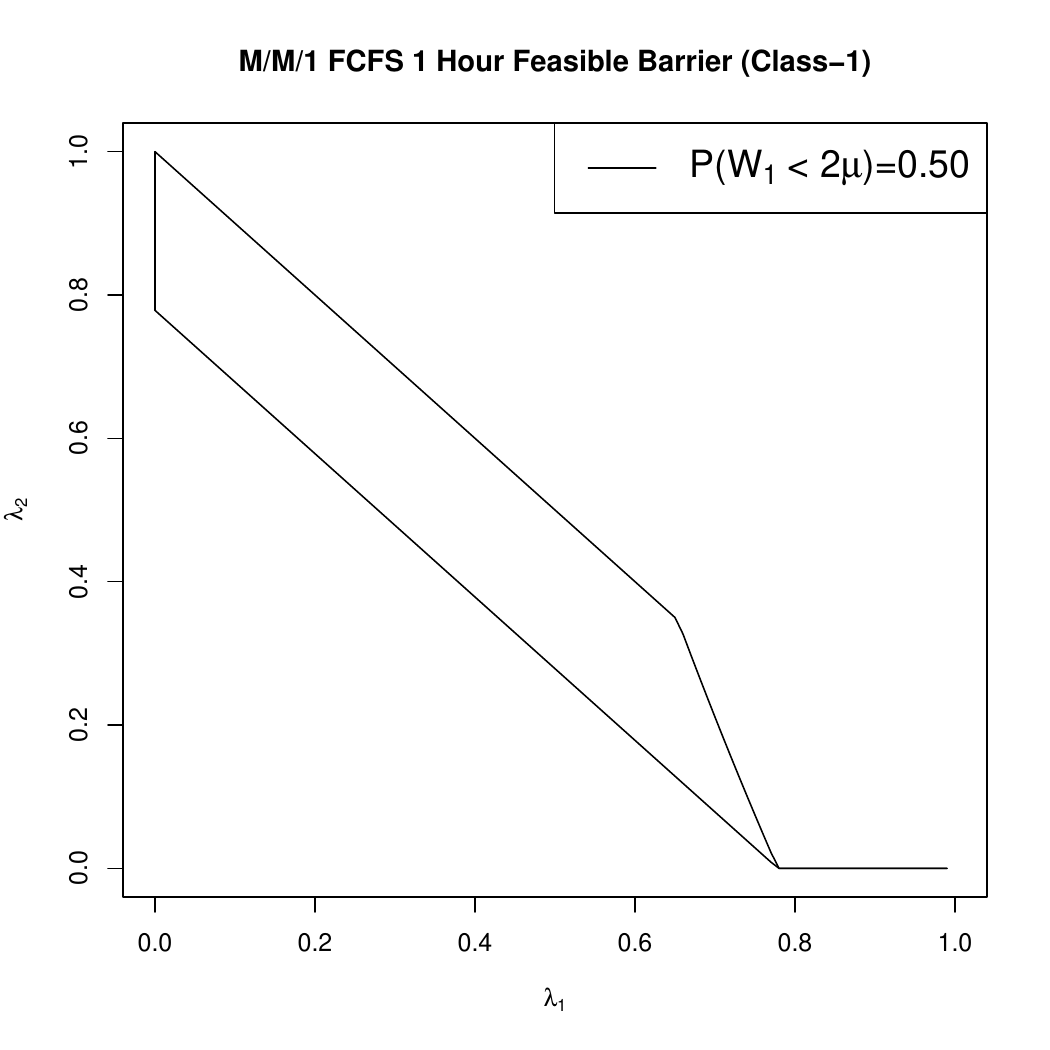}
  }
  \caption{Upper and lower boundaries for $(\lam_1, \lam_2)$ pairs that meet KPI probability for \classname{}-1 waiting time under one hour and half hour and require optimization of $d$.}
  \label{fig:mm1_feasible_bound_class1}
\end{figure}}

\infor{
\begin{figure}[htb]
  \centering
  \subfloat{
    \includegraphics[width=45mm, height=45mm]{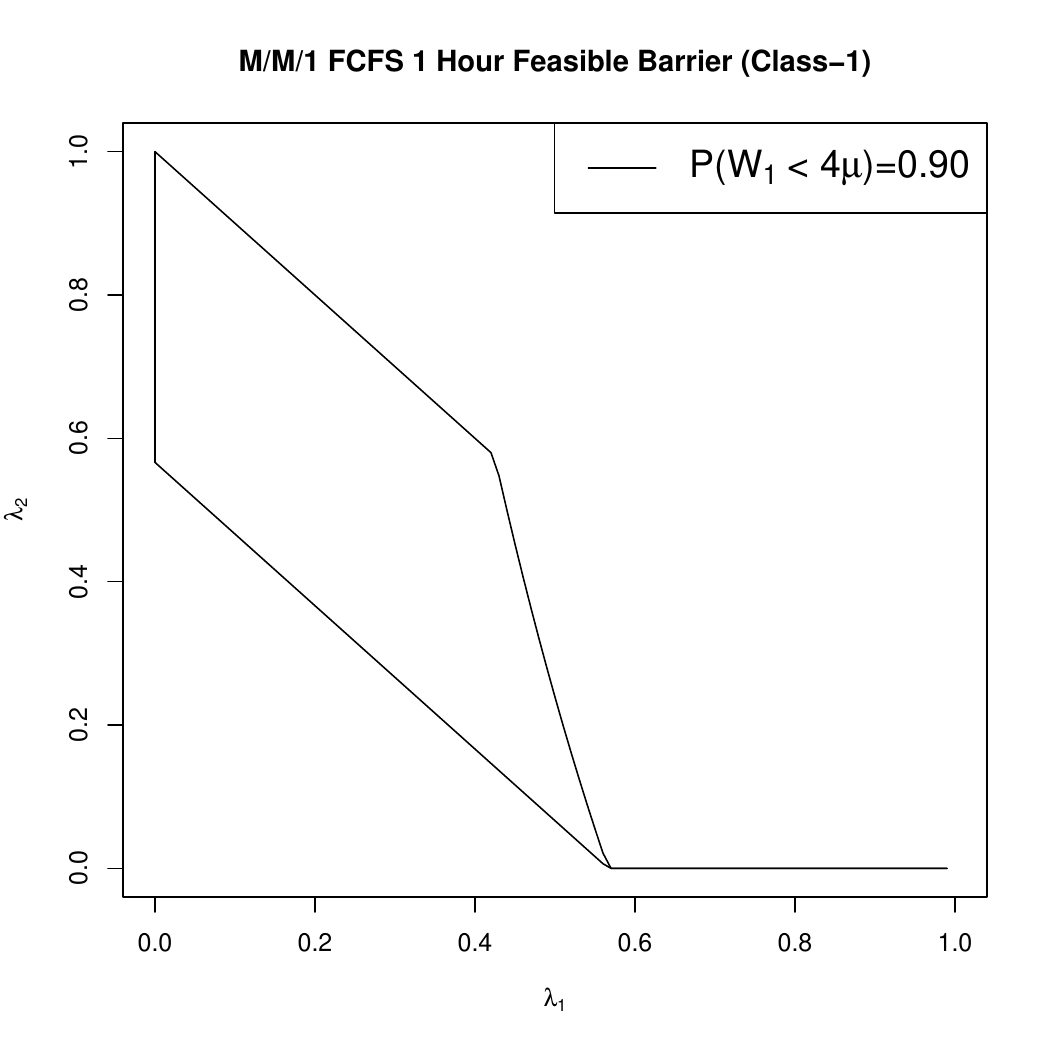}
  }
  \subfloat{
    \includegraphics[width=45mm, height=45mm]{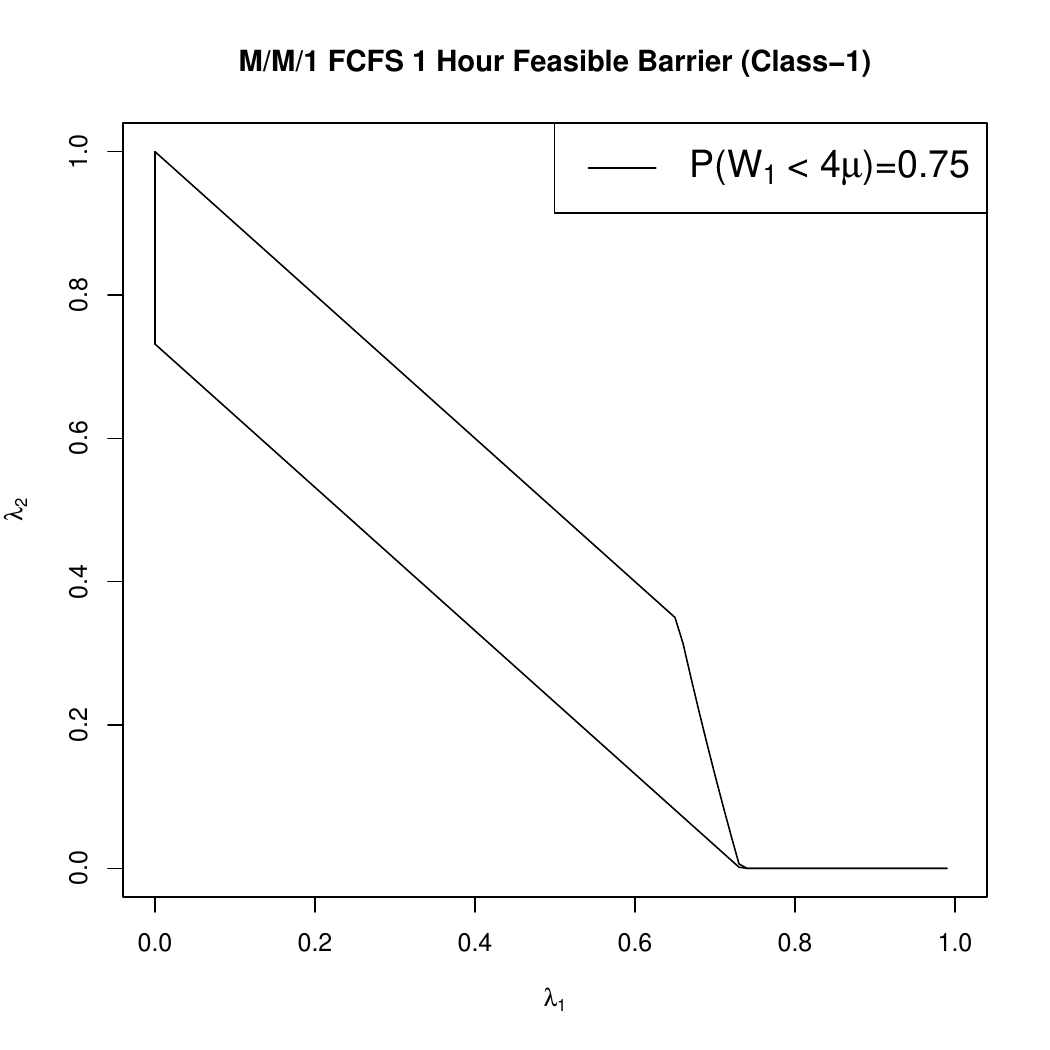}
  }
  \subfloat{
    \includegraphics[width=45mm, height=45mm]{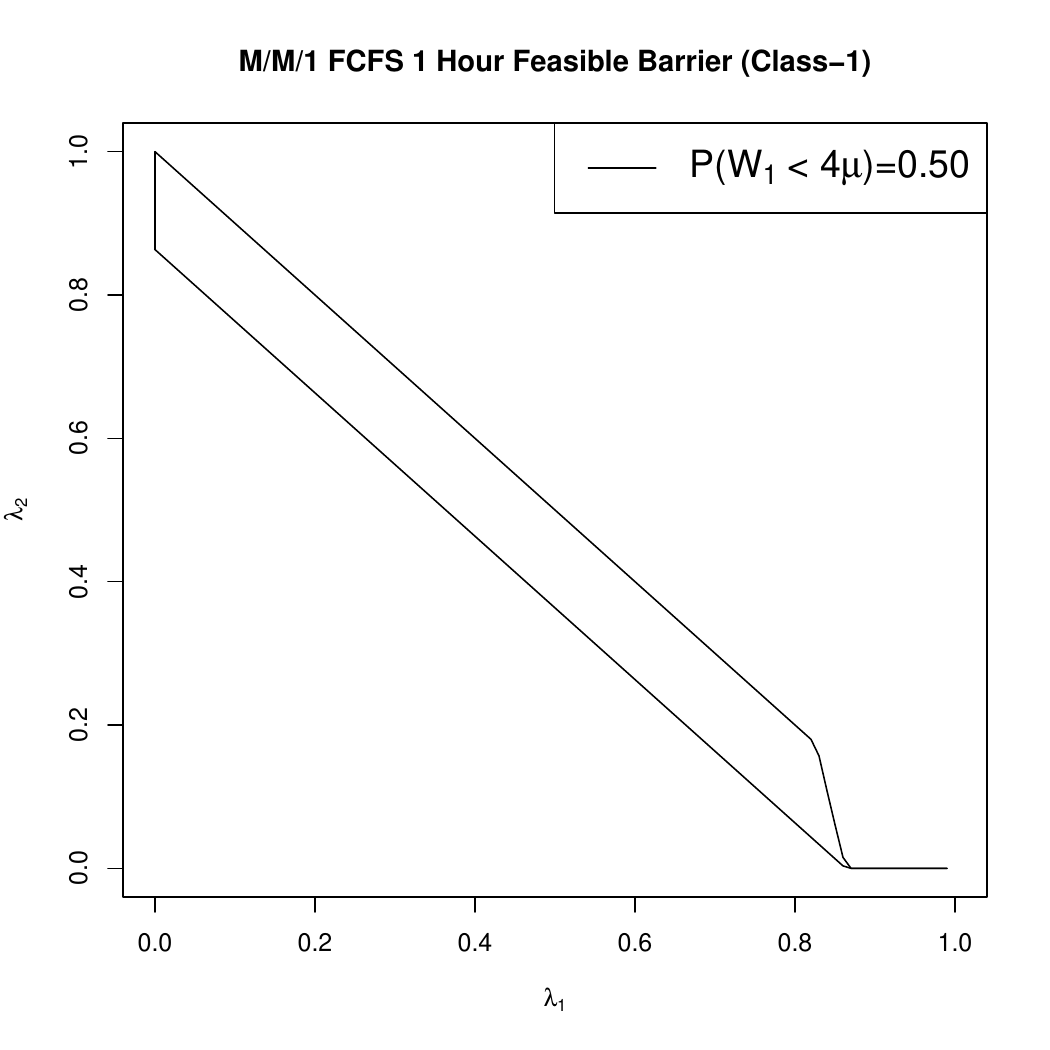}
  }
\\
  \subfloat{
    \includegraphics[width=45mm, height=45mm]{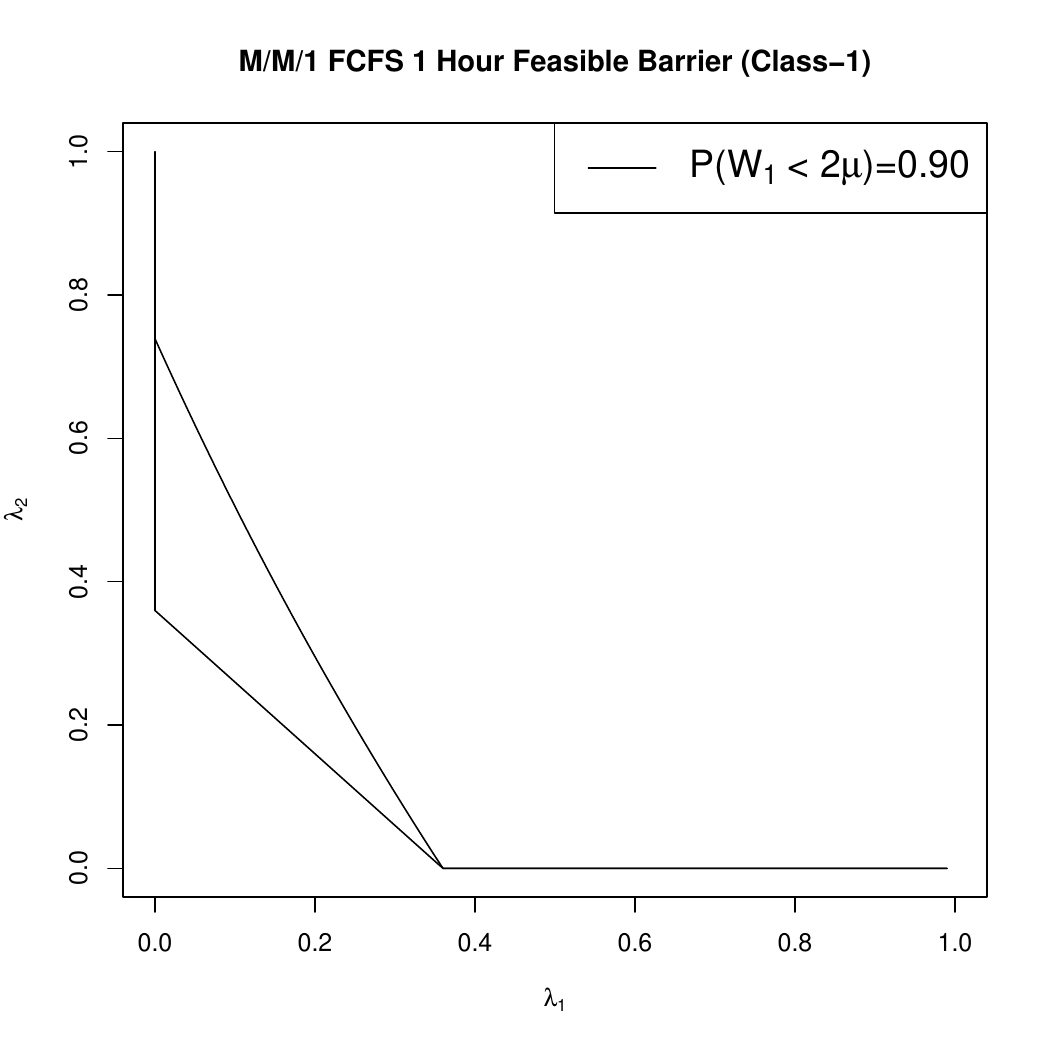}
  }
  \subfloat{
    \includegraphics[width=45mm, height=45mm]{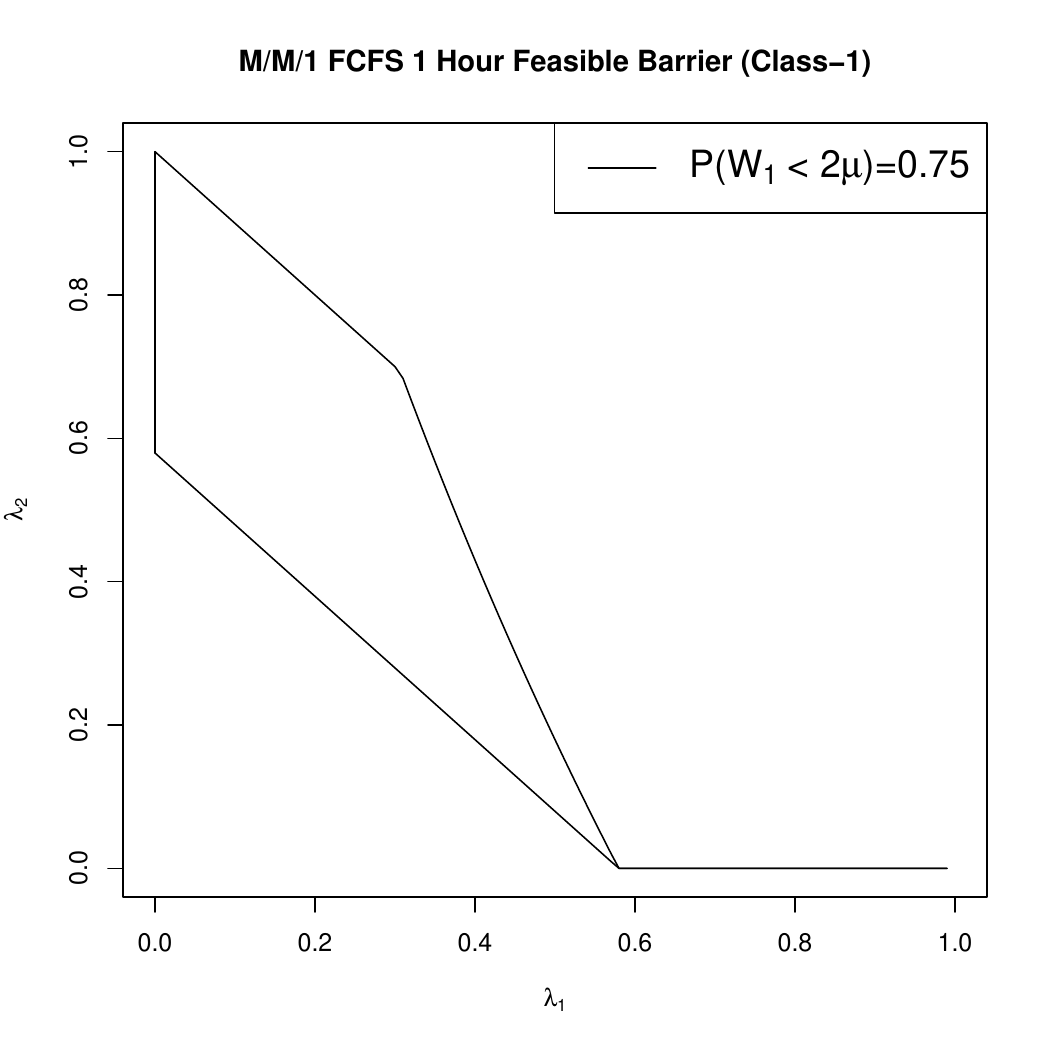}
  }
  \subfloat{
    \includegraphics[width=45mm, height=45mm]{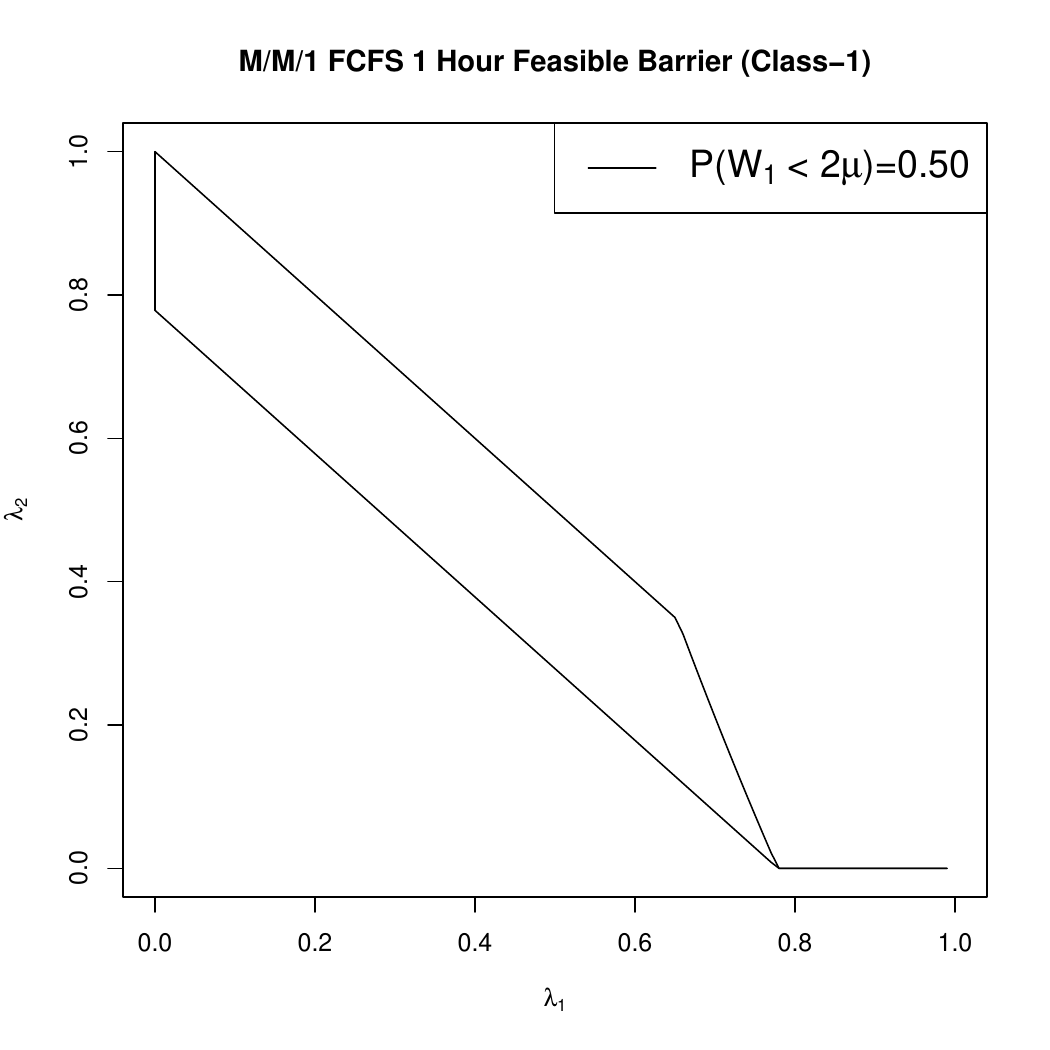}
  }
  \caption{Upper and lower boundaries for $(\lam_1, \lam_2)$ pairs that meet KPI probability for \classname{}-1 waiting time under one hour and half hour and require optimization of $d$.}
  \label{fig:mm1_feasible_bound_class1}
\end{figure}}

The roles of the \fcfsname{} and the \npqname{} are reversed in \cref{fig:mm1_feasible_bound_class1} from \cref{fig:mm1_feasible_bound}. In particular, the \npqname{} is \emph{more} favourable to \classname{}-1 than the \fcfsname{}, and consequently the lower boundary corresponds to when the KPI can be met even in the \fcfsname{} while the upper boundary corresponds to when the KPI cannot be met even in the \npqname{}. By inspection, one can see that the overlap between the regions requiring optimization of $d$ for the advocated KPIs of $\PP(\waitD_2 < 4) \geq 0.85$ and $\PP(\waitD_1 < 2) \geq 0.9$ is nearly negligible; it makes up a sliver of a triangle for $\lam_1$ between $0$ and $0.1$ and $\lam_2$ between $0.55$ and $0.65$. Nonetheless, this specific KPI pair is only a suggestion from a certain context, so we proceed with exploring optimal $(d,b^*(d))$ pairs using the \classname-1 feasible regions.

We now extend \cref{fig:mm1_kpi_examples} using our approximation of the \classname{}-1 waiting time CDF. To do so, we plot the expected \classname{}-2 waiting times for $b^*(d)$ chosen to optimize this expectation subject to still meeting the \classname{}-1 constraint. More specifically, we use
\[
  b^*(d) = \max \left\{b : b \in [0,1], \cdfDparams{d}{b}_1(w) \geq p \right\},
\]
in conjunction with our zero-inflated exponential approximation to $\cdfD_1$. 

In \cref{fig:mm1_kpi_class1_examples_bvals}, we see the expected pattern that $b^*(d)$ is monotonic in $d$. 
As a consequence of our exponential approximation, the KPI at target $w$ and compliance probability $p$ is achieved if and only if\footnote{This follows from rearranging \cref{eqn:zexp-def} with $\alpha=\rho/\EE \waitD_1$.}
\[\label{eqn:kpi-mean}
  \EE \waitD_1 \leq \frac{w\rho}{\log(\rho/(1-p))}.
\]
Thus, since $b^*(d)$ is the largest $b$ that achieves the KPI for a fixed $d$, the \classname{}-1 expected waiting time under $(d,b^*(d))$ is always the constant value equal to the RHS of \cref{eqn:kpi-mean}.
If it were not, there would be some slack in the KPI (i.e., $\EE \waitD_1$ would be \emph{strictly} smaller than the RHS), and thus $b^*(d)$ could be taken larger while still achieving the KPI, which contradicts $b^*(d)$ being the maximum such value.  
By the conservation law, this also keeps the \classname{}-2 expected waiting time constant, which we observe in \cref{fig:mm1_kpi_class1_examples}.

To investigate whether the fact that $(d, b^*(d))$ keeps $\EE \waitD_1$ constant is only an artefact of our exponential approximation, we also computed full \classname{}-2 waiting time CDFs for $(d,b^*(d))$ pairs.
The CDFs were all numerically indistinguishable (up to 7 decimal places) within the same $(\lam_1,\lam_2)$ pair, 
which suggests that optimizing to obtain $b^*(d)$ not only results in constant expected waiting time, but also constant higher-order moments of the \classname{}-2 waiting time.
Thus, we arrive at the same conclusion as we did for \classname{}-2 KPIs.
Specifically, when optimizing the queueing parameters to minimize \classname{}-2 expected waiting time subject to meeting the \classname{}-1 KPI, there is no benefit to having access to the \dapqname{} beyond the \apqname{}. However, we do note that since the \classname{}-2 waiting time is essentially constant as a function of $d$, selecting a \dapqname{} over an \apqname{} may have practical differences that make it preferable.

\preprint{
\begin{figure}[htb]
  \centering
  \subfloat{
    \includegraphics[width=52mm, height=55mm]{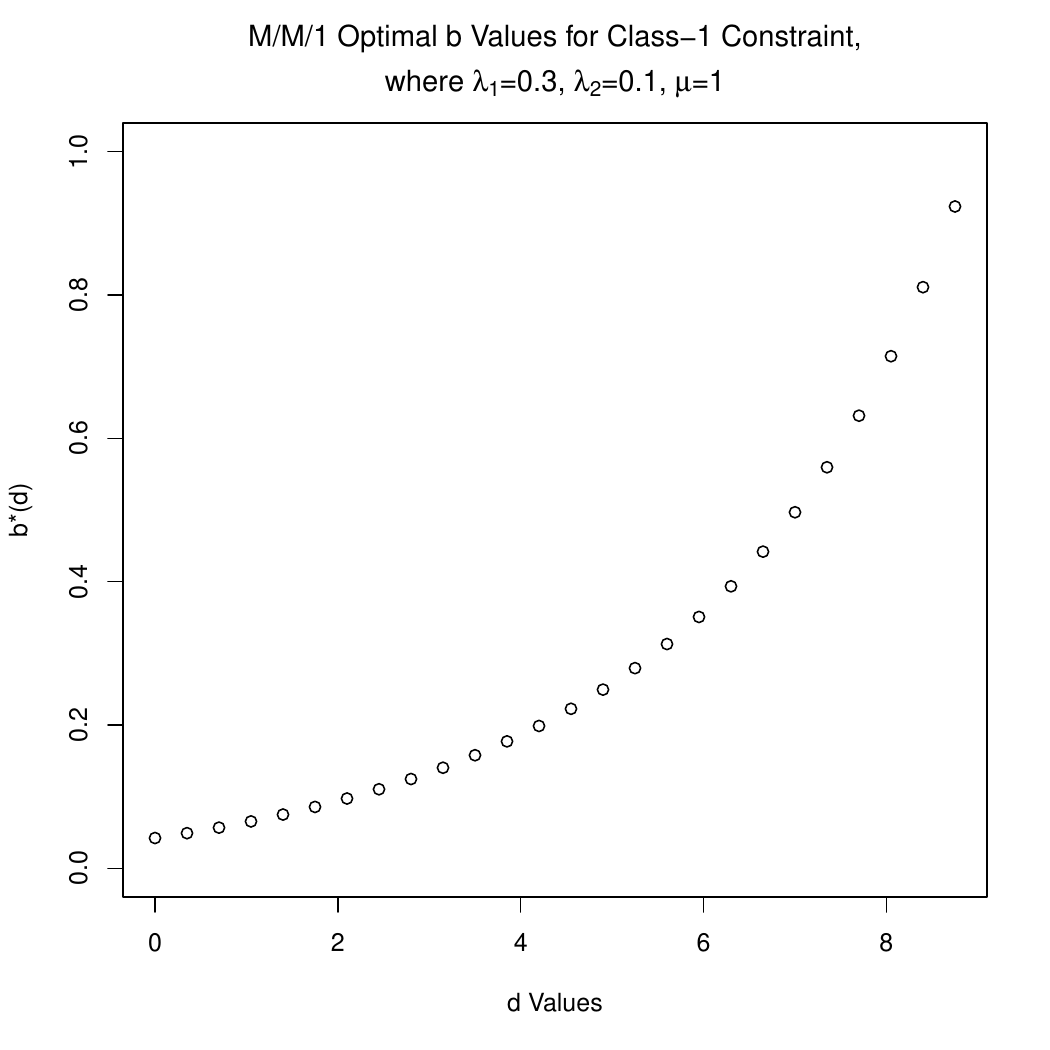}
  }
  \subfloat{
    \includegraphics[width=52mm, height=55mm]{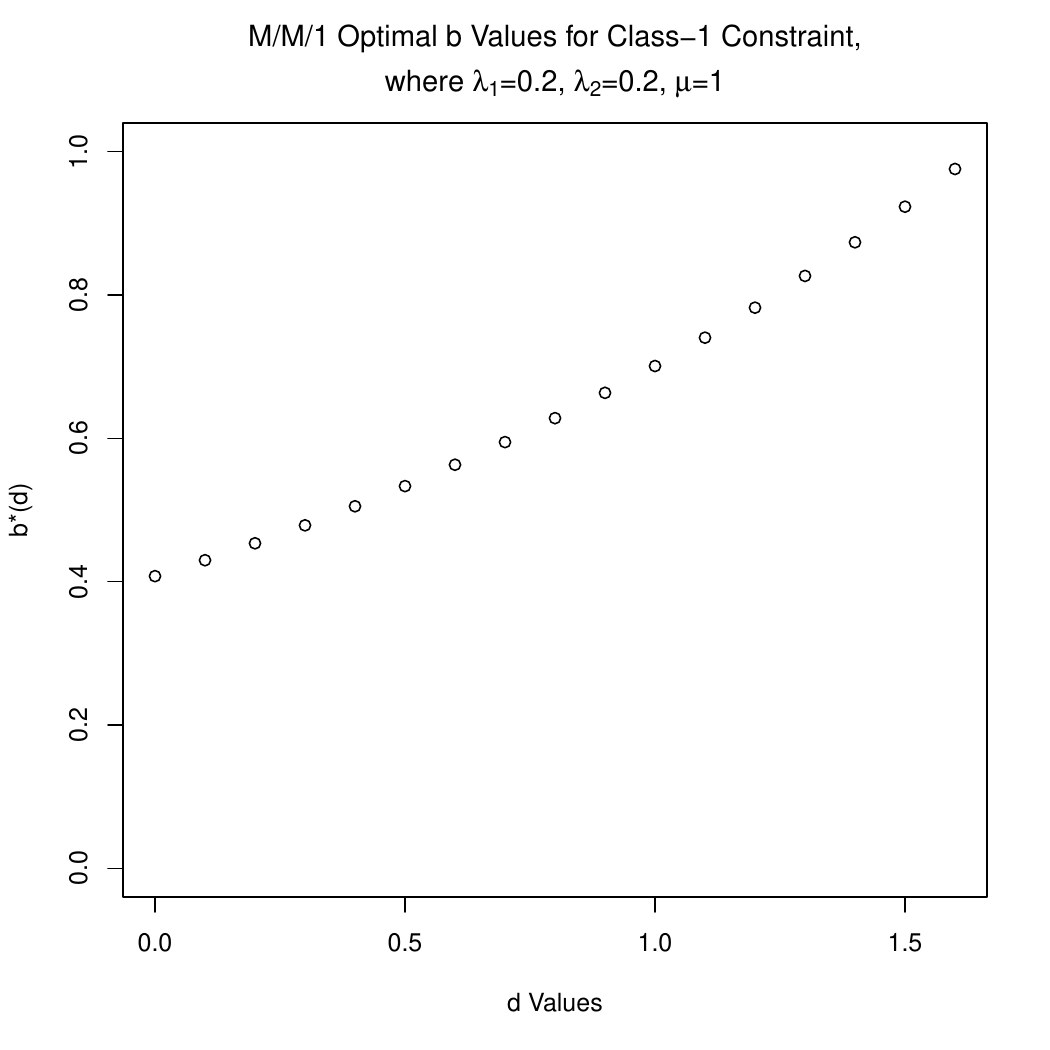}
  }
  \subfloat{
    \includegraphics[width=52mm, height=55mm]{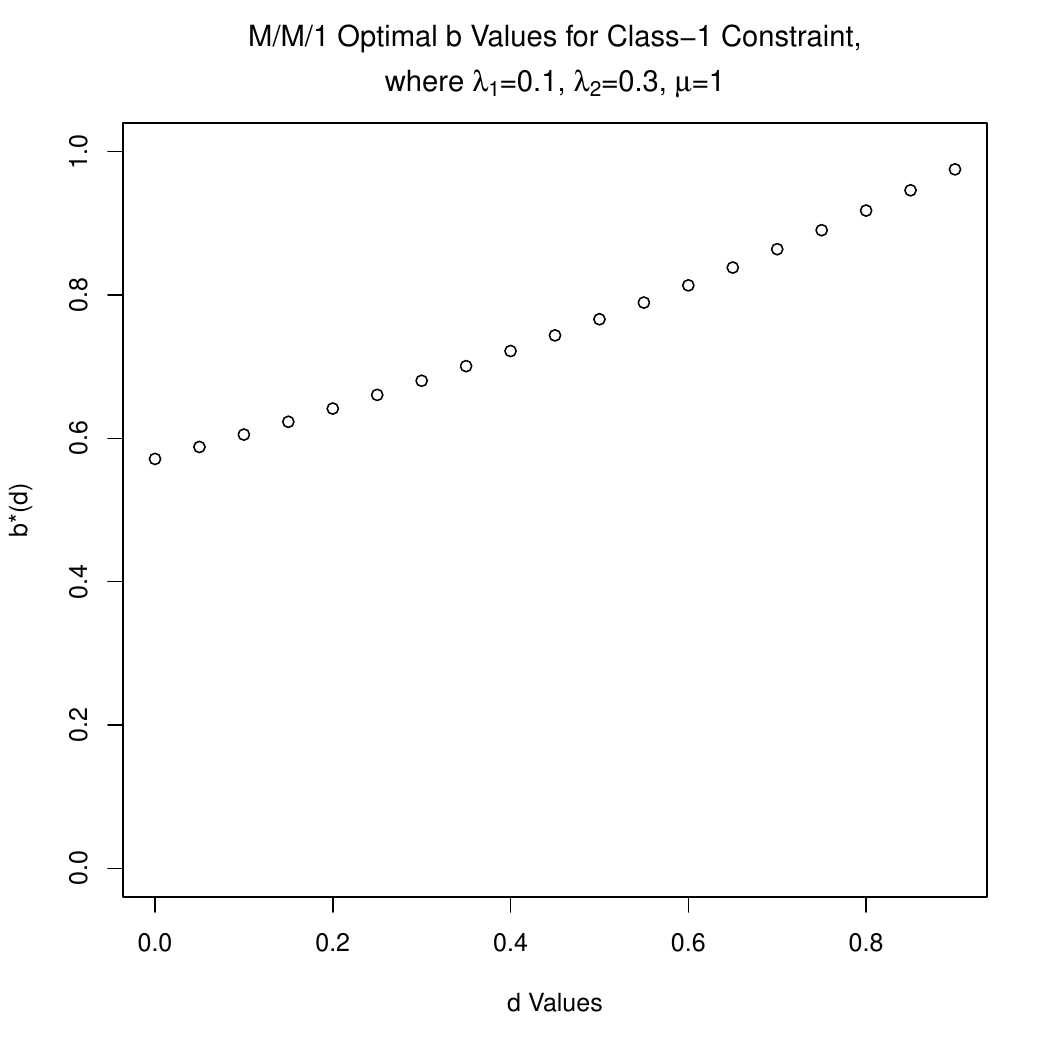}
  }
  \caption{Optimal $b^*(d)$ for delay level $d$ at various occupancy levels for the KPI $\PP(\waitD_1 < 2) \geq 0.9$.}
  \label{fig:mm1_kpi_class1_examples_bvals}
\end{figure}
\begin{figure}[htb]
  \centering
  \subfloat{
    \includegraphics[width=52mm, height=55mm]{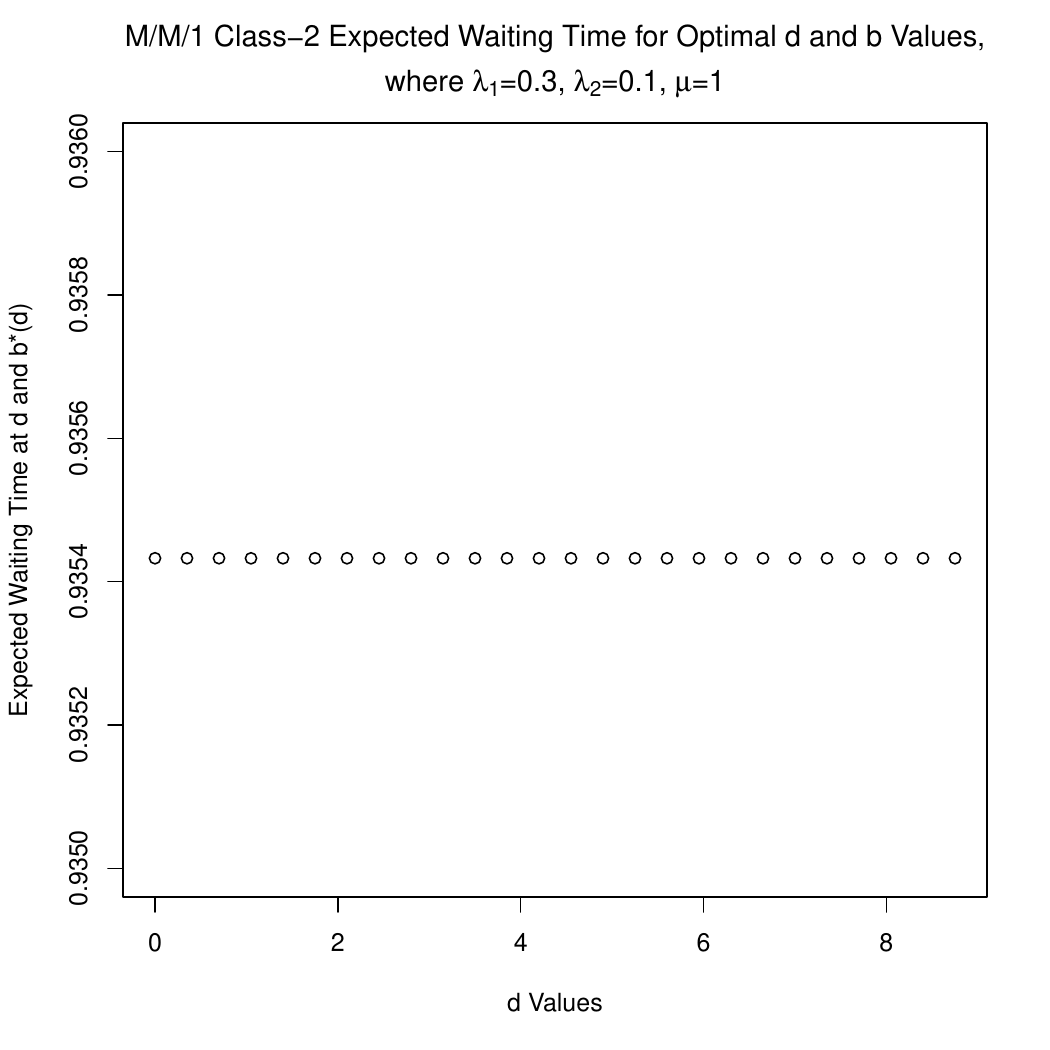}
  }
  \subfloat{
    \includegraphics[width=52mm, height=55mm]{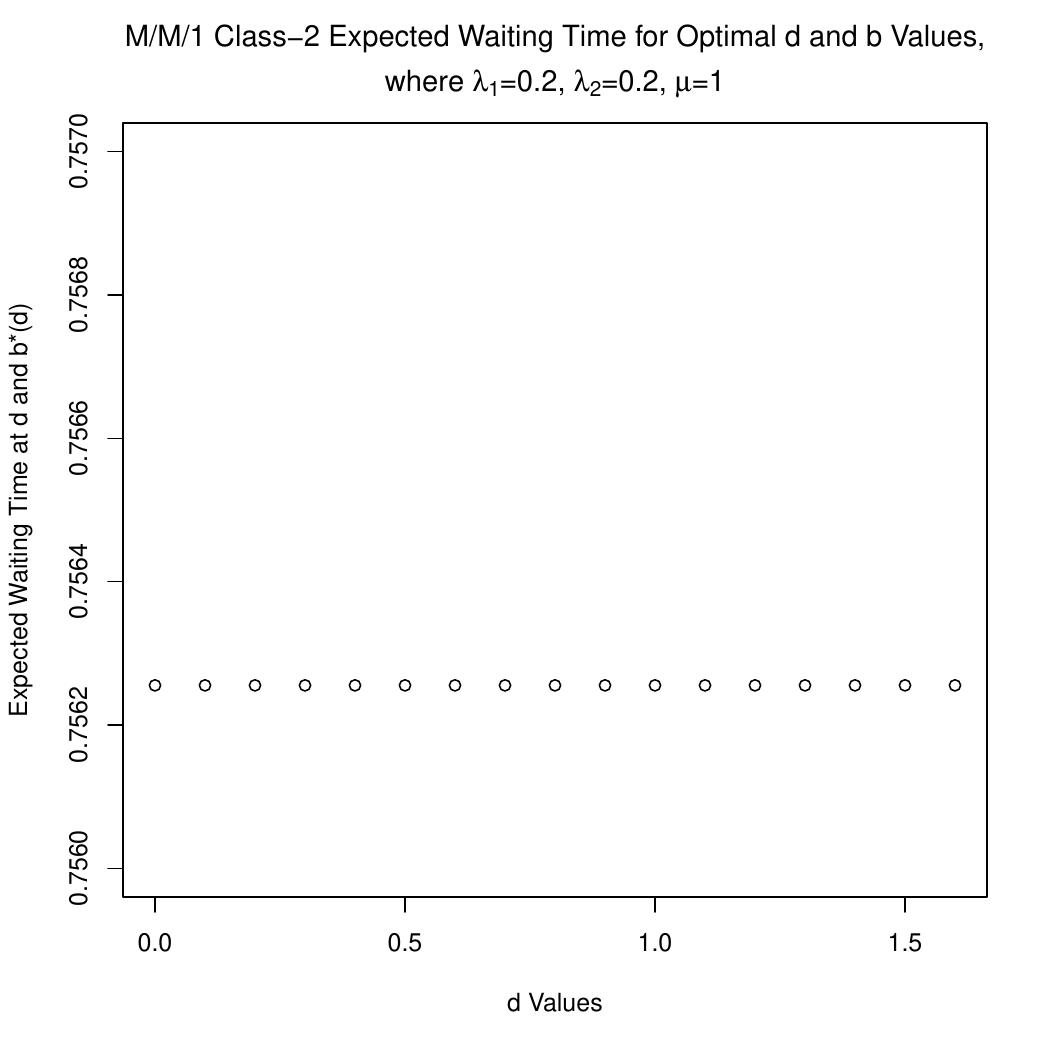}
  }
  \subfloat{
    \includegraphics[width=52mm, height=55mm]{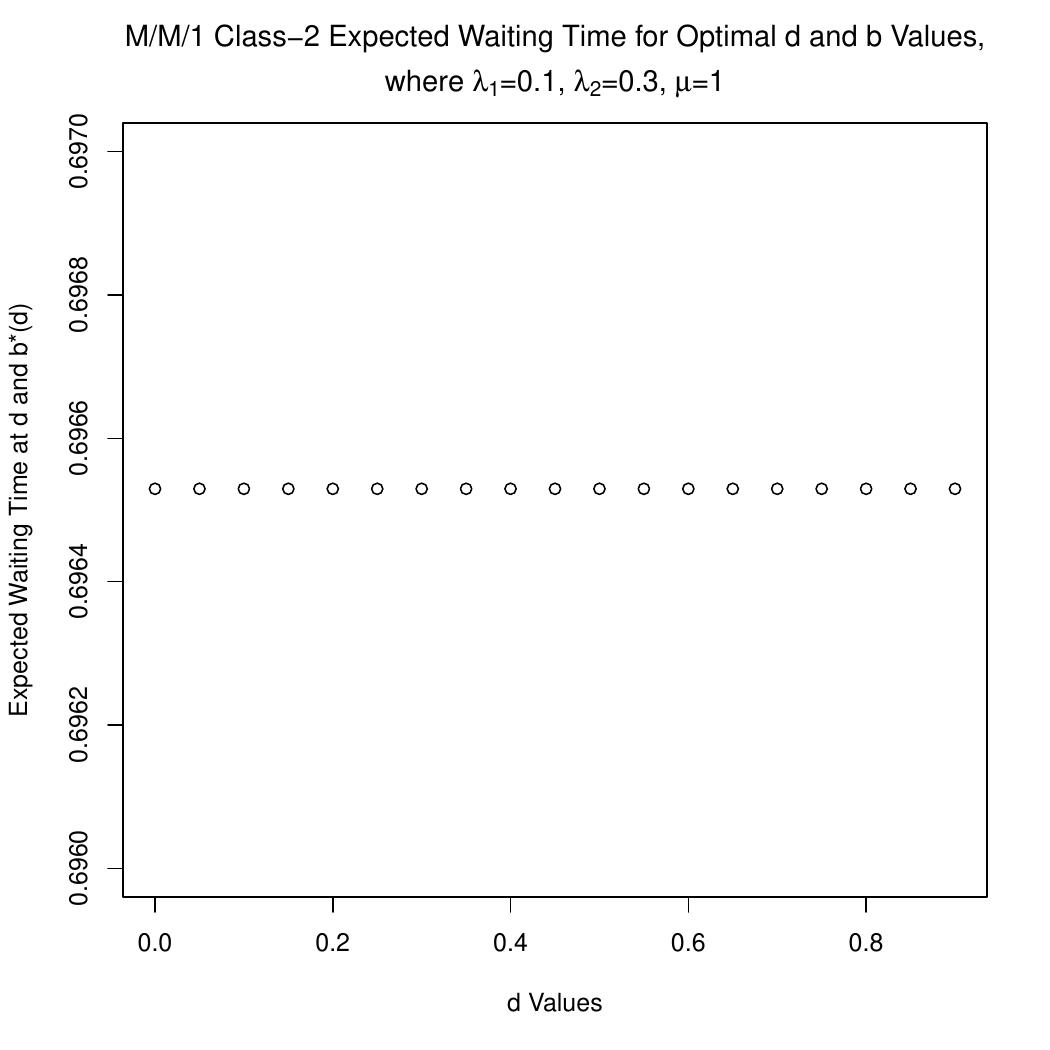}
  }
  \caption{$\EE \waitD_2$ for $(d,b^*(d))$ pairs associated with various occupancy levels for the KPI $\PP(\waitD_1 < 2) \geq 0.9$.}
  \label{fig:mm1_kpi_class1_examples}
\end{figure}}

\infor{
\begin{figure}[htb]
  \centering
  \subfloat{
    \includegraphics[width=45mm, height=45mm]{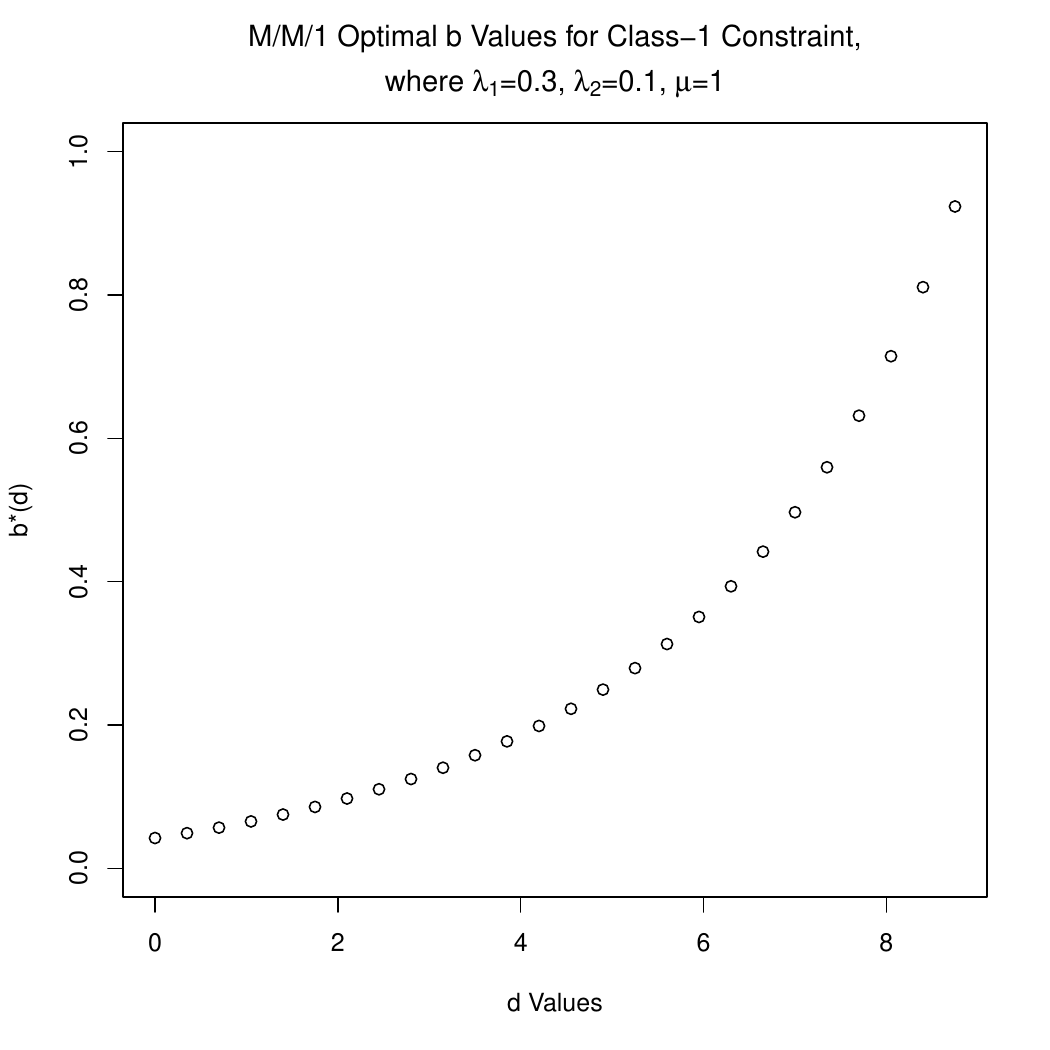}
  }
  \subfloat{
    \includegraphics[width=45mm, height=45mm]{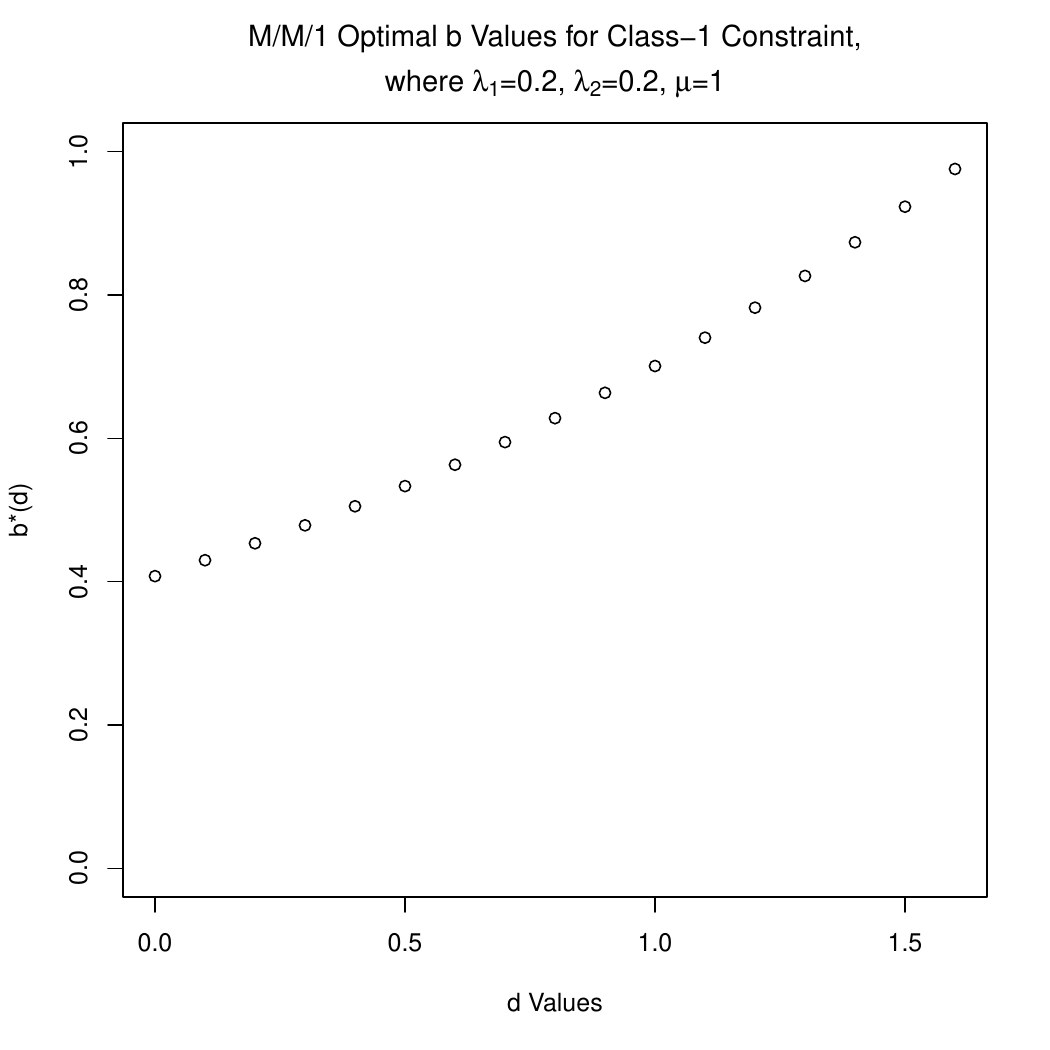}
  }
  \subfloat{
    \includegraphics[width=45mm, height=45mm]{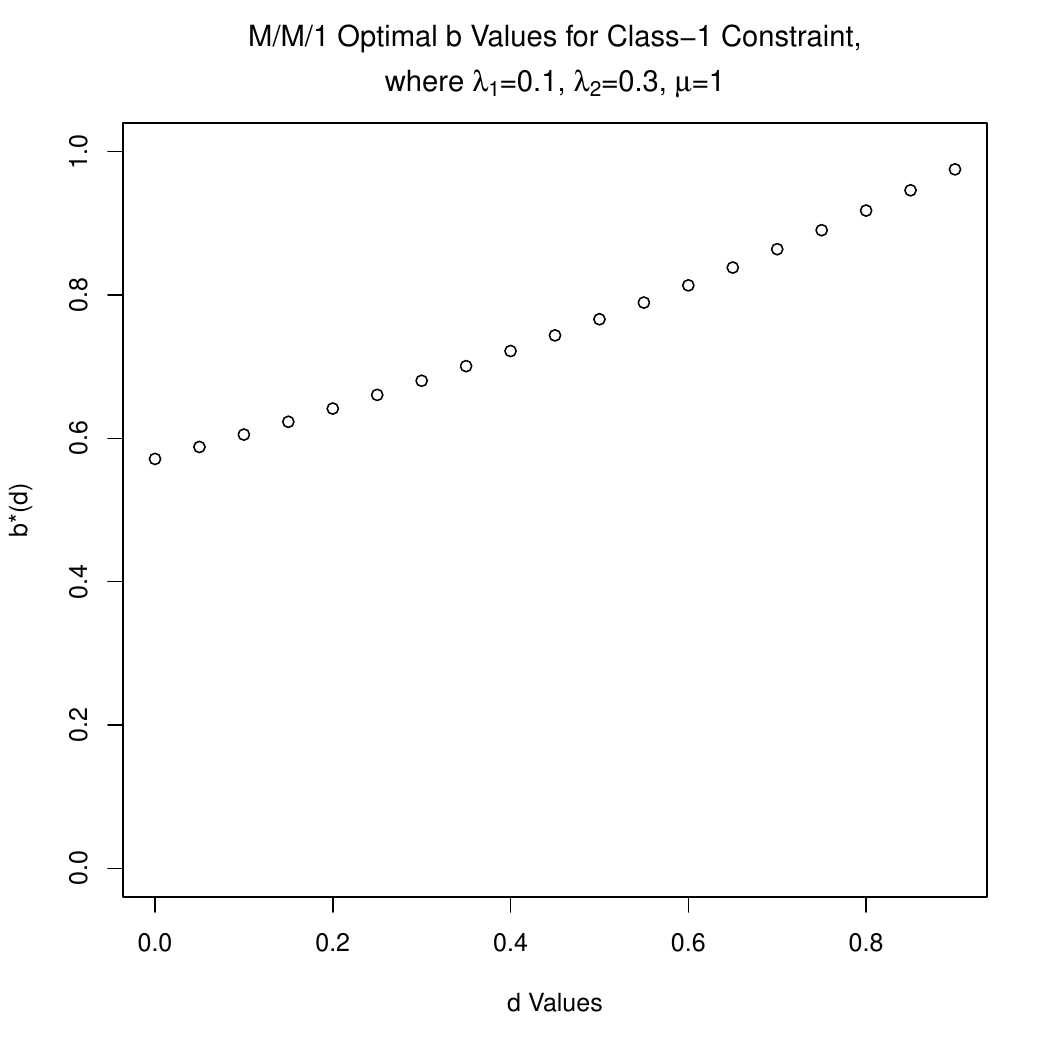}
  }
  \caption{Optimal $b^*(d)$ for delay level $d$ at various occupancy levels for the KPI $\PP(\waitD_1 < 2) \geq 0.9$.}
  \label{fig:mm1_kpi_class1_examples_bvals}
\end{figure}
\begin{figure}[htb]
  \centering
  \subfloat{
    \includegraphics[width=45mm, height=45mm]{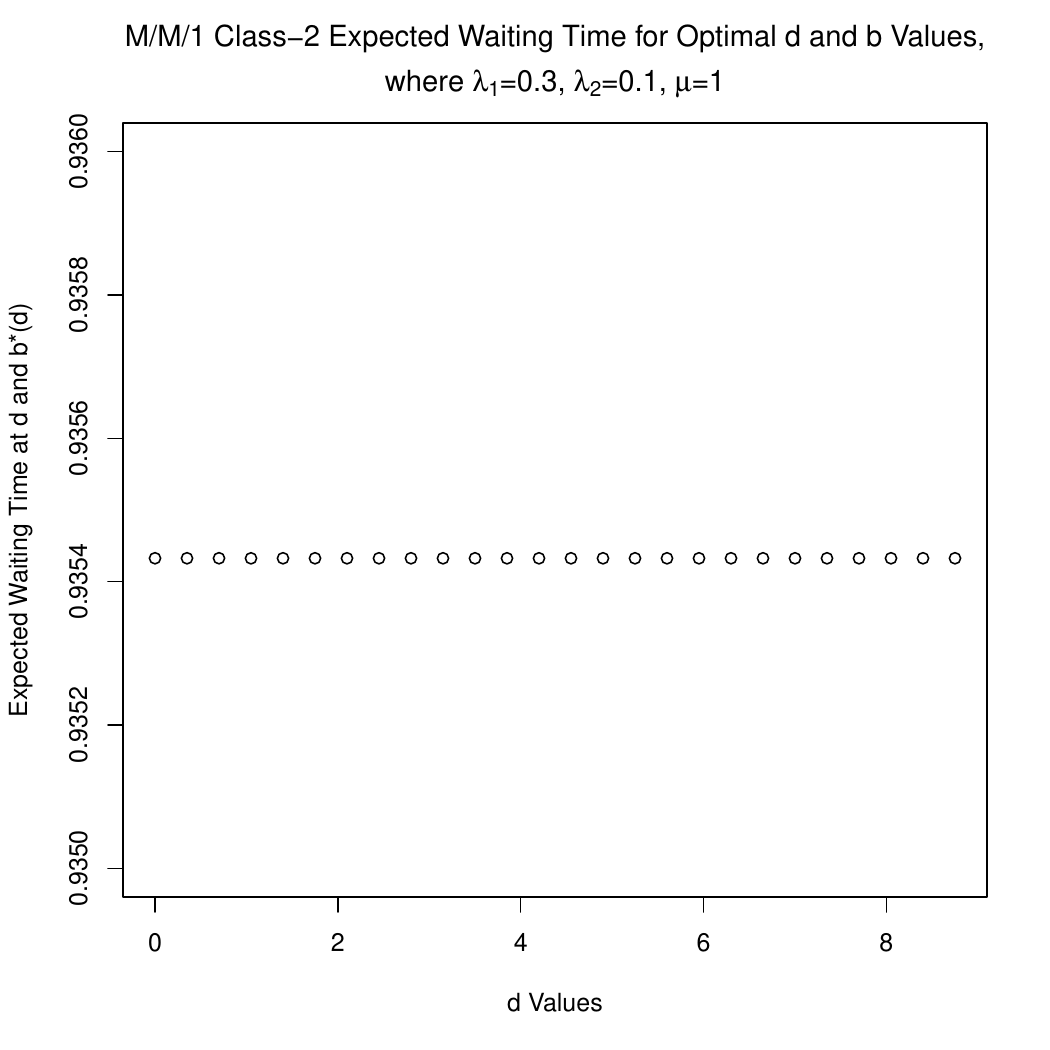}
  }
  \subfloat{
    \includegraphics[width=45mm, height=45mm]{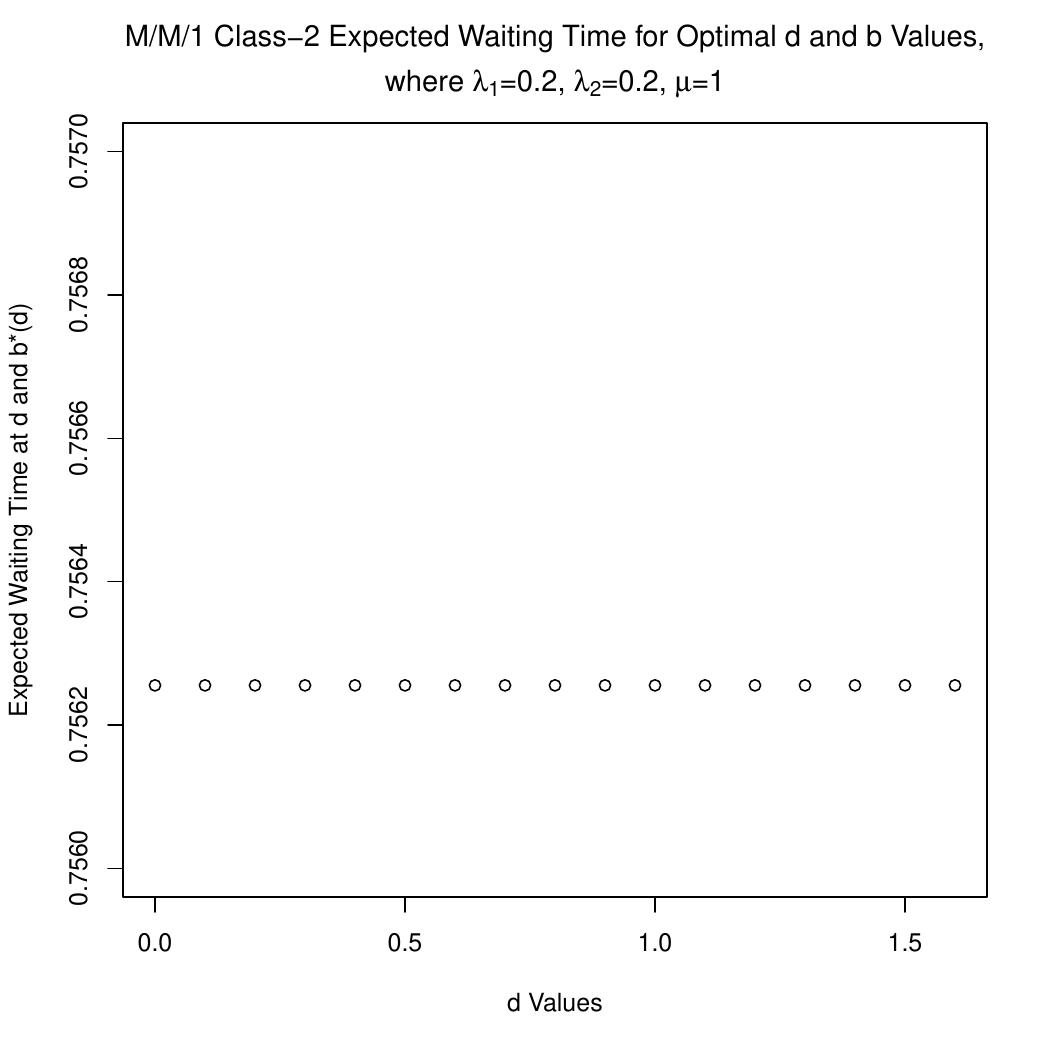}
  }
  \subfloat{
    \includegraphics[width=45mm, height=45mm]{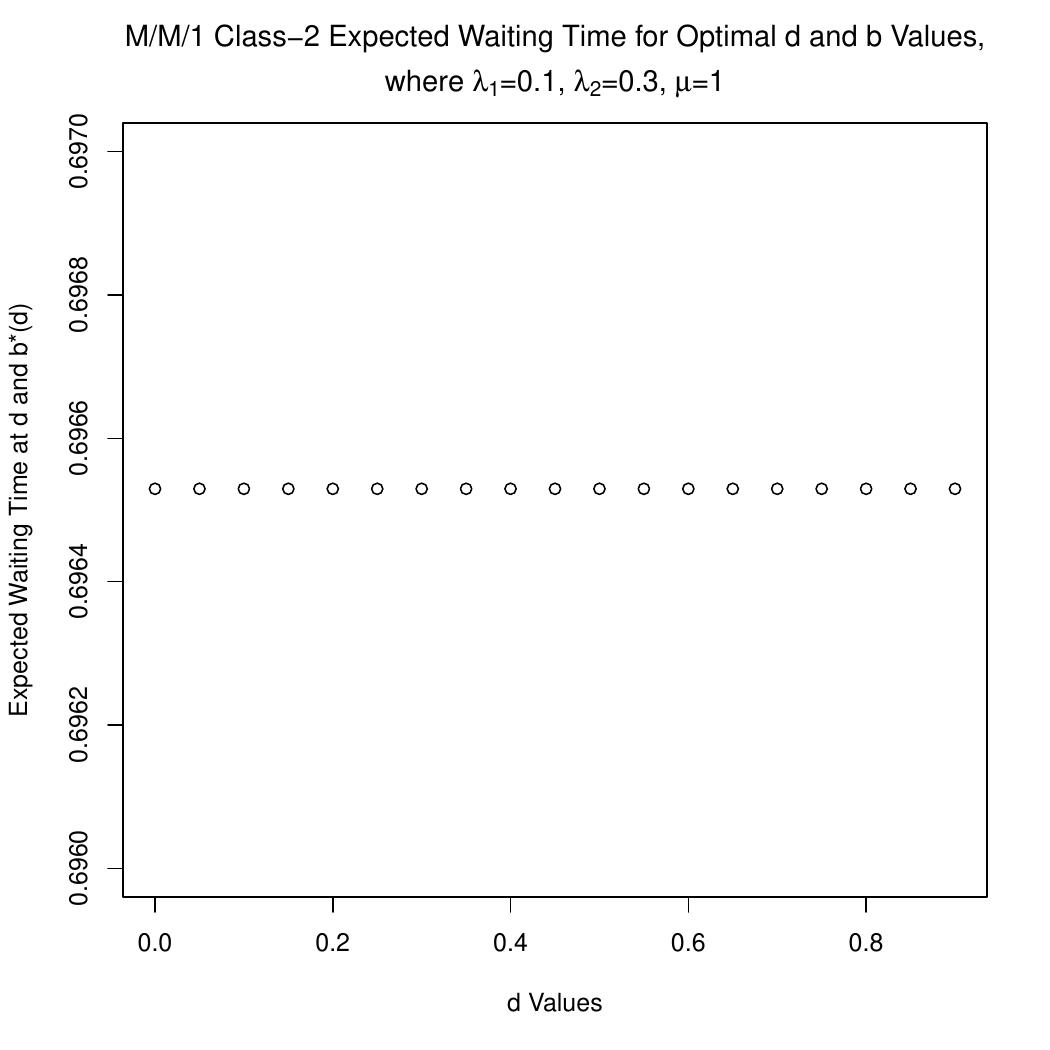}
  }
  \caption{$\EE \waitD_2$ for $(d,b^*(d))$ pairs associated with various occupancy levels for the KPI $\PP(\waitD_1 < 2) \geq 0.9$.}
  \label{fig:mm1_kpi_class1_examples}
\end{figure}}

\section{Conclusions} \label{sec:conclusions}

This paper builds on previous results for \classname{}-2 waiting times by providing an analytical expression for the \classname{}-1 expected waiting time in the M/G/1 \dapqname{}. We provide an algorithm that can be implemented for the M/M/1 and M/D/1 queueing disciplines, and numerically demonstrate the effect of changing the accumulation rate and the delay period. 
We then apply our computation algorithm for the expected waiting time to the health care setting; specifically, waiting times for patients in Canadian emergency departments. Previous analysis of KPI compliance for \classname{}-2 customers is extended by also optimizing over the expected waiting time for \classname{}-1 customers. We conclude that outside of the regions where the \npqname{} suffices or the FCFS fails, the optimal queueing discipline is always the \apqname{}, or equivalently a \dapqname{} with $d=0$.
Using a zero-inflated Exponential approximation, we also investigate optimizing KPI parameters for a \classname{}-1 constraint. We observe that this approximation seems reasonably accurate, and that the same trend holds where the optimal queueing discipline is always the \apqname{} rather than the \dapqname{}.

Major open problems include extending our analysis of expected waiting times to other service distributions of interest and identifying an exact analytical expression for the \classname{}-1 waiting time. Additionally, since this may be intractable beyond simple service distributions, it is of interest to extend our zero-inflated exponential approximations to other situations to facilitate the use of \dapqname{}s in real-world settings.

\section*{Acknowledgements}

This work is supported by an NSERC Canada Graduate Scholarship and an NSERC Discovery Grant. The authors are grateful to Dr.\ Doug Down for his insights regarding the geometric tail decay rate of the M/D/1 queue length, and to Elisheva Schwarz-Zur for pointing out that the statement of \cref{thm:mm1_avg} was incomplete in a preliminary draft.
\preprint{\bibliographystyle{abbrvnat}}
\infor{\bibliographystyle{tfcse}}
\bibliography{references.bib}

\newpage
\appendix

\preprint{
\section{Proofs from \cref{sec:computation}}
\label{app:proofs}
}
\infor{
\section{Proofs from \cref{sec:computation}}
\label{app:proofs}
}

\begin{proof}[Proof of \cref{thm:mm1_avg}]
By Corollary~3.1 of \citet{mojalal19dapq},
\*[
  \EE \waitD_2
    &= \EE \left[\waitD_2 \1\{\waitD_2 \leq d\} \right] +
       \EE \left[\waitD_2 \1\{\waitD_2 > d\} \right] \\
    &= \EE \left[\waitN_2 \1\{\waitN_2 \leq d\} \right] +
       \EE \left[\waitD_2 \1\{\waitD_2 > d\} \right].
\]

Then, using the definition of LST and Theorem~3.2 of \citet{mojalal19dapq},
\*[
    &\hspace{-1em}\EE \left[\waitD_2 \1\{\waitD_2 > d\} \right] \\
    &= - \frac{\dee}{\dee s} \lstD_2(s;d)  \Bigg\rvert_{s=0} \\
    &= \sum_{i=1}^\infty \sta_i \sum_{j=1}^\infty \PP \left[\num_d = j, \num_t > 0 \ \forall \ t \in [0,d) \big\lvert \num_0 = i \right] \left(d - j \frac{\dee}{\dee s} \accDs(s) \bigg\rvert_{s=0} \right),
\]
where the last line follows from $\accDs(0) = 1$.

Now, similarly,
\*[
  \EE \waitN_2
      &= \EE \left[\waitN_2 \1\{\waitN_2 \leq d\} \right] +
         \EE \left[\waitN_2 \1\{\waitN_2 > d\} \right],
\]
and by \cref{lem:npq},
\*[
  \EE \left[\waitN_2 \1\{\waitN_2 > d\} \right]
    &= \sum_{i=1}^\infty \sta_i \sum_{j=1}^\infty \PP \left[\num_d = j, \num_t > 0 \ \forall \ t \in [0,d) \big\lvert \num_0 = i \right] \left(d - j \frac{\dee}{\dee s} \accNs(s) \bigg\rvert_{s=0} \right).
\]

Consequently,
\*[
  &\EE \left[\waitN_2 - \waitD_2 \right] \\
  &= \sum_{i=1}^\infty \sta_i \sum_{j=1}^\infty \PP \left[\num_d = j, \num_t > 0 \ \forall \ t \in [0,d) \big\lvert \num_0 = i \right] j \left(\frac{\dee}{\dee s} \accDs(s) \bigg\rvert_{s=0} - \frac{\dee}{\dee s} \accNs(s) \bigg\rvert_{s=0} \right) \\
  &= \frac{\rho_1 b}{\mu(1 - \rhoA_1)(1 - \rho_1)} \sum_{i=1}^\infty \sta_i \sum_{j=1}^\infty j \PP \left[\num_d = j, \num_t > 0 \ \forall \ t \in [0,d) \big\lvert \num_0 = i \right],
\]
where the last line follows from differentiating the explicit forms of $\accDs$ and $\accNs$.

We use a continuous time Markov chain to handle this conditional probability, which is characterized by the transition matrix
\*[
  P &=
  \left[
  \begin{matrix}
    1 & 0 & 0 & 0 & \cdots \\
    q & 0 & p & 0 & \cdots \\
    0 & q & 0 & p & \cdots \\
    0 & 0 & q & 0 & \cdots \\
    \vdots & \vdots & \vdots & \vdots & \ddots
  \end{matrix}
  \right]
\]
such that for $q = \frac{\mu}{\mu + \lam_1}$, $p = \frac{\lam_1}{\mu + \lam_1}$, and $\nu = \mu + \lam_1$ we obtain
\*[
  \mathbb{P}\left[N_d=j, N_t>0; 0 \leq t \leq d \mid N_0=i \right]
  &= \sum_{k=0}^\infty P_{ij}^k \frac{e^{-\nu d} (\nu d)^k}{k!}.
\]

We are only interested in $i$ and $j$ such that the system is busy, so we define $P_+$ to be the $P$ matrix with its first row and first column removed. Observe that $P_+^k = (P^k)_+$. Now, define the following row and column vectors:
\*[
  \sta_+ = (1-\rho)\left[\rho, \rho^2, \rho^3, \dots \right],
  J_+ = \left[\begin{matrix} 1 \\ 2 \\ 3 \\ \vdots  \end{matrix} \right],
\]
where $\sta_+$ is defined in view of the stationary distribution of the M/M/1 being $\sta_i = (1-\rho)\rho^i$.

Then,
\*[
  \EE \left[\waitN_2 - \waitD_2 \right]
  &= \frac{\rho_1 b}{\mu(1 - \rhoA_1)(1 - \rho_1)} \sum_{k=0}^\infty \frac{e^{-\nu d} (\nu d)^k}{k!}
    \sta_+ P_+^k J_+.
\]

First, observe that
\*[
  \sta_+ P_+ &=
  (1-\rho)\left[
  q \rho^2,
  p \rho + q \rho^3,
  p \rho^2 + q \rho^4,
  p \rho^3 + q \rho^5,
  \dots
  \right].
\]

That is, for all but the first term,
\*[
  (\sta_+ P_+)_\ell
  &= (1-\rho)\rho^{\ell-1}r; \ \ell \geq 2,
\]
for $r = p + q \rho^2$. Repeatedly carrying out this the process and applying induction, we obtain
\*[
  (\sta_+ P_+^k)_\ell
  &= (1-\rho)\rho^{\ell-k} r^k; \ \ell \geq k+1.
\]

Thus, letting $x^{(k)}_\ell = (\sta_+ P_+^k)_\ell / (1-\rho)$ for $\ell, k \in \N$, we can write 
\[
  \sta_+ P_+^k J_+
  &= (1-\rho)\sum_{\ell=1}^k \ell x^{(k)}_\ell + (1-\rho)\sum_{\ell=k+1}^\infty \ell \rho^{\ell - k} r^k.
\label{eqn:mm1_piPJ_t1}
\]

Focusing on the second term, 
\[
  (1-\rho)\sum_{\ell=k+1}^\infty \ell \rho^{\ell - k} r^k
  &= (1-\rho) r^k \sum_{s=1}^\infty (s + k) \rho^s \\
  &= (1-\rho) r^k \left[\frac{\rho}{(1-\rho)^2} + \frac{k \rho}{1-\rho} \right] \\
  &= \rho r^k \left[\frac{1}{1-\rho} + k \right].
\label{eqn:mm1_piPJ_t2}
\]

Combining \cref{eqn:mm1_piPJ_t1,eqn:mm1_piPJ_t2} gives
\*[
  &\hspace{-1em}\sum_{k=0}^\infty \frac{e^{-\nu d} (\nu d)^k}{k!}
    \sta_+ P_+^k J_+  \\
  &= (1-\rho)\sum_{k=0}^\infty \left(\sum_{\ell=1}^k \ell x_\ell^{(k)} \frac{e^{-\nu d} (\nu d)^k}{k!} \right) +
    \sum_{k=0}^\infty \left(\rho r^k \left[\frac{1}{1-\rho} + k \right] \frac{e^{-\nu d} (\nu d)^k}{k!} \right) \\
  &=
    (1-\rho)\sum_{k=0}^\infty \left(\sum_{\ell=1}^k \ell x_\ell^{(k)} \frac{e^{-\nu d} (\nu d)^k}{k!} \right) +
    \frac{\rho e^{-\nu d + r \nu d}}{1-\rho} \sum_{k=0}^\infty \left(\frac{e^{-r \nu d} (r \nu d)^k}{k!} \right) \nonumber \\
    &\qquad +
    \rho e^{-\nu d + r \nu d} \sum_{k=0}^\infty k \left(\frac{e^{-r \nu d} (r \nu d)^k}{k!} \right) \\
  &=
    (1-\rho)\sum_{k=0}^\infty \frac{e^{-\nu d} (\nu d)^k}{k!} \left(\sum_{\ell=1}^k \ell x_\ell^{(k)}  \right) +
    \rho e^{-\nu d + r \nu d} \left(\frac{1}{1-\rho} + r \nu d \right),
\]
where the last line follows from the fact that $\frac{e^{-r\nu d} (r\nu d)^k}{k!}$ is a Poisson($r\nu d$) probability mass function. Then, it remains to observe that the $x^{(k)}_\ell$'s indeed satisfy the recursive formula and that from \citet{kleinrock76vol2}, 
\*[
  \EE \waitN_2 = \frac{\rho}{\mu (1-\rho_1)(1-\rho)}.
\]
Finally, apply the conservation law to get the \classname{}-1 expected waiting time.
\end{proof}

\begin{proof}[Proof of \cref{thm:md1_avg}]
Using the same logic as the proof of \cref{thm:mm1_avg} applied to Corollary~3.2 of \citet{mojalal19dapq},
\[
  &\EE \left[\waitN_2 - \waitD_2 \right] \\
  &= \sum_{i=1}^\infty \sta_i \sum_{j=1}^\infty \PP \left[\num_d = j, \num_t > 0 \ \forall \ t \in [0,d) \big\lvert \num_0 = i \right] \\ 
  &\qquad \times\left[\frac{\dee}{\dee s} \accDsresc{j}(s) \bigg\rvert_{s=0} - \frac{\dee}{\dee s} \accNsresc{j}(s) \bigg\rvert_{s=0} + (j-1) \left(\frac{\dee}{\dee s} \accDs(s) \bigg\rvert_{s=0} - \frac{\dee}{\dee s} \accNs(s) \bigg\rvert_{s=0} \right) \right].\label{eq:md1_avg_full}
\]

Recall that the LST of deterministic service is $\lst(s) = e^{-s/\mu}$. Thus, 
\*[
	\accDs(s) = \exp\left\{-[s+\lamA_1(1-\accDs(s))]/\mu \right\},
\]
so
\*[
  - \frac{\dee}{\dee s} \accDs(s) \bigg\rvert_{s=0} &= \exp\left\{-[\lamA_1(1-\accDs(0))]/\mu \right\} \frac{1}{\mu} \left[1 - \lamA_1 \frac{\dee}{\dee s} \accDs(s) \bigg\rvert_{s=0} \right], \\
  - \frac{\dee}{\dee s} \accDs(s) \bigg\rvert_{s=0} &= \frac{1}{\mu(1 - \rhoA_1)}.
\]

Similarly,
\*[
  - \frac{\dee}{\dee s} \accNs(s) \bigg\rvert_{s=0} &= \frac{1}{\mu(1 - \rho_1)}.
\]

Then, we have the following intermediary terms to assist in computing the residual accreditation interval. Recall that $d \in \N$. Let $\num_{{\resid{}}^-}$ be the number of customers in system immediately before the first service completion after arrival. For $k \in \N$ and $r > 0$, define
\*[
  \numresid{k}{r}
  = \PP \left(\num_{{\resid{}}^-} = k, \num_0 > 0 \lvert {\resid{}} = r \right)
  = \sum_{n=0}^{k-1} \sta_{k-n} e^{-\lam_1 r} \frac{(\lam_1 r)^n}{n!}.
\]
Also, for $i \in \N$ and $r > 0$, define
\*[
  \numarriv{i}{r}
  = \PP \left(i \text{ arrivals in } ({\resid{}}, d) \lvert {\resid{}} = r \right)
  = e^{-\lam_1(d-r)} \frac{(\lam_1(d-r))^{i}}{i!}.
\]

Finally, for $k \leq m \leq \ell$, $j \in \Z_+$, and $r > 0$, define
\*[
  \probempty{m}{k}{r}{j} 
  &= \PP \left(\num_{{\resid{}} + (m - 1)/\mu} = 0, \num_t > 0 \ \forall \ t \in [0,{\resid{}} + (m - 1)/\mu), \num_d = j \lvert \num_{\resid{}^-} = k, {\resid{}} = r \right) \\
  &= \left[e^{-\rho_1(m-1)} \frac{(\rho_1(m-1))^{m-k}}{(m-k)!}\left(\frac{k-1}{m-1} \right) e^{-\rho_1(\ell-(m-1))} e^{\lam_1 r}\right] \\
  &\qquad \times\left[\frac{\left(\rho_1(\ell-(m-1)) - \lam_1 r \right)^{j+\ell-m}}{(j+\ell-m)!}\right] \\
  &= e^{-\lam_1(d-r)} \frac{(\rho_1(m-1))^{m-k}}{(m-k)!}\left(\frac{k-1}{m-1} \right) \frac{\left(\rho_1(\ell-(m-1)) - \lam_1 r \right)^{j+\ell-m}}{(j+\ell-m)!}.
\]
Then, we define
\*[
  {\probempty{}{k}{r}{j}}
  =\PP \left(\exists m \in [d] \text{ s.t. } \num_{{\resid{}} + (m-1)/\mu} = 0, \num_d = j, \num_0 > 0 \lvert \num_{{\resid{}}^-} = k, {\resid{}} = r \right) 
  = \sum_{m=k}^d {\probempty{m}{k}{r}{j}}.   
\]

Recall that since the service length is always $1/\mu$, the unconditional residual service time is $\resid{} \sim \text{Unif}(0,1/\mu)$. 
For each $j \in \N$, let $\resid{j} = \resid{} \, \1\{\num_d = j, \num_t > 0 \ \forall \ t\in[0,d) \}$, with CDF $\cdfres{j}$ and LST $\lstres{j}$.
Consider the case where $d > 0$. Letting $\den(x)$ denote $\frac{\dee}{\dee x} \cdf(x)$ for any $\cdf$,
\*[
 \denres{j}(r) 
 &= \sum_{k=2}^\ell \numresid{k}{r} \left[\numarriv{j+\ell-k}{r} - {\probempty{}{k}{r}{j}} \right] + \sum_{k=\ell+1}^{j+\ell} \numresid{k}{r} \numarriv{j+\ell-k}{r} \\
 &= \sum_{k=2}^{j+\ell} \numresid{k}{r} \numarriv{j+\ell-k}{r} - \sum_{k=2}^\ell \numresid{k}{r} {\probempty{}{k}{r}{j}}.
\]

Next,
\*[
  \lstresc{j}(s) 
  = \int_0^{1/\mu} e^{-sr} \denresc{j}(r) \dee r
  = \int_0^{1/\mu} e^{-sr} \frac{\denres{j}(r)}{\PP(\num_d = j, \num_t > 0 \ \forall \ t \in [0,d))} \dee r,
\]
so
\*[
  - \frac{\dee}{\dee s} \lstresc{j}(s) \bigg\rvert_{s=0}
  = \frac{1}{\PP(\num_d = j, \num_t > 0 \ \forall \ t \in [0,d))} \int_0^{1/\mu} r \denres{j}(r) \dee r.
\]

Thus,
\*[
  -\frac{\dee}{\dee s} \accDsresc{j}(s) \bigg\rvert_{s=0}
  &= \left[-\frac{\dee}{\dee s} \lstresc{j}(s+\lamA_1(1-\accDs(s))) \bigg\rvert_{s=0} \right] \frac{1}{\mu} \left[1-\lamA_1 \frac{\dee}{\dee s} \accDs(s) \bigg\rvert_{s=0}\right] \\
  &= \left[\frac{1}{\PP(\num_d = j, \num_t > 0 \ \forall \ t \in [0,d))} \int_0^{1/\mu} r \denres{j}(r) \dee r\right] \frac{1}{\mu} \left[1 + \frac{\rhoA_1}{1 - \rhoA_1} \right].
\]

Similarly,
\*[
  -\frac{\dee}{\dee s} \accNsresc{j}(s) \bigg\rvert_{s=0}
  = \left[\frac{1}{\PP(\num_d = j, \num_t > 0 \ \forall \ t \in [0,d))} \int_0^{1/\mu} r \denres{j}(r) \dee r\right] \frac{1}{\mu} \left[1 + \frac{\rho_1}{1 - \rho_1} \right]
\]

That is,
\*[
  &\hspace{-2em}\frac{\dee}{\dee s} \accDsresc{j}(s) \bigg\rvert_{s=0} - \frac{\dee}{\dee s} \accNsresc{j}(s) \bigg\rvert_{s=0} \\
  &= \frac{1}{\mu} \left[\frac{\rho_1}{1 - \rho_1} - \frac{\rhoA_1}{1 - \rhoA_1} \right] \left[\frac{1}{\PP(\num_d = j, \num_t > 0 \ \forall \ t \in [0,d))} \int_0^{1/\mu} r \denres{j}(r) \dee r\right] \\
  &= \left[\frac{\rho_1 b}{\mu(1 - \rhoA_1)(1-\rho_1)}\right] \left[\frac{1}{\PP(\num_d = j, \num_t > 0 \ \forall \ t \in [0,d))} \int_0^{1/\mu} r \denres{j}(r) \dee r\right].
\]

Thus, plugging this into \cref{eq:md1_avg_full} gives
\*[
  &\hspace{-2em}\EE \left[\waitN_2 - \waitD_2 \right] \\
  &= \left[\frac{\rho_1 b}{\mu(1 - \rhoA_1)(1-\rho_1)}\right]  \sum_{i=1}^\infty \sta_i \sum_{j=1}^\infty \PP \left[\num_d = j, \num_t > 0 \ \forall \ t \in [0,d) \big\lvert \num_0 = i \right] \\
  &\qquad \times \left[\frac{1}{\PP(\num_d = j, \num_t > 0 \ \forall \ t \in [0,d))} \int_0^{1/\mu} r \denres{j}(r) \dee r + (j-1) \right] \\
  &= \left[\frac{\rho_1 b}{\mu(1 - \rhoA_1)(1-\rho_1)}\right]  \sum_{j=1}^\infty \PP \left[\num_d = j, \num_t > 0 \ \forall \ t \in [0,d) \right] \\
  &\qquad \times\left[\frac{1}{\PP(\num_d = j, \num_t > 0 \ \forall \ t \in [0,d))} \int_0^{1/\mu} r \denres{j}(r) \dee r + (j-1) \right] \\
  &= \left[\frac{\rho_1 b}{\mu(1 - \rhoA_1)(1-\rho_1)}\right]  \sum_{j=1}^\infty \int_0^{1/\mu} r \denres{j}(r) \dee r + \left[\frac{\rho_1 b}{\mu(1 - \rhoA_1)(1-\rho_1)}\right]  \sum_{j=1}^\infty (j-1) \int_0^{1/\mu} \denres{j}(r)\dee r.
\]

It remains to compute these integrals. To do so, observe the following lemma:
\begin{lemma}\label{lem:int_helper}
\*[
  \int_0^b x^n (c-x)^m dx = b^{n+1} \sum_{a=0}^m \frac{1}{n+a+1} {m \choose a} c^{m-a} (-b)^a.
\]
\end{lemma}
\begin{proof}[Proof of \cref{lem:int_helper}]
\*[
  \int_0^b x^n (c - x)^m dx
    &= \int_0^b x^n \sum_{a=0}^m {m \choose a} c^{m-a} (-x)^a dx \\
    &= \sum_{a=0}^m (-1)^a {m \choose a} c^{m-a} \int_0^b x^{n+a} dx \\
    &= \sum_{a=0}^m (-1)^a {m \choose a} c^{m-a} \frac{b^{n+a+1}}{n + a +1} \\
    &= b^{n+1} \sum_{a=0}^m \frac{1}{n+a+1} {m \choose a} c^{m-a} (-b)^a.
\]
\end{proof}

Then,
\*[
  &\hspace{-2em}\int_0^{1/\mu} \sum_{k=2}^{j+\ell} \numresid{k}{r} \numarriv{j+\ell-k}{r} \dee r \\
  &= \sum_{k=2}^{j+\ell} \frac{e^{-\lambda_1 d}}{(j+\ell-k)!} \sum_{n=0}^{k-1} \frac{\sta_{k-n}}{n!} \lambda_1^{j+\ell-k+n} \int_0^{1/\mu} r^n (d-r)^{j+\ell-k} \dee r \\
  &= \sum_{k=2}^{j+\ell} \frac{e^{-\lambda_1 d}}{(j+\ell-k)!} \sum_{n=0}^{k-1} \frac{\sta_{k-n}}{n!} \lambda_1^{j+\ell-k+n} (1/\mu)^{n+1} \sum_{a=0}^{j+\ell-k} \frac{1}{n+a+1} {{j+\ell-k} \choose a} d^{j+\ell-k-a}(-1)^a \\
  &= e^{-\lambda_1 d} \sum_{k=2}^{j+\ell} \sum_{a=0}^{j+\ell-k} 
    \frac{(-1)^a d^{j+\ell-k-a}}{(j+\ell-k-a)! a!}
    \sum_{n=0}^{k-1} \frac{\sta_{k-n} \lambda_1^{j+\ell+n-k}}{\mu^{n+1}(n+a+1) n!}.
\]

Next,
\*[
  &\hspace{-2em}\int_0^{1/\mu} \sum_{k=2}^\ell \numresid{k}{r} {\probempty{}{k}{r}{j}} \dee r \\
  &= \sum_{k=2}^\ell \sum_{n=0}^{k-1} \sta_{k-n} \\ 
  &\qquad \times \int_0^{1/\mu} e^{-\lambda_1 r} \frac{(\lam_1 r)^n}{n!}
      \sum_{m=k}^\ell e^{-\lam_1(d-r)} \frac{(\rho_1(m-1))^{m-k}}{(m-k)!}\left(\frac{k-1}{m-1} \right) \frac{\left(\rho_1(\ell-(m-1)) - \lam_1 r \right)^{j+\ell-m}}{(j+\ell-m)!} \dee r \\
  &= \sum_{k=2}^\ell e^{-\lam_1 d} \sum_{m=k}^\ell \frac{(m-1)^{m-k}}{(m-k)!(j+\ell-m)!} 
      \left(\frac{k-1}{m-1} \right)
      \sum_{n=0}^{k-1} \frac{\sta_{k-n} \lam_1^{j+\ell+n-k}}{n!}\\
   &\qquad \times
      \int_0^{1/\mu} r^n (d-(m-1)/\mu - r)^{j+\ell-m} \dee r \\
  &= e^{-\lam_1 d} \sum_{k=2}^\ell \sum_{m=k}^\ell 
    \frac{(m-1)^{m-k}}{(m-k)!} 
    \left(\frac{k-1}{m-1} \right)
    \sum_{n=0}^{k-1} \frac{\sta_{k-n} \lam_1^{j+\ell+n-k}}{n!} \\
   &\qquad \times
    \sum_{a=0}^{j+\ell-m} \frac{(-1)^a (\ell+1-m)^{j+\ell-m-a}}{\mu^{j+\ell-m-a}(n+a+1) (j+\ell-m-a)! a!} \\
  &= e^{-\lambda_1 d} \sum_{k=2}^\ell \sum_{m=k}^\ell
    \frac{(m-1)^{m-k}}{(m-k)!} 
    \left(\frac{k-1}{m-1} \right)
    \sum_{a=0}^{j+\ell-m} \frac{(-1)^a (\ell+1-m)^{j+\ell-m-a}}{\mu^{j+\ell-m-a}(j+\ell-m-a)! a!}
    \sum_{n=0}^{k-1} \frac{\sta_{k-n} \lambda_1^{j+\ell+n-k}}{(n+a+1) n!}.
\]

Thus,
\*[
  &\hspace{-2em}\int_0^{1/\mu} \denres{j}(r)\dee r \\
  &= e^{-\lambda_1 d} \sum_{k=2}^{j+\ell} \sum_{a=0}^{j+\ell-k} 
    \frac{(-1)^a d^{j+\ell-k-a}}{(j+\ell-k-a)! a!}
    \sum_{n=0}^{k-1} \frac{\sta_{k-n} \lambda_1^{j+\ell+n-k}}{\mu^{n+1}(n+a+1) n!} \\
  &\qquad - \Bigg[e^{-\lambda_1 d} \sum_{k=2}^\ell \sum_{m=k}^\ell
    \frac{(m-1)^{m-k}}{(m-k)!} 
    \left(\frac{k-1}{m-1} \right)\\
  &\qquad\qquad\times
    \sum_{a=0}^{j+\ell-m} \frac{(-1)^a (\ell+1-m)^{j+\ell-m-a}}{\mu^{j+\ell-m-a}(j+\ell-m-a)! a!}
    \sum_{n=0}^{k-1} \frac{\sta_{k-n} \lambda_1^{j+\ell+n-k}}{(n+a+1) n!}\Bigg].
\]

The exact same calculations with an additional $r$ inside the integral gives
\*[
  &\hspace{-2em}\int_0^{1/\mu} r \denres{j}(r)\dee r \\
  &= e^{-\lambda_1 d} \sum_{k=2}^{j+\ell} \sum_{a=0}^{j+\ell-k} 
    \frac{(-1)^a d^{j+\ell-k-a}}{(j+\ell-k-a)! a!}
    \sum_{n=0}^{k-1} \frac{\sta_{k-n} \lambda_1^{j+\ell+n-k}}{\mu^{n+1}(n+a+2) n!}  \\
  &\qquad - \Bigg[e^{-\lambda_1 d} \sum_{k=2}^\ell \sum_{m=k}^\ell
    \frac{(m-1)^{m-k}}{(m-k)!} 
    \left(\frac{k-1}{m-1} \right)\\
   &\qquad\qquad\times
    \sum_{a=0}^{j+\ell-m} \frac{(-1)^a (\ell+1-m)^{j+\ell-m-a}}{\mu^{j+\ell-m-a}(j+\ell-m-a)! a!}
    \sum_{n=0}^{k-1} \frac{\sta_{k-n} \lambda_1^{j+\ell+n-k}}{(n+a+2) n!}\Bigg].
\]

Plugging these last two results in along with the M/D/1 average waiting time from \citet{kleinrock76vol2} gives the first statement of the theorem.

Finally, consider the case where $d=0$. First, observe that
\*[
  &\EE \left[\waitN_2 - \waitD_2 \right] \\
  &= \sum_{i=1}^\infty \sta_i \sum_{j=1}^\infty \PP \left[\num_d = j, \num_t > 0 \ \forall \ t \in [0,d) \big\lvert \num_0 = i \right] \\ 
  &\qquad\times \left[\frac{\dee}{\dee s} \accDsresc{j}(s) \bigg\rvert_{s=0} - \frac{\dee}{\dee s} \accNsresc{j}(s) \bigg\rvert_{s=0} + (j-1) \left(\frac{\dee}{\dee s} \accDs(s) \bigg\rvert_{s=0} - \frac{\dee}{\dee s} \accNs(s) \bigg\rvert_{s=0} \right) \right] \\
  &= \sum_{j=1}^\infty \sta_j \left[\left(\frac{\lam_1 b}{(1 - \lamA_1)(1-\lam_1)}\right) \left(\frac{1}{\sta_j} \int_0^1 r \denres{j}(r) \dee r\right) + (j-1) \left(\frac{\lam_1 b}{(1 - \lamA_1)(1-\lam_1)}\right)\right] \\
  &= \left(\frac{\lam_1 b}{(1 - \lamA_1)(1-\lam_1)}\right) \left[\sum_{j=1}^\infty \int_0^1 r \denres{j}(r) \dee r + \sum_{j=1}^\infty (j-1) \sta_j \right].
\]

Next, we have the following result from \citet{adan09residual}:
\*[
  \int_0^1 r \denres{j}(r) \dee r = \frac{1-\rho}{\lam}\sum_{k=j+1}^\infty \sta_k.
\]

Thus, using the average queue length for an M/D/1 queue,
\*[
  \EE \left[\waitN_2 - \waitD_2 \right] 
  &= \left(\frac{\lam_1 b}{(1 - \lamA_1)(1-\lam_1)}\right) \left[\frac{1-\rho}{\lam} \sum_{j=1}^\infty \sum_{k=j+1}^\infty \pi_k + \frac{\rho^2}{2(1-\rho)} \right] \\
  &= \left(\frac{\lam_1 b}{(1 - \lamA_1)(1-\lam_1)}\right) \left[\frac{1-\rho}{\lam} \sum_{j=1}^\infty \PP(N_0 > j) + \frac{\rho^2}{2(1-\rho)} \right] \\
  &= \left(\frac{\lam_1 b}{(1 - \lamA_1)(1-\lam_1)}\right) \left[\frac{1-\rho}{\lam} (\EE N_0 - \rho) + \frac{\rho^2}{2(1-\rho)} \right] \\
  &= \frac{\lam_1 b}{(1 - \lamA_1)(1-\lam_1)} \frac{\rho}{2(1-\rho)}.
\] 

\end{proof}

\end{document}